\documentclass[10pt,reqno,oneside]{amsproc}
\title[Persistence of conormal analyticity]{Persistence of conormal analyticity for the Euler equations in the half-space}

\author[B.~Ingimarson]{Benjamin Ingimarson}
\address{Department of Mathematics, University of Southern California, Los Angeles, CA 90089}
\email{ingimars@usc.edu}

\author[I.~Kukavica]{Igor Kukavica}
\address{Department of Mathematics, University of Southern California, Los Angeles, CA 90089}
\email{kukavica@usc.edu}

\author[M.~Sammartino]{Marco Sammartino}
\address{Department of Mathematics, Via Archirafi, 34, Palermo Italy}
\email{marco.sammartino@unipa.it}

  \chardef\forshowkeys=0
  \chardef\refcheck=0
  \chardef\showllabel=0
  \chardef\sketches=0
  \chardef\figures=1
\usepackage{enumitem,mathtools}
\usepackage{datetime}
\usepackage{fancyhdr}
\usepackage{comment}
\allowdisplaybreaks
\ifnum\forshowkeys=1

  \usepackage[notref,notcite,color]{showkeys}
\fi
\usepackage[margin=1in]{geometry}
\usepackage{amsmath, amsthm, amssymb}
\usepackage{times}
\usepackage{graphicx}
\usepackage[usenames,dvipsnames,svgnames,table]{xcolor}
\usepackage{marginnote}
\usepackage[unicode,breaklinks=true,colorlinks=true,linkcolor=blue,urlcolor=blue,citecolor=blue]{hyperref}
\usepackage[most]{tcolorbox}
\usepackage{tikz}

\ifnum\refcheck=1
  \usepackage{refcheck}
\fi

\usepackage{import}
\usepackage{xifthen}
\usepackage{pdfpages}
\usepackage{transparent}

\begin{document}
\def\YY{X}
\def\OO{\mathcal O}
\def\SS{\mathbb S}
\def\CC{\mathbb C}
\def\RR{\mathbb R}
\def\TT{\mathbb T}
\def\ZZ{\mathbb Z}
\def\HH{\mathbb H}
\def\RSZ{\mathcal R}
\def\LL{\mathcal L}
\def\SL{\LL^1}
\def\ZL{\LL^\infty}
\def\GG{\mathcal G}
\def\tt{\langle t\rangle}
\def\erf{\mathrm{Erf}}
\def\mgt#1{\textcolor{magenta}{#1}}
\def\ff{\rho}
\def\gg{G}
\def\sqrtnu{\sqrt{\nu}}
\def\ww{w}
\def\ft#1{#1_\xi}
\def\les{\lesssim}
\def\lec{\lesssim}
\def\ges{\gtrsim}
\renewcommand*{\Re}{\ensuremath{\mathrm{{\mathbb R}e\,}}}
\renewcommand*{\Im}{\ensuremath{\mathrm{{\mathbb I}m\,}}}
\ifnum\showllabel=1
\newcommand{\llabel}[1]{%
  \marginnote{\color{lightgray}\normalfont\small (#1)}%
  \notag
}
\else
 \def\llabel#1{\notag}
\fi
\newcommand{\norm}[1]{\left\|#1\right\|}
\newcommand{\nnorm}[1]{\lVert #1\rVert}
\newcommand{\abs}[1]{\left|#1\right|}
\newcommand{\NORM}[1]{|\!|\!| #1|\!|\!|}
\newtheorem{theorem}{Theorem}[section]
\newtheorem{Theorem}{Theorem}[section]
\newtheorem{corollary}[theorem]{Corollary}
\newtheorem{Corollary}[theorem]{Corollary}
\newtheorem{proposition}[theorem]{Proposition}
\newtheorem{Proposition}[theorem]{Proposition}
\newtheorem{Lemma}[theorem]{Lemma}
\newtheorem{lemma}[theorem]{Lemma}
\theoremstyle{definition}
\newtheorem{definition}{Definition}[section]
\newtheorem{remark}[Theorem]{Remark}
\def\theequation{\thesection.\arabic{equation}}
\numberwithin{equation}{section}
\definecolor{mygray}{rgb}{.6,.6,.6}
\definecolor{myblue}{rgb}{9, 0, 1}
\definecolor{colorforkeys}{rgb}{1.0,0.0,0.0}
\newlength\mytemplen
\newsavebox\mytempbox
\makeatletter
\newcommand\mybluebox{%
    \@ifnextchar[
       {\@mybluebox}%
       {\@mybluebox[0pt]}}
\def\@mybluebox[#1]{%
    \@ifnextchar[
       {\@@mybluebox[#1]}%
       {\@@mybluebox[#1][0pt]}}
\def\@@mybluebox[#1][#2]#3{
    \sbox\mytempbox{#3}%
    \mytemplen\ht\mytempbox
    \advance\mytemplen #1\relax
    \ht\mytempbox\mytemplen
    \mytemplen\dp\mytempbox
    \advance\mytemplen #2\relax
    \dp\mytempbox\mytemplen
    \colorbox{myblue}{\hspace{1em}\usebox{\mytempbox}\hspace{1em}}}
\makeatother
\def\bnew{\colr }
\def\enew{\colb }
\def\bold{\colu }
\def\eold{\colb }
\def\rr{r}
\def\weaks{\text{\,\,\,\,\,\,weakly-* in }}
\def\inn{\text{\,\,\,\,\,\,in }}
\def\cof{\mathop{\rm cof\,}\nolimits}
\def\Dn{\frac{\partial}{\partial N}}
\def\Dnn#1{\frac{\partial #1}{\partial N}}
\def\tdb{\tilde{b}}
\def\tda{b}
\def\qqq{u}
\def\lat{\Delta_2}
\def\biglinem{\vskip0.5truecm\par==========================\par\vskip0.5truecm}
\def\inon#1{\hbox{\ \ \ \ \ \ \ }\hbox{#1}}                
\def\onon#1{\inon{on~$#1$}}
\def\inin#1{\inon{in~$#1$}}
\def\FF{F}
\def\andand{\text{\indeq and\indeq}}
\def\ww{w(y)}
\def\ll{{\color{red}\ell}}
\def\ee{\mathrm{e}}
\def\startnewsection#1#2{ \section{#1}\label{#2}\setcounter{equation}{0}}   
\def\nnewpage{ }
\def\sgn{\mathop{\rm sgn\,}\nolimits}    
\def\Tr{\mathop{\rm Tr}\nolimits}    
\def\div{\mathop{\rm div}\nolimits}
\def\curl{\mathop{\rm curl}\nolimits}
\def\dist{\mathop{\rm dist}\nolimits}
\def\id{\mathop{\rm id}\nolimits}
\def\supp{\mathop{\rm supp}\nolimits}
\def\indeq{\quad{}}           
\def\period{.}                       
\def\semicolon{\,;}                  
\def\nts#1{{\cor #1\cob}}
\def\cmi#1{\text{~{{\coli IK: \underline{#1}}}~}}
\def\coli{\color{colorigor}}
\definecolor{colorigor}{rgb}{1, 0.2, 0.8}
\def\colr{\color{red}}
\def\colrr{\color{black}}
\def\colb{\color{black}}
\def\coly{\color{lightgray}}
\definecolor{colorgggg}{rgb}{0.1,0.5,0.3}
\definecolor{colorllll}{rgb}{0.0,0.7,0.0}
\definecolor{colorhhhh}{rgb}{0.3,0.75,0.4}
\definecolor{colorpppp}{rgb}{0.7,0.0,0.2}
\definecolor{coloroooo}{rgb}{0.45,0.0,0.0}
\definecolor{colorqqqq}{rgb}{0.1,0.7,0}
\def\colg{\color{colorgggg}}
\def\collg{\color{colorllll}}
\def\cole{\color{coloroooo}}
\def\coleo{\color{colorpppp}}
\def\cole{\color{black}}
\def\colu{\color{blue}}
\def\colc{\color{colorhhhh}}
\def\colW{\colb}   
\definecolor{coloraaaa}{rgb}{0.6,0.6,0.6}
\def\colw{\color{coloraaaa}}
\def\comma{ {\rm ,\qquad{}} }            
\def\commaone{ {\rm ,\quad{}} }          
\def\les{\lesssim}
\def\nts#1{{\color{blue}\hbox{\bf ~#1~}}} 
\def\ntsf#1{\footnote{\color{colorgggg}\hbox{#1}}} 
\def\blackdot{{\color{red}{\hskip-.0truecm\rule[-1mm]{4mm}{4mm}\hskip.2truecm}}\hskip-.3truecm}
\def\bluedot{{\color{blue}{\hskip-.0truecm\rule[-1mm]{4mm}{4mm}\hskip.2truecm}}\hskip-.3truecm}
\def\purpledot{{\color{colorpppp}{\hskip-.0truecm\rule[-1mm]{4mm}{4mm}\hskip.2truecm}}\hskip-.3truecm}
\def\greendot{{\color{colorgggg}{\hskip-.0truecm\rule[-1mm]{4mm}{4mm}\hskip.2truecm}}\hskip-.3truecm}
\def\cyandot{{\color{cyan}{\hskip-.0truecm\rule[-1mm]{4mm}{4mm}\hskip.2truecm}}\hskip-.3truecm}
\def\reddot{{\color{red}{\hskip-.0truecm\rule[-1mm]{4mm}{4mm}\hskip.2truecm}}\hskip-.3truecm}
\def\tdot{{\color{green}{\hskip-.0truecm\rule[-.5mm]{3mm}{3mm}\hskip.2truecm}}\hskip-.1truecm}
\def\gdot{\greendot}
\def\bdot{\bluedot}
\def\ydot{\cyandot}
\def\rdot{\cyandot}
\def\fractext#1#2{{#1}/{#2}}
\def\ii{\hat\imath}
\def\fei#1{\textcolor{blue}{#1}}
\def\vlad#1{\textcolor{cyan}{#1}}
\def\igor#1{\text{{\textcolor{colorqqqq}{#1}}}}
\def\igorf#1{\footnote{\text{{\textcolor{colorqqqq}{#1}}}}}
\def\Omf{\Omega_{\text f}}
\def\Ome{\Omega_{\text e}}
\def\Omb{\Omega_{\text b}}
\def\Gaf{\Gamma_{\text f}}
\def\Gae{\Gamma_{\text e}}
\def\Gab{\Gamma_{\text b}}
\def\Gac{\Gamma_{\text c}}
\def\Nf{N^{\text f}}
\def\Ne{N^{\text e}}

\newcommand{\p}{\partial}
\renewcommand{\d}{\mathrm{d}}
\newcommand{\UE}{U^{\rm E}}
\newcommand{\PE}{P^{\rm E}}
\newcommand{\KP}{K_{\rm P}}
\newcommand{\uNS}{u^{\rm NS}}
\newcommand{\vNS}{v^{\rm NS}}
\newcommand{\pNS}{p^{\rm NS}}
\newcommand{\omegaNS}{\omega^{\rm NS}}
\newcommand{\uE}{u^{\rm E}}
\newcommand{\vE}{v^{\rm E}}
\newcommand{\pE}{p^{\rm E}}
\newcommand{\omegaE}{\omega^{\rm E}}
\newcommand{\ua}{u_{\rm   a}}
\newcommand{\va}{v_{\rm   a}}
\newcommand{\omegaa}{\omega_{\rm   a}}
\newcommand{\ue}{u_{\rm   e}}
\newcommand{\ve}{v_{\rm   e}}
\newcommand{\omegae}{\omega_{\rm e}}
\newcommand{\omegaeic}{\omega_{{\rm e}0}}
\newcommand{\ueic}{u_{{\rm   e}0}}
\newcommand{\veic}{v_{{\rm   e}0}}
\newcommand{\up}{u^{\rm P}}
\newcommand{\vp}{v^{\rm P}}
\newcommand{\tup}{{\tilde u}^{\rm P}}
\newcommand{\bvp}{{\bar v}^{\rm P}}
\newcommand{\omegap}{\omega^{\rm P}}
\newcommand{\tomegap}{\tilde \omega^{\rm P}}
\newcommand{\eps}{\varepsilon}  
\newcommand{\Hco}{\mathcal{H}_{\rm co}} 
\newcommand{\co}{{\rm co}}
\newcommand{\nor}{{\rm nor}}
\newcommand{\calC}{\mathcal{C}}
\newcommand{\calP}{\mathcal{P}}
\newcommand{\Eco}{\mathcal{E}_{\rm co}}

\newcommand{\eqnb}{\begin{equation}}
\newcommand{\eqne}{\end{equation}}
  
\renewcommand{\up}{u^{\rm P}}
\renewcommand{\vp}{v^{\rm P}}
\renewcommand{\omegap}{\Omega^{\rm P}}
\renewcommand{\tomegap}{\omega^{\rm P}}



\begin{abstract}
We prove the persistence of conormal analyticity through the $H^r$ norm, with $r \geq 4$, for the 3D incompressible Euler equations in the half-space $\Omega = \mathbb{R}_+^3$.
We also establish a lower bound on the radius of conormal analyticity which depends on $\|u\|_{H^r}$ and initial data.
In particular, persistence is established in both of the spaces $X_\tau^\co$ and $A_\tau^\co$, the former being based on the conormal operator $Z_k = \phi^k \partial_3^k$ and the latter being based on $L^k = (\phi \partial_3)^k$.
Finally, we show that the norms $\|\cdot \|_{X_\tau^\co}$ and $\|\cdot \|_{A_\tau^\co}$ are equivalent up to a small loss in radius, the proof of which uses explicit identities relating the two conormal operators.
\end{abstract}

\maketitle
\setcounter{tocdepth}{2} 
\section{Introduction}\label{sec01}
Consider the 3D incompressible Euler equations, 
  \begin{align}
   \begin{split}
    \partial_t u
    + u \cdot \nabla u 
    + \nabla p 
    &= 0,
    \\
    \div u &= 0
   \end{split}
   \label{EQ00}
  \end{align}
in the half-space $\Omega = \mathbb{R}_+^3$ subject to $u\cdot n = 0$, where $n = (0,0,-1)$.
In this paper, we show that a solution $u(t)$ of~\eqref{EQ00} with initial data in $X_{\tau_0}^{\co}\cap H^r$ continues in the space $X_{\tau(t)}^\co$ for as long as it continues in~$H^r$.
We also give lower bounds on the radius of conormal analyticity~$\tau(t)$.
Moreover, the same theorem is established for an alternative conormal analytic space $A_\tau^\co$, and we show that the norms for each space are equivalent up to a small loss in radius.  

Well-posedness for the 3D incompressible Euler equations is well known for $H^r$ with $r > 5/2$, see~\cite{EM,Ka,T}.
In the case of a bounded domain, the study of real analytic persistence began with Bardos and Benachour~\cite{BB}, who used a complex extension and an implicit argument.
Later works, such as~\cite{Ba,Be,LO}, found explicit estimates on the radius of analyticity in various settings. 
In the half-space, these results were improved by the second author and Vicol~\cite{KV2}, who established decay rates depending algebraically on $\exp \int_0^t \|\nabla u\|_{L^\infty}$; see also~\cite{KV1} for the periodic and whole-space settings.
The case when $\Omega$ is a general bounded domain with analytic or Gevrey-class regularity was handled similarly in~\cite{KV3}.
For other results on analyticity for the Euler equations, Navier-Stokes equations, nonlinear heat equation, primitive equations, divergence equation, and more,
see~\cite{CKV,FT,G1,G2,GK1,GK2,Ku1,Ku2,KLS,KM,KOX,KX,SC1}.


That the boundary acts as a catalyst for the generation of small scales, and
possibly of singularities, has recently been put in evidence in several settings.
In two dimensions, where solutions are global, Kiselev and \v{S}ver\'ak~\cite{KS}
constructed smooth solutions in the disk whose vorticity gradient grows at a
double-exponential rate, saturating the upper bound provided by the Yudovich
theory; the growth takes place at a hyperbolic stagnation point located on the
boundary, and away from boundaries only strictly slower growth is known. In three
dimensions, the numerical scenario of Luo and Hou~\cite{LH} for axisymmetric flows
develops its singularity on a ring at the boundary of the cylinder, and its
rigorous counterpart by Chen and Hou~\cite{CH} establishes finite-time blow-up
from smooth data precisely in this geometry; indeed, to the best of our knowledge,
the only blow-up proof currently available for smooth, unforced data is set in a
domain with boundary, the boundary-less construction of Elgindi~\cite{E} requiring
$C^{1,\alpha}$ velocities. In this light, Theorems~\ref{T01} and~\ref{T02}
quantify how fast the boundary can generate scales in the inviscid flow: by the
lower bound~\eqref{EQ15}, the smallest active scale shrinks at most exponentially
in the time-accumulated $H^r$ norm of the solution.

An important technique in the local-existence theory for analytic spaces is the creation of the analytic dissipative term.
This term appears from allowing the radius of analyticity $\tau(t)$ to depend on time and decrease sufficiently rapidly as time increases.  
Higher-order terms that are otherwise dangerous for non-analytic energy estimates may then be absorbed into the analytic dissipation.
This mechanism also has broader utility:
it shows that analyticity is an appropriate functional setting for inviscid problems which exhibit one derivative loss in their energy estimates. 
For a clear demonstration of this fact in application to the inflow-outflow problem, see~\cite{KOS}.
Analytic persistence, however, is when existence in analytic spaces can be continued for as long as the solution stays in a certain subcritical norm. 
In general, persistence depends on conservation properties of the nonlinear term. 
This can be observed in the proofs of Theorems~\ref{T01}--\ref{T02} below, where the top-order term coming from the advection is canceled do to incompressibility and impermeability, and a careful analysis ensues to ensure an additional factor of the analyticity radius $\tau$ is present in all quadratic terms $\|u\|_{X_\tau^\co} \|u\|_{Y_\tau^\co}$. 
In this way, local analytic persistence is substantially more difficult to obtain than local analytic existence.
We also note that persistence has unexpected fundamental applications, including a new method of constructing local-in-time $H^r$ solutions to the Euler equations, see~\cite{KO}. 


Conormal spaces are often used when studying the inviscid limit of the
Navier-Stokes equations.
Indeed, for the half-space and in a small vicinity of the boundary $\{x_3 = 0\}$,
a solution $u^\nu$ is expected to behave like $u^\nu \sim U(t,x_1,x_2,
x_3/\sqrt{\nu})$, where $\nu$ is the viscosity parameter and $U$ is some profile.
As $\nu \to 0$, the normal derivative of $u^\nu$ becomes singular.
Hence, a natural space for this problem replaces the normal derivative
$\partial_3$ with $Z = \phi(x_3)\partial_3$, where $\phi(x_3) \sim x_3$ near the
boundary: Flat normal derivatives of the profile are of size $\nu^{-1/2}$ near
the wall, while the conormal derivatives $Z^k u^\nu$ remain bounded uniformly in
$\nu$, since $x_3 \partial_3$ is invariant under the rescaling $x_3 \mapsto
x_3/\sqrt{\nu}$.
For a treatment of Sobolev conormal spaces for the Euler equations,
see~\cite{AK,MR}.
At the same time, in the absence of special structure, the available
justifications of the inviscid limit require analyticity
in the tangential variables near the boundary in order to control the
instabilities of the boundary layer~\cite{SC1,SC2,FTZ,NN,GVMM,DGV,KVW}.
The spaces $X_\tau^\co$ and $A_\tau^\co$ combine precisely these two
features---a tangential radius uniform up to the wall, and normal regularity
measured through $Z$---and Theorems~\ref{T01} and~\ref{T02} show that the Euler
flow propagates this structure, with the explicit radius~\eqref{EQ15}.
They may thus be regarded as the inviscid anchor of a uniform-in-$\nu$ theory in
conormal analytic classes, an analytic counterpart of the conormal Sobolev
framework of~\cite{MR}, in which the limit $u^\nu \to u$ would be pursued
without ever requiring normal analyticity---a property that no-slip
approximations do not enjoy uniformly in~$\nu$.
A final point concerning the zero-viscosity limit is that the analysis of the
Navier--Stokes equations requires handling the heat semigroup of the
half-space in the complex-analytic setting. Since $\mathrm{Re}\,(y-y')^{2}>0$
precisely when $|\mathrm{Im}\,(y-y')|<|\mathrm{Re}\,(y-y')|$, the Gaussian
kernel decays along shifted contours only in domains with a \emph{causal
	structure}, in which the real part dominates the imaginary
part~\cite{SC2,ACS}; near the boundary these are exactly the conoidal domains
of holomorphy whose real-variable counterpart is given by the conormal spaces
adopted in this paper.

The conormal analytic classes have a transparent geometric meaning. The finiteness
of~\eqref{EQ13} (or of~\eqref{EQ202}) yields the pointwise bounds
$|\partial^{\alpha'}\partial_3^{k} u(x)| \les |\alpha'|!\,k!\,
(c\tau)^{-|\alpha'|}(c\tau\phi(x_3))^{-k}$: the tangential radius of analyticity is
uniform up to the wall, while the normal radius is comparable to $\tau\phi(x_3)$ and
vanishes linearly at the boundary, so that the natural domain of holomorphy is a
conoid rather than a strip, as in~\cite{SC1,SC2}. The class is accordingly strictly larger than the
uniform one of~\cite{KV2,KV3}: it contains velocities which are tangentially
analytic but only finitely regular in the normal direction, for instance behaving
like $x_3^{7/2}$ near the wall.

In this paper, we consider the incompressible Euler equations in the half-space using the weighted conormal analytic norms $X_\tau^\co$ and~$A_\tau^\co$.
Compared to previous work on analyticity for the Euler equations in the presence of a boundary, 
the use of a conormal norm introduces several additional complications which we describe below.

The paper is summarized as follows. 
In Section~\ref{sec02}, we start by defining the analytic norms $X_\tau^\co$ and $Y_\tau^\co$, which are based on conormal Sobolev spaces $\Hco^{m,3}$ and $\Eco^{m,3}$, see~\eqref{EQ09}--\eqref{EQ13} and~\eqref{EQ200}--\eqref{EQ202} for definitions.
We then state Theorems~\ref{T01},~\ref{T02}, and~\ref{T03} which are proved in Sections~\ref{sec03},~\ref{sec04}, and~\ref{sec05}, respectively.
Note that in previous literature, such as~\cite{KV1,KV2}, the analytic norms are based on sums of $L^2$ norms $\|\partial^\alpha u\|_{L^2}$, 
whereas, in our approach,
we base them on the $H^3$ norm~$\|J^\alpha u\|_{H^3}$.
Indeed, the $H^3$ norm is essential due to the lack of embeddings for Sobolev conormal spaces.  

In Section~\ref{sec03} we prove Theorem~\ref{T01}, which relies on delicate a~priori estimates of infinite sums coming from the advection and pressure terms in~\eqref{EQ00}.
These estimates are listed in Lemmas~\ref{L01}--\ref{L03} and are proven in Sections~\ref{sec03a} and~\ref{sec03b}.
With these estimates, the proof of Theorem~\ref{T01} follows by choosing a function $\tau$ which decreases sufficiently rapidly so as to admit analytic dissipation, see~\eqref{EQ24}--\eqref{EQ25}.
For the construction of solutions, we must first obtain a local-existence result with weaker radius $\sigma(t)$ in order to get existence in arbitrarily high-order conormal Sobolev spaces. 
We then use the persistence estimates to upgrade the radius to $\tau(t)$ and iterate the local-existence argument up to the $H^r$ maximal time of existence~$T_\ast$.
Interestingly, the proof of this theorem and its associated lemmas do not require $\phi$ to be analytic. 
Though we do take $\phi(z) = z (1+z)^{-1}$, which is uniformly analytic on $\mathbb{R}^+$, it is likely that we could replace it with a non-analytic function which behaves like $\phi(x_3) \sim x_3$ for $x_3$ sufficiently small, see~\cite{FTZ}.

A key element of our estimates for Theorem~\ref{T01} is conormalization, which is the primary function of Lemmas~\ref{L06} and~\ref{L09} preceding the proofs of Lemmas~\ref{L01} and~\ref{L03}, respectively.
For instance, 
denoting $\bar{\partial}$ as any one of $\partial_1$ or $\partial_2$, 
$\beta,\gamma \in \mathbb{N}_0^3$ as multi-indices with $\gamma \leq \beta$,
and $J$ as defined in~\eqref{EQ09}--\eqref{EQ10},
the product $J^\gamma \bar{\partial}u_3 J^{\beta - \gamma} \partial_3 u_1$, 
which appears in the pressure estimates, is not immediately controlled by conormal norms due to the normal derivative. 
A simple remedy is to group the product as
  \begin{equation}
   J^\gamma \bar{\partial}u_3 J^{\beta - \gamma}\partial_3 u_1
   =
   (\phi^{-1} J^\gamma \bar{\partial}u_3) 
   (\phi J^{\beta - \gamma} \partial_3 u_1)
   =
   (\phi^{-1} J^{\gamma} \bar{\partial} u_3)
   (J^{\beta - \gamma +e_3}u_1)
   ,
   \label{introa}
  \end{equation}
from which the second term in the product is now ``conormalized'', and an elementary product bound may be used to control the entire term. 
In turn, this may produce a loss of one derivative relative to the original product when considering $\phi^{-1} J^\gamma \bar{\partial}  u_3$.
When $\gamma_3 \neq 0$, this causes no difficulty because, by incompressibility,
  \begin{equation}
   \phi^{-1} J^\gamma \bar{\partial} u_3 
   =
   \partial^{\gamma'} \phi^{\gamma_3-1} \partial_3^{\gamma_3-1} \bar{\partial} \partial_3u_3
   =
   J^{\gamma - e_3} \bar{\partial}(-\partial_1u_1 -\partial_2 u_2)
   ,
   \llabel{introb}
  \end{equation}
where $\gamma' = (\gamma_1,\gamma_2)$.
Indeed, after substitution into~\eqref{introa}, we obtain a conormalized product at an energy level analogous to the original product, $J^\gamma \bar{\partial} u_3 J^{\beta-\gamma} \partial_3 u_1$.
The resulting term can then be handled using standard product estimates and Sobolev embeddings. 
However, when $\gamma_3 = 0$, we must apply Hardy's inequality and incompressibility to bound, formally,
  \begin{equation}
   \|\phi^{-1}J^\gamma \bar{\partial} u_3\|
   \les
   \|J^\gamma \bar{\partial}^2 u\|
   ,
   \llabel{introc}
  \end{equation}
thus exhibiting a loss of one derivative when compared to the original product.
A similar situation occurs in the advection estimates.
To overcome this, we perform a delicate case analysis on the powers $J^\alpha$ in Lemmas~\ref{L06} and~\ref{L09} which exploits the ``room'' for embeddings allowed by the $H^3$ norm in the definition of $\Hco^{m,3}$ and certain restrictions on the powers of~$J$. 
Together, these features allow the estimates to close. 
Section~\ref{sec03a} then proceeds with the proof of Lemmas~\ref{L01} and~\ref{L02}.

Next, Section~\ref{sec03b} contains higher-regularity elliptic estimates, see Lemma~\ref{L07}, which antecede the main pressure estimate in Lemma~\ref{L03}. 
Note that previous literature suggests the inclusion of a small weight $\eps^{\alpha_3}$ in $\|\cdot\|_{\Hco^{m,3}}$ to handle the transfer of derivatives arising in such elliptic estimates.
However, we define $\Hco^{m,3}$, see~\eqref{EQ11}, in such a way that is compatible with the structure of the problem, see Lemmas~\ref{L07} and~\ref{L08} and the definition of $M_\alpha$, removing the need for this small weight.
The section then concludes with the proof of Lemma~\ref{L03}. 

In Section~\ref{sec04}, we state Lemmas~\ref{L10}--\ref{L12} and prove Theorem~\ref{T02}, establishing persistence in the space $A_\tau^\co$ through the $H^r$ norm, for $r \geq 4$. 
In Section~\ref{sec04a}, we prove Lemmas~\ref{L10} and~\ref{L11} which control terms arising from the advection.
The advection estimates for $A_\tau^\co$ share a similar structure to that of $X_\tau^\co$, however, the use of $L^k =(\phi \partial_3)^k$ in $A_\tau^\co$ makes a commutator term like $[L^{\alpha_3 - \beta_3},\partial_3]u$ unavoidable. 
This commutator necessitates an estimate on arbitrarily high orders of conormal derivatives~$L^j \phi$.
In Lemmas~\ref{L13} and~\ref{L14}, we control this term using the uniform analyticity of $\phi(z) = z(1+z)^{-1}$ and a change of variable through the Lambert-$W$ function.
This indicates that $A_\tau^\co$-persistence is dubious with a non-analytic conormal weight~$\phi$. 
In Section~\ref{sec04b}, we establish the pressure estimate in Lemma~\ref{L12}.
The technique used in Section~\ref{sec03b} to control the pressure for $X_\tau^\co$-persistence is not applicable here due to the structure of $L^k = (\phi \partial_3)^k$.
Instead, we employ a technique which transfers conormal derivatives to tangential derivatives and uses a small weight $\eps^{\alpha_3}$ in our definition of $\Eco^{m,3}$, see~\eqref{EQ201}, to absorb the highest-order terms appearing in the estimate.
The proof of Lemma~\ref{L12} then follows from techniques in previous sections. 

In Section~\ref{sec05}, we prove Theorem~\ref{T03} which states that, with $\eps = 1$ in the definition of $A_\tau^\co$, 
the norms $\|\cdot \|_{X_\tau^\co}$ and $\|\cdot \|_{A_\tau^\co}$ are equivalent up to a small loss in radius.
For the proof, it is necessary to derive estimates and exact identities relating the two operators $(\phi \partial_x)^n$ and~$\phi^n \partial_x^n$. 
This analysis is condensed into Propositions~\ref{P01} and~\ref{P02} which we state and prove in the Appendix.
When seeking such identities, 
they should be analogous to
  \begin{equation}
   (x\partial_x)^n
   =
   \sum_{k=1}^n 
     {n \brace k}
     x^k \partial_x^k
   \quad\quad
   \text{and}
   \quad\quad
   x^n \partial_x^n
   =
   \sum_{k=1}^n
     (-1)^{n-k}
     {n \brack k}
     (x\partial_x)^k
  ,
   \llabel{EQ400}
  \end{equation}
where ${n\brace k}$ and ${n \brack k}$ denote the Stirling numbers of the second and first kind, respectively.
These formulas were obtained by H.~Scherk in his 1823 PhD thesis\footnote{Since his thesis is in Latin, the reader may wish to read an abridged English translation in~\cite[Appendix~A]{BF}.}~\cite{S}.
Upon expanding $(x\partial_x)^n$ for small powers, Scherk found that the expansion of $(x\partial_x)^n$ is of the form~$\sum_{k=1}^n a_k^n x^k \partial_x^k$.
A recurrence relation,
  \begin{equation}
   a_k^n
   =
   a_{k-1}^{n-1}
   + k a_k^{n-1}
   ,
   \label{EQ400a}
  \end{equation}
is then promptly derived by considering $(x\partial_x)^n = (x\partial_x)(x\partial_x)^{n-1}$ and equating the coefficients of~$x^k\partial_x^k$.
From the recurrence~\eqref{EQ400a}, Scherk establishes an explicit formula for $a_k^n$ which, in the modern day, is identified with Stirling numbers of the second kind.
In our case, a similar analysis is necessary in expanding~$(\phi\partial_x)^n$.
However, a recurrence relation such as~\eqref{EQ400a} is not immediately at hand because an expansion of $(\phi \partial_x)^n$ involves coefficients which are polynomials of successively high derivatives of~$\phi$. 
Instead, our approach employs a much more abstract formula by Briand, Lopes, and Rosas~\cite{BLR} which computes the normally ordered form of $(h \partial)^n$, where $\partial$ is a derivation of a noncommutative ring $A$, and $h$ is a central element of~$A$; see also~\cite{M} for related results.
This formula, 
recounted in Lemma~\ref{L17} for $\partial = \partial_x$ and $h \colon \mathbb{R}_+\to \mathbb{R}$ a smooth function,
contains universal constant coefficients $c_\lambda^n$ whose explicit values are unknown aside from a wide array of implicit combinatorial definitions, see Remark~\ref{R06} and~\cite[Section~4.1]{BLR}.
However, the proof of Theorem~\ref{T03} compares analytic norms and, hence, requires precise estimates on the coefficients arising in the expansion
  \begin{equation}
   (\phi\partial_x)^n
   =
   \sum_{k=1}^n d_k^n(x) \phi^k \partial_x^k
   .
   \llabel{EQ400b}
  \end{equation}
Indeed, the desired estimates are established in Proposition~\ref{P01}, which we summarize below. 
Due to the special structure of $\phi$, the non-constant terms in $d_k^n(x)$ are quite tractable, see~\eqref{EQ401a} and~\eqref{EQ413}.
This leaves us with the task of explicitly bounding the constant coefficients 
  \begin{equation}
   S_{n,k} 
   =
   \sum_{\lambda \vdash n-k}
     c_\lambda^n
     \left( \prod_{i=1}^{\ell(\lambda)} \lambda_i!\right)
  ,
  \llabel{EQ400c}
  \end{equation}
where $\lambda \vdash n-k$ means that $\lambda = (\lambda_1, \dots, \lambda_\ell)$ is a partition of~$n-k$.
In fact, we derive a closed-form identity for $S_{n,k}$, which is given in Lemma~\ref{L18}.
The approach is similar to that of Scherk in that we use a recurrence relation for the normally ordered form $(h\partial)^n$, see~\cite{BLR} or~\eqref{EQ404a} below, to derive the recurrence
  \begin{equation}
   S_{n+1,k} 
   = 
   S_{n,k-1} + (2n-k) S_{n,k}
   .
   \llabel{EQ400e}
  \end{equation}
With Lemma~\ref{L18}, we then give an explicit bound on $\|d_k^n\|_{W^{3,\infty}}$, which is amenable to the analytic estimates in Theorem~\ref{T03}. 
Finally, in Proposition~\ref{P02}, we use a change of variables to derive 
  \begin{equation}
   \phi^n \partial_x^n
   =
   \sum_{k=1}^n a_k^n(x) (\phi \partial_x)^k
   ,
   \llabel{EQ400d}
  \end{equation}
along with explicit estimates on the coefficients~$a_k^n(x)$.
The change of variables is defined so that $\partial_x^k = (\psi(y) \partial_y)^k$, 
after which we are free to apply Proposition~\ref{P01} to complete the proof. 

\section{Preliminaries and main results}\label{sec02}
For a multi-index $\alpha = (\alpha_1, \alpha_2, \alpha_3) \in \mathbb{N}_0^3$, denote $\alpha' = (\alpha_1, \alpha_2)$.
We define the conormal operators
  \begin{equation}
   Z_k
   =
   \phi^k\partial_3^k
   \comma 
   k\in \mathbb{N}
   \label{EQ09}
  \end{equation}
and
  \begin{equation}
   J^\alpha
   =
   \partial^{\alpha'} Z_{\alpha_3}
   ,
   \label{EQ10}
  \end{equation}
where $\phi(z) = z(1+z)^{-1}$.
We also define the weighted conormal seminorm 
  \begin{equation}
   |u|_{\Hco^{m,3}}
   =
   \sum_{|\alpha| = m} M_\alpha \|J^\alpha u\|_{H^3}
   \label{EQ11}
  \end{equation}
with 
  \begin{equation}
   M_\alpha 
   = 
   {\alpha_1 + \alpha_2 \choose \alpha_1}
   .
   \llabel{EQ12}
  \end{equation}
For $\tau > 0$, we define the conormal analytic space $X_\tau^\co = \{v \in C^{\infty}(\Omega) \colon \|v\|_{X_\tau^\co} < \infty\}$, 
where
  \begin{equation}
   \|v\|_{X_\tau^\co} 
   = 
   \sum_{m =1}^\infty 
     |u|_{\Hco^{m,3}} \frac{\tau^{m-1}}{(m-1)!}
   .
   \label{EQ13}
  \end{equation}
We also define $Y_\tau^\co = \{v \in C^\infty(\Omega) \colon \|v\|_{Y_\tau^\co} < \infty\}$, 
with
  \begin{equation}
   \|v\|_{Y_\tau^\co}
   =
   \sum_{m = 2}^\infty 
     |u|_{\Hco^{m,3}} \frac{(m-1)\tau^{m-2}}{(m-1)!}
   .
   \label{EQ14}
  \end{equation}
Finally,
for $\tau(t)$ continuous,
we say that $u \in C([0,T_\ast); X_{\tau(t)}^\co)$ if, for all $t_0 < T_\ast$ and $\sigma < \tau(t_0)$,  
  \begin{equation}
   \lim_{t \to t_0} \|u(t) - u(t_0)\|_{X_\sigma^\co} 
   = 0
   .
   \llabel{EQ15a}
  \end{equation}

The first theorem gives persistence in $X_\tau^\co$ through the $H^r$ norm. 
The proof is in Section~\ref{sec03}.

\begin{Theorem}
\label{T01}
Let $u_0 \in H^r \cap X_{\tau_0}^\co$, where $r \geq 4$ and $ \tau_0 > 0$, be divergence free. 
Let $u \in  C([0,T_\ast); H^r)$ be the unique classical solution where $T_\ast$ is the maximal time of existence.
Then there exists a strictly decreasing positive smooth function $\tau(t)$ such that $\tau(0) = \tau_0$ and $u(t) \in X_{\tau(t)}^{\co}$ for all $t < T_\ast$. 
Moreover, $u \in C([0,T_\ast);X_{\tau(t)}^\co)$, and the radius $\tau(t)$ satisfies
  \begin{equation}
   \tau(t) 
   \geq 
   \tau_0  
   \exp \Bigg(
           -C \int_0^t \left( \|u\|_{H^r}^2 + \|u_0\|_{H^r} + \int_0^s \|u\|_{H^r}^2 \, d\sigma\right)
   ds \Bigg)
   .
   \label{EQ15}
  \end{equation}
\end{Theorem}

Our second theorem establishes persistence in a different variant of the conormal norm.
We define the operators 
  \begin{equation}
   L^k
   =
   (\phi \partial_3)^k
   \comma
   k \geq 1
   \llabel{EQ200a}
  \end{equation}
and
  \begin{equation}
   K^\alpha
   =
   \partial^{\alpha'} L^{\alpha_3}
   \comma
   \alpha \in \mathbb{N}_0^3
   .
   \label{EQ200}
  \end{equation}
For $\eps > 0$, we introduce the seminorm
  \begin{equation}
   |u|_{\Eco^{m,3}}
   =
   \sum_{|\alpha| = m} 
     M_\alpha \eps^{\alpha_3} \|K^\alpha u\|_{H^3}
   .
   \label{EQ201}
  \end{equation}
For $\tau > 0$, we define  the conormal analytic space  $A_\tau^\co = \{v \in C^\infty(\Omega) \colon \|v\|_{A_\tau^\co} < \infty\}$,
where
  \begin{equation}
   \|v\|_{A_\tau^\co}
   =
   \sum_{m=1}^\infty
     |u|_{\Eco^{m,3}} \frac{\tau^{m-1}}{(m-1)!}
   .
   \label{EQ202}
  \end{equation}
We also introduce $B_\tau^\co = \{v \in C^\infty(\Omega) \colon \|v\|_{B_\tau^\co} < \infty\}$, where
  \begin{equation}
   \|v\|_{B_\tau^\co}
   =
   \sum_{m=2}^\infty |u|_{\Eco^{m,3}} \frac{(m-1)\tau^{m-2}}{(m-1)!}
   .
   \label{EQ203}
  \end{equation}
The second theorem establishes persistence in $A_\tau^\co$ through the $H^r$ norm.
The proof is in Section~\ref{sec04}.
\begin{Theorem}
\label{T02}
Let $u_0 \in H^r \cap A_{\tau_0}^\co$, where $r \geq 4$ and $\tau_0 > 0$, be divergence free.
Let $u \in C([0,T_\ast); H^r)$ be the unique classical solution where $T_\ast$ is the maximal time of existence.
Then there exists a strictly decreasing positive smooth function $\tau(t)$ such that $\tau(0) = \tau_0$ and $u(t) \in A_{\tau(t)}^\co$ for all $t < T_\ast$. 
Moreover, $u \in C([0,T_\ast); X_{\tau(t)}^\co)$, and the radius $\tau(t)$ of conormal analyticity satisfies~\eqref{EQ15}.
\end{Theorem}
\begin{remark}
\label{R10}
For the proofs of Theorems~\ref{T01} and~\ref{T02}, we take $r = 4$ for simplicity. 
However, the proof for $r \geq 4$ is obtained by following the same
arguments as for~$H^4$.
We may also consider strengthening $X_\tau^\co$ and $A_\tau^\co$ by replacing $\Hco^{m,3}$ and $\Eco^{m,3 }$ with
  \begin{equation}
   |u|_{\Hco^{m,r-1}}
   =
   \sum_{|\alpha| =m}
     M_\alpha \|J^\alpha u\|_{H^{r-1}}
   \llabel{EQ203a}
  \end{equation}
and 
  \begin{equation}
   |u|_{\Eco^{m,r-1}}
   =
   \sum_{|\alpha| =m} 
    M_\alpha \|K^\alpha u\|_{H^{r-1}}
  .
    \llabel{EQ203b}
  \end{equation}
In this case, persistence in $H^p$ holds in either analytic norm for $p \geq r$.
\end{remark}
\begin{remark}
\label{R11}
Our proofs for Theorems~\ref{T01} and~\ref{T03} actually give a slightly sharper persistence result.
The conormal analytic norm $X_\tau^\co$ persists through the $\mathcal{H}$ norm, defined in~\eqref{EQ500} below,
and the $A_\tau^\co$-norm persists through a similar norm as $\mathcal{H}$ with $J$ replaced by~$K$.
To see this, we do the same conormalization arguments demonstrated in the proof of Lemma~\ref{L21} below.
\end{remark}

In the following sections, we frequently use the inequality,
  \begin{equation}
   \|\phi^{-1}v_3\|_{L^\infty}
   \les
   \|v_3\|_{L^\infty}
   + \|\bar{\partial} v\|_{L^\infty}
   ,
   \label{EQ28}
  \end{equation}
where $v$ is incompressible with $v_3 = 0$ on~$\partial \Omega$. 
We also have
  \begin{equation}
   \|\phi^{-1} v_3\|_{H^3}
   \les
   \|v_3\|_{H^3}
   + \|\bar{\partial}v \|_{H^3}
   .
   \label{EQ28b}
  \end{equation}
Both~\eqref{EQ28} and~\eqref{EQ28b} are proven using Hardy's inequality and incompressibility. 

Occasionally, we use the multiplicative Sobolev inequality,
  \begin{equation}
   \|ab\|_{H^k} 
   \les
   \|a\|_{H^\ell} \|b\|_{H^m}
   ,
   \label{EQ28a}
  \end{equation}
where either $\ell, m \geq k$ and $\ell + m > k+1.5$, or $\ell,m > k \geq 0$ and $\ell + m = k +1.5$.

Next, we record two lemmas which are essential for the estimates in Sections~\ref{sec03} and~\ref{sec04}. 
For proofs, see~\cite[Lemmas~4.1 and~4.2]{KV2}. 

\begin{Lemma}
\label{L05a}
We have 
  \begin{equation}
   M_\alpha M_\beta^{-1} M_{\alpha- \beta}^{-1} 
   {\alpha \choose \beta}
   \leq 
   {|\alpha'| \choose |\beta'|}{\alpha_3 \choose \beta_3}
   \leq 
   {|\alpha| \choose |\beta|}
   \label{EQ30a}
  \end{equation}
for all $\alpha,\beta \in \mathbb{N}_0^3$ with $\beta \leq \alpha$.
\end{Lemma}

\begin{Lemma}
\label{L05}
Let $\{x_\lambda\}_{\lambda \in \mathbb{N}_0^3}$ and $\{y_\lambda\}_{\lambda \in \mathbb{N}_0^3}$ be sequences of real numbers. 
Then we have
  \begin{equation}
   \sum_{|\alpha| = m} \sum_{|\beta| = j, \beta \leq \alpha} 
     x_\beta y_{\alpha- \beta}
   =
   \left( \sum_{|\beta| = j} x_\beta \right)
   \left( \sum_{|\gamma| = m-j} y_\gamma \right)
   .
   \llabel{EQ30}
  \end{equation}
\end{Lemma}
Finally, we discuss our third theorem. 
Define the norm $P_\tau^\co$ as $A_\tau^\co$ with $\eps = 1$.
The following theorem shows that $\|\cdot \|_{X_\tau^\co}$ and $\|\cdot \|_{P_\tau^\co}$ are equivalent up to a small loss of radius.
The proof is provided in Section~\ref{sec05}.
\begin{Theorem}
\label{T03}
Let $C_1$ and $C_2$ be the constants from Propositions~\ref{P01} and~\ref{P02}, respectively. 
If $0 < \tau < \min\{1,(4C_1)^{-1}, (2C_2)^{-1}\}$, then
  \begin{equation}
   (1-2C_2\tau)
   \|u\|_{X_{\rho_2(\tau)}^\co}
   \les
   \|u\|_{P_\tau^\co}
   \les
   \frac{1}{\sqrt{1-4C_1\tau}}
   \|u\|_{X_{\rho_1(\tau)}^\co}
   ,
   \label{EQ30a}
  \end{equation}
where 
  \begin{equation}
   \rho_1(\tau)
   =
   \frac{2\tau}{1+ \sqrt{1-4C_1\tau}}
   =
   \tau + C_1 \tau^2
   + O(\tau^3)
   \quad\quad
   \text{as $\tau \to 0$}
   \llabel{EQ30b}
  \end{equation}
and $\rho_2(\tau) = \tau - C_2 \tau^2$.
\end{Theorem}
\begin{remark}
\label{R03a}
If we replace $P_\tau^\co$ by $A_\tau^\co$ and include $\eps^{\alpha_3}$ in our definition of $X_\tau$,
then Theorem~\ref{T03} holds with $C_1$ and $C_2$ replaced by $C_1\eps$ and $C_2\eps$, respectively. 
Moreover, Theorem~\ref{T01} still holds for $X_\tau^\co$ when  $\eps^{\alpha_3}$ is included.
One could then derive persistence in $A_\tau^\co$ through $X_\tau^\co$ and vice versa at the cost of an instantaneous loss of radius after $t= 0$.
\end{remark}

\section{Proof of Theorem~\ref{T01}: Persistence in $X_\tau^\co$}\label{sec03}
Throughout the rest of the paper, the notation ``$\les $'' means
$\leq$ up to a positive multiplicative constant which is independent of $m$ or~$\alpha$.
We will also often write $\bar{\partial}$ to mean any one of $\partial_1$ or $\partial_2$ and $u_h$ to mean any one of $u_1 $ or~$u_2$.
Also, $\{e_j\}_{j=1}^3$ denotes the standard $\mathbb{R}^3$ basis, and $\bar{e}$ denotes any one of $e_1$ or~$e_2$.

Before we prove Theorem~\ref{T01}, we state several lemmas.
The first one contains estimates on commutator terms originating from the advection in~\eqref{EQ00}.

\begin{Lemma}
\label{L01}
Assume the hypotheses of Theorem~\ref{T01}. 
Then
  \begin{equation}
   \sum_{m = 1}^\infty \sum_{|\alpha| = m} \sum_{0 < \beta \leq \alpha} 
   {\alpha \choose \beta} M_\alpha \|J^\beta u \cdot J^{\alpha - \beta} \nabla u\|_{H^3}
   \les
   \|u\|_{H^4}^2
   + (\tau + \tau^2)\|u\|_{Y_\tau^\co} ( \|u\|_{H^4} + \|u\|_{X_\tau^\co})
   .
   \label{EQ16}
  \end{equation}
\end{Lemma}

Our estimates on the transport term also produce some lower-order quantities which are bounded in the second lemma.

\begin{Lemma}
\label{L02}
Assume the hypotheses of Theorem~\ref{T01}. 
Then
  \begin{equation}
   \sum_{m = 1}^\infty \sum_{|\alpha| = m} 
     M_\alpha \|u_3[J^\alpha, \partial_3] u\|_{H^3}
     \frac{\tau^{m-1}}{(m-1)!}
   \les
   \|u\|_{H^4}^2 
   + \tau\|u\|_{H^4}\|u\|_{Y_\tau^\co}
   \label{EQ17}
  \end{equation}
and 
  \begin{equation}
   \sum_{m =1}^\infty \sum_{|\alpha| = m} \sum_{|\gamma| \leq 3} \sum_{0 < \mu \leq \gamma}
     M_\alpha \|\partial^\mu u \cdot \nabla \partial^\gamma J^\alpha u\|_{L^2} 
     \frac{\tau^{m-1}}{(m-1)!}
   \les
   \|u\|_{H^4}^2 
   + \tau\|u\|_{H^4}\|u\|_{Y_\tau^\co}
   .
   \label{EQ18}
  \end{equation}
\end{Lemma}

The third lemma gives estimates on the pressure. 

\begin{Lemma}
\label{L03}
Assume the hypotheses of Theorem~\ref{T01}. 
Then
  \begin{equation}
   \|\nabla p\|_{X_\tau^\co} 
   \les
   \|u\|_{H^4}^2
   + (\tau + \tau^2 ) \|u\|_{Y_\tau^\co} (\|u\|_{H^4} + \|u\|_{X_\tau^\co})
   .
   \llabel{EQ19}
  \end{equation}
\end{Lemma}

Lemmas~\ref{L01} and~\ref{L02} are proven in Section~\ref{sec03} and Lemma~\ref{L03} is proven in Section~\ref{sec04}. 

We also note the following estimate which is used to show the time continuity of our solution in~$X_\tau^\co$.
It is proven at the end of Section~\ref{sec03b}. 

\begin{Lemma}
\label{L20}
Let $0 < \sigma < \tau$.
We have
  \begin{equation}
   \|u\cdot \nabla u\|_{X_\sigma^\co} 
   + \|\nabla p\|_{X_\sigma^\co}
   \les
   \|u\|_{H^4}^2
   + \|u\|_{H^4}\|u\|_{X_\tau^\co}
   + \|u\|_{X_\tau^\co}^2
   ,
   \llabel{EQ19a}
  \end{equation}
where the implicit constant depends on $\sigma$ and~$\tau$.
\end{Lemma}

Next, define $\|f\|_{\widetilde{X}_\tau^\co} = \|f\|_{\mathcal{H}} + \|f\|_{X_\tau^\co}$ and $\|f\|_{\widetilde{Y}_\tau^\co} = \|f\|_{H^3} + \|f\|_{Y_\tau^\co}$,
where
  \begin{equation}
   \|f\|_{\mathcal{H}} 
   =
   \|f\|_{H^3}
   + \sum_{|\alpha| = 1} \|J^\alpha f\|_{H^3}
   .
   \label{EQ500}
  \end{equation}
We note that $\|f\|_{\mathcal{H}} \les \|f\|_{H^4}$.
Finally, we state a lemma which is needed for the construction of solutions in the proof of Theorem~\ref{T01}.
It is proven in Section~\ref{sec03c}.

\begin{Lemma}
\label{L21}
We have
  \begin{equation}
   \|u \cdot \nabla u\|_{\widetilde{X}_\tau^\co}
   + \|\nabla p\|_{\widetilde{X}_\tau^\co}
   \leq
   C\|u\|_{\widetilde{X}_\tau^\co} \|u\|_{\widetilde{Y}_\tau^\co}
   .
   \label{EQ501}
  \end{equation}
\end{Lemma}

We are now ready to prove Theorem~\ref{T01}. 

\begin{proof}[Proof of Theorem~\ref{T01}]
We begin with a priori estimates.
That is, we assume that $u(t) \in X_{\tau(t)}^\co$ for some decreasing
positive smooth function $\tau$  with $\tau(0) = \tau_0$ and show that if $\tau$ decreases sufficiently rapidly, then 
  \begin{equation}
   \|u\|_{X_\tau^\co} 
   \les 
   \|u_0\|_{X_{\tau_0}^\co}
   + \int_0^t \|u\|_{H^4}^2\, ds
   .
   \label{EQ27}
  \end{equation}
Indeed, from the definitions~\eqref{EQ13} and~\eqref{EQ14}, 
we have
  \begin{equation}
   \frac{d}{dt} \|u\|_{X_\tau^\co}
   = 
   \dot{\tau} \|u\|_{Y_\tau^\co}
   + \sum_{m= 1}^\infty
       \left( \frac{d}{dt} |u |_{\Hco^{m,3}}\right) 
       \frac{\tau^{m-1}}{(m-1)!}
   .    
   \label{EQ20}
  \end{equation}
We estimate $(d/dt) |u|_{\Hco^{m,3}}$. 
Applying $J^\alpha$ to~\eqref{EQ00} yields
  \begin{equation}
   \partial_t J^\alpha u 
   + u \cdot \nabla J^\alpha u
   + \sum_{0 < \beta \leq \alpha} {\alpha \choose \beta} 
       J^\beta u \cdot J^{\alpha - \beta} \nabla u
   + u_3 [J^\alpha, \partial_3] u
   + J^\alpha \nabla p 
   = 0
   .
   \llabel{EQ21}
  \end{equation}
Further applying $\partial^\gamma$, with $0 \leq |\gamma| \leq 3$, gives us 
\begin{align}
   \begin{split}
    &
    \partial_t \partial^\gamma J^\alpha u
    + u \cdot \nabla \partial^\gamma J^\alpha u
    + \sum_{0 < \mu \leq \gamma} {\gamma \choose \mu} 
      \partial^\mu u \cdot \nabla \partial^{\gamma-\mu} J^\alpha u
    \\&\indeq\indeq
    + \partial^\gamma \sum_{0 < \beta \leq \gamma } {\alpha \choose \beta} 
      J^\beta u \cdot J^{\alpha - \beta} \nabla u
    +\partial^\gamma (u_3 [J^\alpha, \partial_3] u)
    + \partial^\gamma J^\alpha \nabla p
    = 0
    .
   \end{split}
   \label{EQ22}
  \end{align}
Now, we multiply~\eqref{EQ22} with $M_\alpha$, test with $\partial^\gamma J^\alpha u$, and sum over $|\alpha| = m$ and $0 \leq |\gamma| \leq 3$,
to get
\begin{align}
   \begin{split}
    \frac{d}{dt} |u|_{\Hco^{m,3}}
    &\leq
    \sum_{|\alpha| = m} \sum_{0 < \beta \leq \alpha} {\alpha \choose \beta} 
      M_\alpha \|J^\beta u\cdot J^{\alpha - \beta} \nabla u\|_{H^3}
    + \sum_{|\alpha| = m} 
        M_\alpha \|u_3[J^\alpha, \partial_3] u\|_{H^3}
    \\&\indeq\indeq
    + \sum_{|\alpha| = m}\sum_{0 \leq |\gamma| \leq 3} \sum_{0 < \mu \leq \gamma} {\gamma \choose \mu} 
        M_\alpha \|\partial^\mu u \cdot \nabla \partial^\gamma J^\alpha u\|_{L^2}
    + \sum_{|\alpha| = m}
        M_\alpha \|J^\alpha \nabla p\|_{H^3}
    ,
   \end{split}
   \label{EQ23}
  \end{align}
where the transport term $u\cdot \nabla \partial^\gamma J^\alpha u$ vanished due to the incompressibility and impermeability conditions. 
Finally, after multiplying~\eqref{EQ23} by $\tau^{m-1}(m-1)!^{-s}$, summing over $m\geq 1$, and invoking Lemmas~\ref{L01}--\ref{L03}, 
we infer that
  \begin{equation}
   \frac{d}{dt} \|u\|_{X_\tau^\co}
   -\dot{\tau} \|u\|_{Y_\tau^\co}
   \leq
   C_0\|u\|_{H^4}^2
   + C_1(\tau+\tau^2) \|u\|_{Y_\tau^\co}( \|u\|_{H^4} + \|u\|_{X_\tau^\co})
   ,
   \label{EQ24}
  \end{equation}
where $C_0,C_1$ denote positive universal constants. 
Now, choose a smooth  positive decreasing function $\tau$ such that 
  \begin{equation}
   \dot{\tau} 
   + C_1(\tau + \tau^2) (\|u\|_{H^4} + \|u\|_{X_\tau^\co}) 
   \leq 
   0
   .
   \label{EQ25}
  \end{equation}
The inequalities \eqref{EQ24} and \eqref{EQ25} then imply
  \begin{equation}
   \frac{d}{dt} \|u\|_{X_\tau^\co} 
   \les
   \|u\|_{H^4}^2
   .
   \llabel{EQ26}
  \end{equation}
Upon integrating over $[0,t ]$, 
we obtain~\eqref{EQ27}, while
integration of~\eqref{EQ25} yields~\eqref{EQ15}.

To conclude the a priori estimates, we show that $u \in C([0,T_\ast); X_{\tau(t)}^\co)$.
Fix $t_0 < T_\ast$ and $\sigma < \tau(t_0)$.
By continuity of $\tau$, there exists a compact interval $I \ni t_0$ such that $\sigma < \tau(t)$ for all $t \in I$.
Then, by Lemma~\ref{L20}, the bound~\eqref{EQ27}, and time-continuity of $u$ in $H^4$, we have $\sup_{t \in I}\|u_t(t)\|_{X_{\sigma}^\co} \leq M$ for some $M > 0$.
Hence, for $t \in I$, 
  \begin{equation}
   \|u(t) - u(t_0)\|_{X_{\sigma}^\co}
   \leq
   \int_{t_0}^t \|u_t(s)\|_{X_\sigma^\co}
   \leq
   M(t-t_0)
   .
   \llabel{EQ27a}
  \end{equation}

Now, we prove that the solution is analytic as obtained in the
a~priori estimates; in particular, we prove that we can take
$\tau(t) = G(t)$, where $G(t)$ is the right-hand side of~\eqref{EQ15}.
We begin with establishing local existence in the space $\widetilde{X}_{\sigma(t)}^\co$, where $\sigma(t) = \tau_0 - Mt$ and $M > 1$.
More precisely, we show that for every  $u_0 \in X_{\tau_0}^\co \cap H^4$, there exists a sufficiently large $M > 1$, a sufficiently small $T_0 > 0$, and a unique local-in-time solution $u \in C([0,T_0]; \widetilde{X}_\sigma^\co)$ to~\eqref{EQ00} satisfying 
  \begin{equation}
   \sup_{t \in [0,T_0]} \|u\|_{\widetilde{X}_\sigma^\co}
   + M \int_0^{T_0} 
       \|u\|_{\widetilde{Y}_\sigma^\co}
   \leq
   A
   ,
   \llabel{EQ511}
  \end{equation}
where 
  \begin{equation}
   A
   =
   4\|u_0\|_{\widetilde{X}_{\tau_0}^\co}
   .
   \llabel{EQ512}
  \end{equation}
To show this, we consider the sequence given by 
  \begin{equation}
   u_t^{(n+1)}
   + u^{(n)} \cdot \nabla u^{(n)}
   + \nabla p^{(n)}
   = 0
   \label{EQ513}
  \end{equation}
and 
  \begin{align}
   \begin{split}
    -\Delta p^{(n)}
    &=
    \partial_i u_j^{(n)} \partial_j u_i^{(n)}
    \inon{in $\Omega$,}
    \\
    \partial_3 p^{(n)}
    &=
    0
    \inon{on $\partial \Omega$}
    ,
   \end{split}
   \label{EQ514}
  \end{align}
with $u^{(0)} = u_0$.
By~\eqref{EQ514} and induction on $n$, we have $\nabla \cdot u^{(n)} = 0$ and $u_3^{(n)} = 0$.
Now, we claim that there exist $M > 1$ and $T_0 = T_0(M)$ such that 
  \begin{equation}
   \sup_{t \in [0,T_0]} \|u^{(n)}\|_{\widetilde{X}_\sigma^\co}
   + M \int_0^{T_0} \|u^{(n)}\|_{\widetilde{Y}_\sigma^\co}
   \leq
   A
   \label{EQ515}
  \end{equation}
for all $n \geq 0$. 
Indeed, by Lemma~\ref{L21} and testing~\eqref{EQ513} with $u$ in the $\widetilde{X}_\sigma^\co$-norm, we obtain
  \begin{equation}
   \frac{d}{dt} \|u^{(n+1)}\|_{\widetilde{X}_\sigma^\co}
   + M \|u^{(n+1)}\|_{Y_\sigma^\co}
   \leq
   C \|u^{(n)}\|_{\widetilde{X}_\sigma^\co} \|u^{(n)}\|_{\widetilde{Y}_\sigma^\co}
   .
   \label{EQ516}
  \end{equation}
For $n=0$, the bound~\eqref{EQ515} is immediately true.
For $n \geq 0$, we assume that it holds for $n$ and use~\eqref{EQ516} to write
  \begin{align}
   \begin{split}
    &
    \sup_{t\in [0,T_0]} \|u^{(n+1)}\|_{\widetilde{X}_\sigma^\co}
    + M \int_0^{T_0} \|u^{(n+1)}\|_{\widetilde{Y}_\sigma^\co}
    \\&\indeq\indeq\leq
    (1+T_0M) \sup_{t \in [0,T_0]} \|u^{(n+1)}\|_{\widetilde{X}_\sigma^\co}
    + M \int_0^{T_0} \|u^{(n+1)}\|_{Y_\sigma^\co}
    \\&\indeq\indeq\leq
    2\|u_0\|_{\widetilde{X}_{\tau_0}^\co}
    + C\sup_{t \in [0,T_0]} 
      \|u^{(n)}\|_{\widetilde{X}_\sigma^\co}
      \int_0^{T_0} \|u^{(n)}\|_{\widetilde{Y}_\tau^\co}
    \\&\indeq\indeq\leq
    2 \|u_0\|_{\widetilde{X}_{\tau_0}^\co}
    + C A^2M^{-1}
    \leq 
    A
    ,
   \end{split}
   \llabel{EQ517}
  \end{align}
where we set $M \geq 10 CA$ and  $T_0 \leq \min \{1,\tau_0\}/M$.
Hence, the bound~\eqref{EQ515} holds uniformly for all $n \geq 0$. 
By a standard fixed point argument, we obtain a solution $u \in C([0,T_0]; \widetilde{X}_\sigma^\co)$.
Since the initial data is in $H^r$, this coincides with the classical $H^r$ solution on $[0,T_0]$ due to uniqueness.
This completes the proof of local existence.

Next, we show that, in fact, $u(t) \in X_{\tau(t)}^\co$ for all $t \in [0,T_0]$, where $\tau(t) = G(t)$.
For $N \geq 1$, let $\tau_N \colon [0,T_0] \to (0,\tau_0]$ be a strictly decreasing  positive smooth function with $\tau_N(0)=  \tau_0$
and define the truncated norm
  \begin{equation}
   \|u\|_{X_{\tau_N}^{\co,N}}
   =
   \sum_{m=1}^N \sum_{|\alpha| =m} 
     M_\alpha 
     \|J^\alpha u\|_{H^3}
     \frac{\tau_N^{m-1}}{(m-1)!}
   .
   \llabel{EQ518}
  \end{equation}
Note that this is well-defined since local existence in $C([0,T_0];\widetilde{X}_\sigma^\co)$ with any positive radius $\sigma$ implies $\|J^\alpha u\|_{H^3} < \infty$ for all $\alpha \in \mathbb{N}_0^3$.
We then repeat the a~priori estimates in~\eqref{EQ20}--\eqref{EQ24} for the truncated norm to obtain 
  \begin{equation}
   \frac{d}{dt} \|u\|_{X_{\tau_N}^{\co,N}}
   - \dot{\tau}_N \|u\|_{Y_{\tau_N}^{\co,N}}
   \leq
   C_0\|u\|_{H^4}^2
   + C_1(\tau_N + \tau_N^2) \|u\|_{Y_{\tau_N}^{\co,N}} 
   (\|u\|_{X_{\tau_N}^{\co,N}} + \|u\|_{H^4})
   ,
   \llabel{EQ519}
  \end{equation}
where $C_0,C_1$ are positive constants independent of~$N$.
If $\tau_N$ decreases sufficiently fast, then
  \begin{equation}
   \dot{\tau}_N
   + C_1(\tau_N + \tau_N^2)
     (\|u\|_{X_{\tau_N}^{\co,N}} + \|u\|_{H^4})
   \leq 0
   \label{EQ520}
  \end{equation}
and 
  \begin{equation}
   \|u(t)\|_{X_{\tau_N}^{\co,N}}
   \les
   \|u_0\|_{X_{\tau_0}^{\co,N}}
   + \int_0^t \|u\|_{H^4}^2
   \leq
   \|u_0\|_{X_{\tau_0}^\co}
   + \int_0^t \|u\|_{H^4}^2
   =: \widehat{M}(t)
   \label{EQ521}
  \end{equation}
for $t \in [0,T_0]$.
Hence, choosing $\tau_N = \tau$, where $\tau$ satisfies 
  \begin{equation}
  \dot{\tau}  
   + C_1(\tau + \tau^2)
   (\widehat{M}(t) + \|u\|_{H^4})
   ,
   \llabel{EQ522}
  \end{equation}
is sufficient for~\eqref{EQ520} to hold. 
Notably, $\tau$ is uniform in $N$ and $\|u(t)\|_{X_\tau^{\co,N}}$ is bounded and monotonic by~\eqref{EQ521}. 
Passing to the limit, we find that  $u(t) \in X_{\tau(t)}^\co$ and satisfies the estimate~\eqref{EQ521} for all $t \in [0,T_0]$.
Moreover, by the continuity arguments in the a~priori estimates, we have $u \in C([0,T_0];X_{\tau(t)}^\co)$.

Finally, we demonstrate that this construction may be iterated up to the $H^4$ maximal time of existence~$T_\ast$. 
Without loss of generality, assume that $\tau_0 < 1$.
Let $0 < \widehat{T} < T_\ast$. 
Note that 
  \begin{equation}
   10CA 
   \leq 
   C_2 \widehat{M}(\widehat{T}) 
   ,
   \llabel{EQ523}
  \end{equation}
for some time-independent constant~$C_2$.
Our previous local-existence and persistence arguments with $M$ replaced by $\widehat{M}(\widehat{T})$ yield a solution $u \in C([0,T_0]; X_{\tau(t)}^{\co})$ where $T_0 = \tau_0/(C_2 \widehat{M}(\widehat{T}))$.
Then, with the radius $\sigma(t) = \tau(T_0) - C Mt$ and $A$ replaced by 
  \begin{equation}
   A_1
   =
   4\|u(T_0)\|_{X_{\tau(T_0)}^\co}
   ,
   \llabel{EQ524}
  \end{equation}
we obtain a continuation of the solution to $T_1 = T_0 + \tau(T_0)/(C_2\widehat{M}(\widehat{T}))$ since $A_1 \leq C_2 \widehat{M}(\widehat{T})$ by the first iteration's persistence argument. 
Continuing with this, we obtain 
  \begin{equation}
   T_n 
   = 
   \frac{1}{C_2\widehat{M}(\widehat{T})}
   \sum_{k=0}^{n-1} \tau(T_k)
   .
   \llabel{EQ525}
  \end{equation}
Provided that $T_k < \widehat{T}$, the definition of $ G(t)$ implies that $\tau(T_k) > \tau(\widehat{T})$.
Hence, 
  \begin{equation}
   T_n 
   \geq 
   \frac{n \tau(\widehat{T})}{C_2 \widehat{M}(\widehat{T})}
   =:
   C(\widehat{T}) n
   .
   \llabel{EQ526}
  \end{equation}
Since $C(\widehat{T})$ is independent of $n$, there exists $n \geq 0$ such that $T_{n-1} < \widehat{T} < T_n$.
This completes the proof of the theorem. 
\end{proof}

\subsection{Advection estimates}\label{sec03a}
Here, we prove Lemmas~\ref{L01} and~\ref{L02}. 
Before this, we establish a conormalization lemma.  

\begin{Lemma}
\label{L06}
Let $\alpha,\beta \in \mathbb{N}_0^3$ with $0 < \beta < \alpha$. 
Then 
  \begin{equation}
   \|J^\beta u_3 J^{\alpha- \beta} \partial_3 u\|_{H^3}
   \les
   E_{\alpha,\beta}^1
   + E_{\alpha,\beta}^2
   + E_{\alpha,\beta}^3
   + E_{\alpha,\beta}^4
   ,
   \label{EQ31}
  \end{equation}
where
  \begin{align}
   \begin{split}
    E_{\alpha,\beta}^1
    &=
    1_{\{\beta_3 = 0, \alpha_3 \geq 1, \beta' < \alpha'\}}
      \|J^\beta \bar{\partial} u\|_{H^3} \|J^{\alpha -\beta + e_3 - \bar{e}}u\|_{H^3},
    \\
    E_{\alpha,\beta}^{2}
    &= 
    1_{\{\beta_3 = 0, \alpha_3 \geq 1, \beta' = \alpha'\}}
      \alpha_3
      (\|J^\beta u\|_{H^3} + \|J^\beta \bar{\partial} u\|_{H^3}) \|J^{\alpha- \beta} u\|_{H^3} ,
    \\
    E_{\alpha,\beta}^3
    &=
    1_{\{\beta_3 \geq 1\}}
      \|J^{\beta - e_3} \bar{\partial} u\|_{H^3} \|J^{\alpha - \beta + e_3} u\|_{H^3}
    + 1_{\{\beta_3 = 0\}} 
      \|J^\beta u\|_{H^3} \|J^{\alpha- \beta + e_3} u\|_{H^3},
    \\
    E_{\alpha,\beta}^4
    &=
    1_{\{\beta_3 = 0, \alpha_3 \geq 1, \beta' < \alpha'\}}
      \|J^\beta u\|_{H^3} \|J^{\alpha - \beta + e_3 -\bar{e}} u\|_{H^3}
      + 1_{\{\beta_3 = 0, \alpha_3 = 0\}}
      \|J^\beta u\|_{H^3} \|J^{\alpha - \beta} u\|_{H^3}
    .
   \end{split}
   \llabel{EQ32}
  \end{align}
\end{Lemma}
\begin{remark}
\label{R01}
The terms $E_{\alpha,\beta}^i$ are organized based on energy level. 
Indeed, for the estimates in the proof of Lemma~\ref{L01}, $E_{\alpha,\beta}^3$ and $E_{\alpha,\beta}^4$ can effectively be identified as
  \begin{equation}
   E_{\alpha,\beta}^3
   \sim
   \|J^\beta u\|_{H^3} \|J^{\alpha - \beta + e_3}u\|_{H^3}
   \llabel{EQ34}
  \end{equation}
and
  \begin{equation}
   E_{\alpha,\beta}^4
   \sim 
   1_{\{\beta_3 = 0, \alpha_3 \geq 1, \beta' < \alpha'\}}
     \|J^\beta u\|_{H^3} \|J^{\alpha - \beta} u\|_{H^3}
   .
   \llabel{EQ35}
  \end{equation}
The terms $E_{\alpha,\beta}^1$ and $E_{\alpha,\beta}^2$ are separated by energy level as well; however, they are special in that  the restriction on $\alpha, \beta$ plays a significant role in the closure of their estimates in the proof of Lemma~\ref{L01}.
\end{remark}

\begin{proof}[Proof of Lemma~\ref{L06}]
The proof consists of estimating the left-hand side of~\eqref{EQ31} under each of the following cases:
  \begin{align}
   \begin{split}
    &
    \beta_3 \geq 1,
    \\&
    \beta_3 =0, 
    \alpha_3 \geq 1, 
    \beta' < \alpha',
    \\&
    \beta_3 = 0,
    \alpha_3 \geq 1,
    \beta' = \alpha',
    \\&
    \beta_3 = 0,
    \alpha_3 = 0.
   \end{split}
   \label{EQ33}
  \end{align}
Before any casework, note that
  \begin{equation}
   J^\beta u_3 J^{\alpha - \beta}\partial_3 u
   =
   \phi^{-1} J^\beta u_3 \phi J^{\alpha - \beta} \partial_3 u
   =
   \phi^{-1} J^\beta u_3 J^{\alpha - \beta + e_3} u
   .
   \label{EQ36}
  \end{equation}
If $\beta_3 \geq 1$,
we have
  \begin{equation}
   \phi^{-1} J^\beta u_3
   =
   \partial^{\beta'} \phi^{\beta_3 - 1} \partial_3^{\beta_3 -1} \partial_3 u_3
   =
   J^{\beta - e_3} (- \partial_1 u_1 -\partial_2u_2)
   \llabel{EQ37}
  \end{equation}
by incompressibility.
Since $H^3$ is an algebra,
we bound 
  \begin{equation}
   1_{\{\beta_3 \geq 1\}}
     \|J^\beta u _3 J^{\alpha - \beta} \partial_3 u\|_{H^3}
   \les
   1_{\{\beta_3 \geq 1\}}
     \|J^{\beta - e_3} \bar{\partial} u\|_{H^3} \|J^{\alpha - \beta + e_3}u\|_{H^3}
   .
   \label{EQ39}
  \end{equation}
If $\beta_3 = 0$, $\alpha_3 \geq 1$, and $\beta' < \alpha'$,
we instead use the product estimate,
  \begin{equation}
   \|fg\|_{H^3} 
   \les
   \|f\|_{L^\infty} \|g\|_{H^3}
   + \|f\|_{H^3} \|g\|_{L^\infty}
   ,
   \label{EQ40}
  \end{equation}
on~\eqref{EQ36}.
By~\eqref{EQ28b}, 
  \begin{equation}
   \|\phi^{-1} J^\beta u_3\|_{H^3}
   \les
   \|J^\beta u\|_{H^3}
   + \|J^\beta \bar{\partial} u\|_{H^3}
   ,
   \label{EQ41}
  \end{equation}
and, by~\eqref{EQ28} and the embedding $H^3 \hookrightarrow W^{1,\infty}$,
  \begin{equation}
   \|\phi^{-1} J^\beta u_3\|_{L^\infty}
   \les
   \|J^\beta u\|_{L^\infty} 
   + \|J^\beta \bar{\partial} u\|_{L^\infty}
   \les
   \|J^\beta u\|_{H^3}
   .
   \llabel{EQ42}
  \end{equation}
For $J^{\alpha - \beta + e_3}u$, 
observe that
  \begin{equation}
   \|J^{\alpha - \beta + e_3} u\|_{L^\infty}
   =
   \|\bar{\partial} J^{\alpha - \beta + e_3 -\bar{e}} u\|_{L^\infty}
   \les
   \|J^{\alpha - \beta + e_3 - \bar{e}} u\|_{H^3}
   ,
   \label{EQ43}
  \end{equation}
where we used $\beta' < \alpha'$ in the first equality.
Combining the estimates~\eqref{EQ41}--\eqref{EQ43} with~\eqref{EQ36} and~\eqref{EQ40}
yields
  \begin{align}
   \begin{split}
    &
    1_{\{\beta_3 = 0, \alpha_3 \geq 1, \beta' < \alpha'\}}
      \|J^\beta u_3 J^{\alpha- \beta}\partial_3 u\|_{H^3}
    \\&\indeq\les
    1_{\{\beta_3 = 0, \alpha_3 \geq 1, \beta' < \alpha'\}}
      (\|J^\beta u\|_{H^3} + \|J^\beta \bar{\partial} u\|_{H^3}) \|J^{\alpha - \beta + e_3 - \bar{e}} u\|_{H^3}
    \\&\indeq\indeq
    +1_{\{\beta_3 = 0, \alpha_3 \geq 1, \beta' < \alpha'\}}
      \|J^\beta u\|_{H^3} \|J^{\alpha- \beta + e_3} u\|_{H^3}
    .
   \end{split}
   \label{EQ44}
  \end{align}
If $\beta_3 =0$, $\alpha_3 \geq 1$, and $\beta' = \alpha'$, 
we repeat the computations in the previous case, except that we replace~\eqref{EQ43} with
  \begin{equation}
   \|J^{\alpha - \beta + e_3}u\|_{L^\infty}
   \les
   \alpha_3
   \|J^{\alpha - \beta} u\|_{H^3}
   ,
   \label{EQ45}
  \end{equation}
giving
  \begin{align}
   \begin{split}
   &
   1_{\{\beta_3 = 0, \alpha_3 \geq 1, \beta' = \alpha'\}} 
     \|J^\beta u_3J^{\alpha - \beta} \partial_3 u\|_{H^3}
   \\&\indeq\les
   1_{\{\beta_3 = 0, \alpha_3 \geq 1, \beta' = \alpha'\}}
     \alpha_3
     (\|J^\beta u\|_{H^3} + \|J^\beta \bar{\partial} u\|_{H^3})\|J^{\alpha - \beta} u\|_{H^3} 
   \\&\indeq\indeq
   + 1_{\{\beta_3 = 0, \alpha_3 \geq 1, \beta' = \alpha'\}}
     \|J^\beta u\|_{H^3} \|J^{\alpha - \beta + e_3}u\|_{H^3}
     .
   \end{split} 
   \label{EQ46}
  \end{align}
To see that~\eqref{EQ45} holds,
note that 
  \begin{equation}
   \phi J^{\alpha - \beta} \partial_3 u
   = 
   \phi \partial_3 (J^{\alpha - \beta} u)
   - \alpha_3 \phi' J^{\alpha - \beta} u
   ,
   \llabel{EQ47}
  \end{equation}
where $\|\phi\|_{W^{1,\infty}} \les 1$.

Finally, we consider the case when $\beta_3 = 0$ and $\alpha_3 = 0$.
We take a different approach in that, using the Leibniz rule, we explicitly expand the $H^3$ norm, 
obtaining
  \begin{equation}
   \|J^\beta u_3 J^{\alpha - \beta} \partial_3 u\|_{H^3}
   \les
   \sum_{0 \leq |\gamma| \leq 3} \sum_{0 \leq \mu \leq \gamma} {\gamma \choose \mu}
     \|\partial^\mu J^\beta u_3 \partial^{\gamma - \mu} J^{\alpha - \beta} \partial_3 u\|_{L^2}
   .
   \label{EQ48}
  \end{equation}
Without loss of generality, we only consider the terms in the sum corresponding to $|\gamma| = 3$ as the terms from $|\gamma|=0,1,2$ give subordinate bounds. 
Splitting off the case when $\mu = 0$, 
we have
  \begin{align}
   \begin{split}
    &
    \sum_{|\gamma| = 3} \sum_{0 \leq \mu \leq \gamma} {\gamma \choose \mu}
     \|\partial^\mu J^\beta u_3 \partial^{\gamma - \mu} J^{\alpha - \beta} \partial_3 u\|_{L^2}
    \\&\indeq\indeq \les
    \sum_{|\gamma| = 3}
      \|\phi^{-1}J^\beta u_3 \phi\partial^\gamma J^{\alpha - \beta}\partial_3 u\|_{L^2}
    + \sum_{|\gamma| = 3} \sum_{j=1}^3 \sum_{|\mu| = j, \mu \leq \gamma}
    \|\partial^\mu J^\beta u_3 \partial^{\gamma - \mu} J^{\alpha - \beta} \partial_3 u\|_{L^2} 
    ,
   \end{split} 
   \label{EQ49}
  \end{align}
where, by~\eqref{EQ28} and the embedding $H^3 \hookrightarrow W^{1,\infty}$,
  \begin{align}
   \begin{split}
    &
    \|\phi^{-1} J^\beta u_3 \phi \partial^\gamma J^{\alpha - \beta} \partial_3 u\|_{L^2} 
    \\&\indeq\les
    \|\phi^{-1} J^\beta u_3\|_{L^\infty}
    \|\phi \partial^\gamma J^{\alpha - \beta} \partial_3 u\|_{L^2}
    \les
    \|J^\beta u\|_{H^3}
    \left( \|[\phi, \partial^\gamma] J^{\alpha - \beta} \partial_3 u\|_{L^2}+ \|J^{\alpha - \beta + e_3} u\|_{H^3} \right)
    \\&\indeq\les
    \|J^\beta u\|_{H^3}
    \left(\|J^{\alpha - \beta} u\|_{H^3} + \|J^{\alpha - \beta + e_3} u\|_{H^3}\right)
    .
   \end{split}
   \label{EQ50}
  \end{align}
For the third inequality above,
note that $[\phi,\partial^\gamma]f$ contains at most two derivatives on~$f$.
Next, the rightmost term in~\eqref{EQ49} is bounded without any conormalization.
Indeed, we simply note that $[\partial_3,J^{\alpha-\beta}] = 0$ when $\alpha_3 = \beta_3 = 0$ and use the multiplicative Sobolev inequality~\eqref{EQ28a} 
to get 
  \begin{align}
   \begin{split}
   &
   \sum_{|\gamma| = 3} \sum_{j=1}^3 \sum_{|\mu| = j, \mu \leq \gamma}
     \|\partial^\mu J^\beta u_3 \partial^{\gamma - \mu} J^{\alpha - \beta} \partial_3 u\|_{L^2}
   \\&\indeq
   \les
   \sum_{|\gamma| = 3} \sum_{j=1}^3 \sum_{|\mu| = j , \mu \leq \gamma}
     \|\partial^\mu J^\beta u_3\|_{H^{3-j}} \|\partial^{\gamma - \mu} \partial_3 J^{\alpha - \beta} u\|_{H^{j-1}}
   \\&\indeq
   \les 
   \|J^\beta u\|_{H^3}\|J^{\alpha - \beta} u\|_{H^3}
   .
   \end{split}
   \label{EQ51}
  \end{align}
Hence, from~\eqref{EQ48},~\eqref{EQ50}, and~\eqref{EQ51}, 
we obtain
  \begin{equation}
   1_{\{\beta_3 = 0, \alpha_3 = 0\}}
     \|J^\beta u_3 J^{\alpha - \beta} \partial_3 u\|_{H^3}
   \les
   1_{\{\beta_3 =0, \alpha_3 = 0\}}
     \|J^\beta u\|_{H^3}
     \left( \|J^{\alpha - \beta} u\|_{H^3} + \|J^{\alpha - \beta + e_3} u\|_{H^3}\right)
   .
   \label{EQ52}
  \end{equation}

Finally, since the cases given by~\eqref{EQ33} are exhaustive, we add together the estimates~\eqref{EQ39},~\eqref{EQ44},~\eqref{EQ46}, and~\eqref{EQ52} to acquire~\eqref{EQ31}. 
\end{proof}

We are now ready to prove Lemma~\ref{L01}.

\begin{proof}[Proof of Lemma~\ref{L01}]
First, we split the sum on the left-hand side of~\eqref{EQ16} into tangential and normal parts.
More precisely,
  \begin{equation}
   \sum_{m=1}^\infty \sum_{|\alpha| = m} \sum_{0 < \beta \leq \alpha} {\alpha \choose \beta} 
     M_\alpha \|J^\alpha u \cdot J^{\alpha - \beta} \nabla u\|_{H^3}
     \frac{\tau^{m-1}}{(m-1)!}
   \leq
   \sum_{m= 1}^\infty \sum_{j=1}^m 
     (\calC_{m,j}^{\rm tan} 
     + \calC_{m,j}^{\rm nor})
   ,
   \label{EQ53}
  \end{equation}
where 
  \begin{equation}
   \calC_{m,j}^{\rm tan} 
   = 
   \sum_{|\alpha| = m} \sum_{ |\beta| = j, \beta \leq \alpha} {\alpha \choose \beta}
   M_\alpha \|J^\beta u_h J^{\alpha - \beta} \bar{\partial} u\|_{H^3}
     \frac{\tau^{m-1}}{(m-1)!}
   ,
   \llabel{EQ54}
  \end{equation}
and 
  \begin{equation}
   \calC_{m,j}^{\rm nor}
   =
   \sum_{|\alpha| = m} \sum_{|\beta| =j, \beta \leq \alpha} {\alpha \choose \beta}
     M_\alpha \|J^\beta u_3 J^{\alpha - \beta} \partial_3 u\|_{H^3}
     \frac{\tau^{m-1}}{(m-1)!}
   .
   \llabel{EQ55}
  \end{equation}

We claim that
  \begin{equation}
   \sum_{m=1}^\infty \sum_{j=1}^m 
     \calC_{m,j}^{\rm tan}
   \les
   \|u\|_{H^4}^2
   + \tau \|u\|_{H^4} \|u\|_{Y_\tau^\co}
   + \tau \|u\|_{X_\tau^\co} \|u\|_{Y_\tau^\co}
   .
   \label{EQ55a}
  \end{equation}
Indeed, 
we start by extracting the lower-order terms out of the sum.
That is,
  \begin{equation}
   \sum_{m = 1}^\infty \sum_{j =1}^m 
     \calC_{m,j}^{\rm tan} 
   = 
   \calC_{1,1}^{\rm tan}
   + \sum_{m =2}^\infty 
       \calC_{m,m}^{\rm tan}
   + \sum_{m =2}^\infty \sum_{j=1}^{m-1} 
       \calC_{m,j}^{\rm tan}
   ,
   \label{EQ56}
  \end{equation}
where, for the first two terms,
  \begin{equation}
   \calC_{1,1}^{\rm tan}
   = 
   \sum_{|\alpha| = 1} \sum_{|\beta| = 1, \beta \leq \alpha} {\alpha \choose \beta}
     M_\alpha \|J^\beta u_h J^{\alpha - \beta} \bar{\partial} u\|_{H^3}
   \les
   \sum_{|\alpha| = 1} \sum_{|\beta| = 1}
     \|J^\beta u\|_{H^3}\|J^{\alpha - \beta} \bar{\partial}u\|_{H^3}
   \les 
   \|u\|_{H^4}^2
   ,
   \llabel{EQ57}
  \end{equation}
and, denoting $\mathcal{A}_{\alpha,\beta,s} = M_\alpha M_\beta^{-1} {\alpha \choose \beta} (m-1)^{-1}$,
  \begin{align}
   \begin{split}
    \sum_{m =2}^\infty \calC_{m,m}^{\rm tan}
    &=
    \sum_{m =2}^\infty \sum_{|\alpha| = m} \sum_{|\beta| = m} {\alpha \choose \beta}
      M_\alpha \|J^\beta u_h J^{\alpha - \beta} \bar{\partial} u\|_{H^3}
      \frac{\tau^{m-1}}{(m-1)!}
    \\&\les
    \tau \|u\|_{H^4} 
    \sum_{m= 2}^\infty \sum_{|\alpha| = m} \sum_{|\beta| = m}
      \left( M_\beta \|J^\beta u\|_{H^3} \frac{(m-1)\tau^{m-2}}{(m-1)!}\right)
      \mathcal{A}_{\alpha,\beta,s}
    \\&\les
    \tau \|u\|_{H^4}\|u\|_{Y_\tau^\co}
    .
   \end{split}
   \label{EQ58}
  \end{align}
In order to justify the last inequality in~\eqref{EQ58}, we must show that $\mathcal{A}_{\alpha,\beta,s}$ is uniformly bounded in $m$ for $m \geq 2$.  
Indeed, since $|\alpha| = |\beta| = m$ and $\beta \leq \alpha$ imply $\alpha = \beta$, we get 
  \begin{equation}
   \mathcal{A}_{\alpha,\beta,s} 
   = 
   \frac{1}{m-1}
   \leq
   1
   \llabel{EQ59}
  \end{equation}
for $m \geq 2$.
Finally, we consider the third term on the right-hand side of~\eqref{EQ56}. 
We have
  \begin{align}
   \begin{split}
    \sum_{m =2}^\infty \sum_{j=1}^{m-1} 
      \calC_{m,j}^{\rm tan}
    &=
    \sum_{m= 2}^\infty \sum_{|\alpha| = m} \sum_{j=1}^{m-1} \sum_{|\beta| = j, \beta \leq \alpha} {\alpha \choose \beta}
      M_\alpha \|J^\beta u_h J^{\alpha - \beta} \bar{\partial} u\|_{H^3}
      \frac{\tau^{m-1}}{(m-1)!}
    \\&\les
    \tau 
    \sum_{m=2}^\infty \sum_{|\alpha| = m} \sum_{j=1}^{m-1} \sum_{|\beta| = j, \beta \leq \alpha}
      \left(M_\beta \|J^\beta u\|_{H^3} \frac{\tau^{j-1}}{(j-1)!}\right)
    \\&\indeq\indeq\indeq\indeq
      \times 
      \left( M_{\alpha - \beta} \|J^{\alpha - \beta} \bar{\partial} u\|_{H^3} \frac{(m-j) \tau^{m-j-1}}{(m-j)!}\right)
      \mathcal{A}_{\alpha,\beta,s}
    \\&\les
    \tau \|u\|_{X_\tau^\co} \|u\|_{Y_\tau^\co}
    .
   \end{split}
   \llabel{EQ60}
  \end{align}
In the last line of the above, we used Lemma~\ref{L05a} to obtain
  \begin{align}
   \begin{split}
    \mathcal{A}_{\alpha, \beta, s}
    &=
    M_\alpha M_\beta^{-1}M_{\alpha-\beta}^{-1} {\alpha \choose \beta} \frac{(j-1)! (m-j)!}{(m-1)!} \frac{1}{m-j}
  \\&\leq
    {m \choose j} {m-1\choose j-1}^{-s} \frac{1}{m-j}
    = 
    {m-1 \choose j-1}^{1-s} \frac{m}{j(m-j)} 
    \les 1
   \end{split}
   \llabel{EQ61}
  \end{align}
for $m \geq 2$ and $1\leq j \leq m-1$. 
Also, we used Lemma~\ref{L05} to change the mixed sum to the product $\|u\|_{X_\tau^\co} \|u\|_{Y_\tau^\co}$.
Throughout the paper, Lemma~\ref{L05} is frequently used in precisely this manner, 
and we henceforth omit explicit reference to it. 
This completes the proof of~\eqref{EQ55a}.

Next, we claim that
  \begin{equation}
   \sum_{m=1}^\infty \sum_{j=1}^m 
     \calC_{m,j}^{\nor}
   \les
   \|u\|_{H^4}^2 
   + (\tau + \tau^2)\|u\|_{Y_\tau^\co}
     (\|u\|_{H^4} + \|u\|_{X_\tau^\co})
   .
   \llabel{EQ61a}
  \end{equation}
Indeed,
  \begin{equation}
   \sum_{m =1}^\infty \sum_{j=1}^m 
     \calC_{m,j}^{\rm nor}
   =
   \calC_{1,1}^{\rm nor}
   + \sum_{m =1}^\infty
       \calC_{m,m}^{\rm \nor}
   + \sum_{m = 2}^\infty \sum_{j=1}^{m-1} 
       \calC_{m,j}^{\rm nor}
   ,
   \label{EQ62}
  \end{equation}
where, similarly to the tangential estimates above,
  \begin{equation}
   \calC_{1,1}^{\rm nor} 
   + \sum_{m =1}^\infty 
       \calC_{m,m}^{\rm nor} 
   \les 
   \|u\|_{H^4}^2
   + \tau \|u\|_{H^4} \|u\|_{Y_\tau^\co}
   .
   \label{EQ63}
  \end{equation}
For the third term on the right-hand side of~\eqref{EQ62}, Lemma~\ref{L06} gives
  \begin{equation}
   \sum_{m =2}^\infty \sum_{j=1}^{m-1} 
     \calC_{m,j}^{\rm nor}
   \les
   \sum_{m=2}^\infty \sum_{|\alpha| = m} \sum_{j=1}^{m-1} \sum_{|\beta| =j,\beta \leq \alpha} 
     \sum_{i=1}^4 {\alpha \choose \beta}
     M_\alpha E_{\alpha,\beta}^i 
     \frac{\tau^{m-1}}{(m-1)!}
   .
   \llabel{EQ64}
  \end{equation}
For $E_{\alpha,\beta}^1$, 
we obtain
  \begin{align}
   \begin{split}
    &
    \sum_{m =2}^\infty \sum_{|\alpha| = m} \sum_{j=1}^{m-1} \sum_{|\beta| =j} {\alpha \choose \beta}
      M_\alpha 1_{\{\beta_3 = 0, \alpha_3 \geq 1, \beta' < \alpha'\}}
      \|J^\beta \bar{\partial} u\|_{H^3} \|J^{\alpha - \beta + e_3 - \bar{e}} u\|_{H^3}
      \frac{\tau^{m-1}}{(m-1)!}
    \\&\les
    \tau
    \sum_{m=2}^\infty \sum_{|\alpha| = m} \sum_{j=1}^{m-1} \sum_{|\beta| =j}
      \left( M_\beta \|J^\beta \bar{\partial} u\|_{H^3}\frac{j \tau^{j-1}}{j!}\right)
    \\&\indeq\indeq\indeq\indeq
      \times \left( M_{\alpha - \beta} \|J^{\alpha - \beta + e_3 - \bar{e}} u\|_{H^3} \frac{\tau^{m-j-1}}{(m-j-1)!}\right)
      \mathcal{A}_{\alpha, \beta, s}
      1_{\{\beta_3 = 0, \alpha_3 \geq 1\}}
    \\&\les
    \tau \|u\|_{X_\tau^\co} \|u\|_{Y_\tau^\co}
    ,
   \end{split}
   \llabel{EQ65}
  \end{align}
where we note that, by Lemma~\ref{L05a}, 
  \begin{equation}
   M_\alpha M_\beta^{-1} M_{\alpha-\beta}^{-1} {\alpha \choose \beta}
   1_{\{\beta_3 =0, \alpha_3 \geq 1\}}
   \leq
   {m- \alpha_3 \choose j}
   ,
   \llabel{EQ66}
  \end{equation}
and hence,
for $m\geq  2$ and $1\leq j \leq m-1$,
  \begin{align}
   \begin{split}
   \mathcal{A}_{\alpha,\beta,s} 
   1_{\{\beta_3 = 0, \alpha_3 \geq 1\}}
   &= 
   M_\alpha M_\beta^{-1} M_{\alpha- \beta}^{-1} {\alpha \choose \beta}
   \frac{j! (m-j-1)!}{(m-1)!} 
   \frac{1}{j}
   1_{\{\beta_3 = 0, \alpha_3 \geq 1\}}
   \\&\leq
   {m-\alpha_3\choose j} {m-1 \choose j}^{-s} \frac{1}{j}
   1_{\{\alpha_3 \geq 1\}}
   \leq 
   {m-1\choose j}^{1-s} \frac{1}{j}
   \leq 1
   .
   \end{split}
   \llabel{EQ66b}
  \end{align}

Next, we consider
  \begin{equation}
   E_{\alpha,\beta}^2
   =
   1_{\{\beta_3 = 0, \alpha_3 \geq 1, \beta' = \alpha'\}}
     \alpha_3
     (\|J^\beta u\|_{H^3} + \|J^\beta \bar{\partial} u\|_{H^3}) 
     \|J^{\alpha- \beta} u\|_{H^3} 
   ,  
   \label{EQ66a}
  \end{equation}
and begin by bounding the sum on the highest order product in~\eqref{EQ66a}. 
We write 
  \begin{align}
   \begin{split}
    & 
    \sum_{m =2}^\infty \sum_{|\alpha| = m} \sum_{j=1}^{m-1} \sum_{|\beta| =j} {\alpha \choose \beta}
      M_\alpha 1_{\{\beta_3 = 0, \alpha_3 \geq 1, \beta' = \alpha'\}}
      \alpha_3
      \|J^\beta \bar{\partial} u\|_{H^3} \|J^{\alpha - \beta} u\|_{H^3}
      \frac{\tau^{m-1}}{(m-1)!}
    \\&\indeq\indeq\les
    \tau \|u\|_{H^4}
    \sum_{m=2}^\infty \sum_{|\alpha| = m} \sum_{|\beta| = m-1}
      \left( M_\beta \|J^\beta \bar{\partial} u\|_{H^3} \frac{(m-1)\tau^{m-1}}{(m-1)!}\right)
      \mathcal{A}_{\alpha,\beta,s}^1 1_{\{\beta_3 = 0, \alpha' = \beta'\}}
    \\&\indeq\indeq\indeq
    + \tau 
    \sum_{m = 3} \sum_{|\alpha| =m} \sum_{j=1}^{m-2} \sum_{|\beta| = j}
      \left( M_\beta \|J^\beta \bar{\partial} u\|_{H^3} \frac{\tau^j}{j!}\right)
    \\&\indeq\indeq\indeq\indeq\indeq\indeq
      \times \left( M_{\alpha-\beta} \|J^{\alpha - \beta} u\|_{H^3} \frac{(m-j-1)\tau^{m-j-2}}{(m-j-1)!}\right)
      \mathcal{A}_{\alpha, \beta,s}^21_{\{\beta_3 = 0, \alpha' = \beta'\}}
    \\&\indeq\indeq\les
    \tau\|u\|_{H^4} \|u\|_{Y_\tau^\co}
    + \tau \|u\|_{X_\tau^\co} \|u\|_{Y_\tau^\co}
    .
   \end{split}
   \llabel{EQ67}
  \end{align}
Note that $\alpha_3 = m-j$ when $\beta_3 = 0$ and $\beta' = \alpha'$.
Hence, for $m\geq 2$ and $|\beta| = m-1$,
  \begin{align}
   \begin{split}
   \mathcal{A}_{\alpha,\beta,s}^1 
   1_{\{\beta_3 = 0, \beta' = \alpha'\}}
   &=
   M_\alpha M_\beta^{-1}{\alpha \choose \beta}\frac{\alpha_3 }{m-1}1_{\{\beta_3 = 0, \beta' = \alpha'\}}
   \\&\leq
   {|\alpha'|\choose |\beta'|}{\alpha_3 \choose 0} 
   \frac{1}{m-1}
   1_{\{\beta' = \alpha'\}}
   = 
   \frac{1}{m-1}
   \leq 1
   ,
   \end{split}
   \llabel{EQ68}
  \end{align}
and, for $m\geq 3$ and $1\leq j \leq m-2$, 
  \begin{align}
   \begin{split}
    \mathcal{A}_{\alpha,\beta,s}^{2} 
    1_{\{\beta_3 = 0, \beta' = \alpha'\}}
    &=
    M_\alpha M_\beta^{-1} M_{\alpha_\beta}^{-1} {\alpha \choose \beta}
    \frac{j!(m-j-1)!}{(m-1)!} \frac{\alpha_3}{m-j-1}
    1_{\{\beta_3 = 0, \beta' = \alpha'\}}
    \\&\leq
    {|\alpha'|\choose |\beta'|} {\alpha_3 \choose 0}
    {m-1 \choose j}^{-s} \frac{m-j}{m-j-1}
    1_{\{\beta' = \alpha'\}}
    \les 1
    .
   \end{split}
   \llabel{EQ69}
  \end{align}
For the lower order term in $E_{\alpha,\beta}^2$, a similar computation yields
  \begin{align}
   \begin{split}
    &
    \sum_{m = 2}^\infty \sum_{|\alpha| = m} \sum_{j=1}^{m-1} \sum_{|\beta| = j} {\alpha \choose \beta}
      M_\alpha 1_{\{\beta_3 = 0, \alpha_3 \geq 1, \beta' = \alpha'\}}
      \alpha_3
      \|J^\beta u\|_{H^3}\|J^{\alpha - \beta} u\|_{H^3} 
      \frac{\tau^{m-1}}{(m-1)!}
    \\&\indeq\indeq\les
    \tau \|u\|_{H^4}^2
    + \tau^2 \|u\|_{H^4} \|u\|_{Y_\tau^\co}
    + \tau^2 \|u\|_{X_\tau^\co} \|u\|_{Y_\tau^\co}
    .
   \end{split}
   \llabel{EQ70}
  \end{align}
Next, we consider~$E_{\alpha,\beta}^3$. 
Due to Remark~\ref{R01}, it suffices to replace $E_{\alpha,\beta}^3$ with $\|J^\beta u\|_{H^3}\|J^{\alpha - \beta + e_3}u\|_{H^3}$.
Observe that
  \begin{align}
   \begin{split}
    &
    \sum_{m =2}^\infty \sum_{|\alpha| = m} \sum_{j=1}^{m-1} \sum_{|\beta| =j} {\alpha \choose \beta}
      M_\alpha \|J^\beta u\|_{H^3} \|J^{\alpha - \beta+e_3} u\|_{H^3} 
      \frac{\tau^{m-1}}{(m-1)!}
    \\&\indeq\indeq\les
    \tau
    \sum_{m=2}^\infty \sum_{|\alpha| = m} \sum_{j=1}^{m-1} \sum_{|\beta| =j} 
      \left( M_\beta \|J^\beta u\|_{H^3} \frac{\tau^{j-1}}{(j-1)!}\right)
    \\&\indeq\indeq\indeq\indeq\indeq\indeq
      \times \left(M_{\alpha - \beta}\|J^{\alpha - \beta + e_3}u\|_{H^3} \frac{(m-j) \tau^{m-j-1}}{(m-j)!}\right)
      \mathcal{A}_{\alpha,\beta,s}
    \\&\indeq\indeq\les
    \tau\|u\|_{X_\tau^\co} \|u\|_{Y_\tau^\co}
    ,
   \end{split}
   \llabel{EQ71}
  \end{align}
where, for $m\geq 2$ and $1\leq j \leq m-1$, 
  \begin{align}
   \begin{split}
    \mathcal{A}_{\alpha,\beta,s}
    &= 
    M_\alpha M_\beta^{-1} M_{\alpha-\beta}^{-1}{\alpha \choose \beta}
    \frac{(j-1)! (m-j)!}{(m-1)!} \frac{1}{m-j}
    \\&\leq
    {m \choose j} {m-1 \choose j-1}^{-s} \frac{1}{m-j}
    =
    {m-1\choose j-1}^{1-s} \frac{m}{j(m-j)}
    \leq 1
    .
   \end{split}
   \llabel{EQ72}
  \end{align}
Finally, for $E_{\alpha,\beta}^4$, it suffices to replace it with $1_{\{\beta_3 = 0, \beta' < \alpha'\}}\|J^\beta u\|_{H^3} \|J^{\alpha - \beta} u\|_{H^3}$. 
We obtain
  \begin{align}
   \begin{split}
    &
    \sum_{m =2}^\infty \sum_{|\alpha| = m}\sum_{j=1}^{m-1} \sum_{|\beta| = j}{\alpha \choose \beta}
      M_\alpha 1_{\{\beta_3 = 0, \beta' < \alpha'\}} \|J^\beta u\|_{H^3} \|J^{\alpha - \beta} u\|_{H^3} 
      \frac{\tau^{m-1}}{(m-1)!}
    \\&\indeq\indeq\les
    \tau \|u\|_{H^4}^2
    + \tau^2 \|u\|_{H^4} 
      \sum_{m= 3}^\infty \sum_{|\alpha| = m} \sum_{|\beta| = m-1} 
        \left( M_\beta \|J^\beta u\|_{H^3} \frac{(m-2) \tau^{m-3}}{(m-2)!} \right) 
      \mathcal{A}_{\alpha,\beta,s}^1 1_{\{\beta_3 = 0\}}
    \\&\indeq\indeq\indeq
    + \tau^2 
      \sum_{m= 3}^\infty \sum_{|\alpha| = m} \sum_{j=1}^{m-2} \sum_{|\beta| = j} 
        \left( M_\beta \|J^\beta u\|_{H^3} \frac{\tau^{j-1}}{(j-1)!} \right)
    \\&\indeq\indeq\indeq\indeq\indeq\indeq
    \times \left( M_{\alpha - \beta} \|J^{\alpha - \beta} u\|_{H^3} \frac{(m-j-1)\tau^{m-j-2}}{(m-j-1)!}\right)
    \mathcal{A}_{\alpha,\beta,s}^21_{\{\beta_3 = 0\}}
    \\&\indeq\indeq\les
    \tau \|u\|_{H^4}^2 
    + \tau^2\|u\|_{H^4} \|u\|_{Y_\tau^\co}
    + \tau^2 \|u\|_{X_\tau^\co}\|u\|_{Y_\tau^\co}
    .
   \end{split}
   \llabel{EQ73}
  \end{align}
Indeed, $m\geq 3$, 
  \begin{align}
   \begin{split}
   \mathcal{A}_{\alpha,\beta,s}^11_{\{\beta_3 = 0\}}
   &= 
   M_\alpha M_\beta^{-1} {\alpha \choose \beta} \frac{(m-2)!}{(m-1)!} \frac{1}{m-2}
   1_{\{\beta_3 = 0\}}
   \\&\les 
   M_{\alpha - \beta}{m \choose m-1} \frac{1}{(m-1)}\frac{1}{m-2}
   \les 1
   ,
   \end{split}
   \label{EQ74}
  \end{align}
and, for $m \geq 3$ and $1 \leq j \leq m-2$,
  \begin{align}
   \begin{split}
    \mathcal{A}_{\alpha,\beta,s}^2
    1_{\{\beta_3 = 0\}}
    &=
    M_\alpha M_\beta^{-1}M_{\alpha-\beta}^{-1} {\alpha \choose \beta}
    \frac{(j-1)! (m-j-1)!}{(m-1)!} \frac{1}{m-j-1}
    1_{\{\beta_3 = 0\}}
    \\&\leq
    {m\choose j} {m-1\choose j-1}^{-s} \frac{1}{(m-j) (m-j-1)}
    =
    {m-1\choose j-1}^{1-s} \frac{m}{j(m-j) (m-j-1)}
    \leq 1
    .
   \end{split}
   \label{EQ75}
  \end{align}
Note that in~\eqref{EQ74} and~\eqref{EQ75}, we did not need to utilize the restriction $\beta_3 = 0$.
This completes the proof. 
\end{proof}

Now we prove Lemma~\ref{L02}, which gives estimates on two lower-order terms associated with the advection term in~\eqref{EQ00}.

\begin{proof}[Proof of Lemma~\ref{L02}]
First, we establish~\eqref{EQ17}.  
Since $\phi[J^\alpha, \partial_3] u = -\alpha_3 \phi' J^\alpha u$, we get
  \begin{equation}
   \|u_3[J^\alpha,\partial_3]u\|_{H^3}
   \les
   \|\phi^{-1}u_3\|_{H^3} \|\phi [J^\alpha,\partial_3]u\|_{H^3}
   \les
   \alpha_3 \|u\|_{H^4} \|J^{\alpha} u\|_{H^3}
   .
   \llabel{EQ76}
  \end{equation}
Hence,
  \begin{align}
   \begin{split}
    &
    \sum_{m= 1}^\infty \sum_{|\alpha| = m}
      M_\alpha \|u_3 [J^\alpha,\partial_3]u\|_{H^3} 
      \frac{\tau^{m-1}}{(m-1)!}
    \\&\indeq\les
    \|u\|_{H^4}^2
    + \tau\|u\|_{H^4}
    \sum_{m= 2}^\infty \sum_{|\alpha| = m}
      \left(M_\alpha \|J^\alpha u\|_{H^3} \frac{(m-1)\tau^{m-2}}{(m-1)!}\right)
      \frac{\alpha_3}{m-1}
    \\&\indeq\les
    \|u\|_{H^4}^2
    + \tau \|u\|_{H^4} \|u\|_{Y_\tau^\co}
    ,
   \end{split}
   \llabel{EQ77}
  \end{align}
where we note that $\frac{\alpha_3}{m-1} \leq \frac{m}{m-1}\leq 1$ for  $m \geq 2$. 

Next, we prove~\eqref{EQ18}. 
We only consider the terms in the sum corresponding to $|\gamma| = 3$ as the terms for $|\gamma|=0,1,2$ give subordinate bounds.
From the multiplicative Sobolev inequality~\eqref{EQ28a}, we have
  \begin{align}
   \begin{split}
    &
    \sum_{m = 1}^\infty \sum_{|\alpha| = m} \sum_{j=1}^3 \sum_{|\mu| = j, |\gamma| = 3} {\gamma \choose \mu}
      M_\alpha \|\partial^\mu u \cdot \nabla \partial^{\gamma-\mu} J^\alpha u\|_{L^2}
      \frac{\tau^{m-1}}{(m-1)!}
    \\&\indeq\les
    \|u\|_{H^4}^2
    +\sum_{m= 2}^\infty \sum_{|\alpha| = m} \sum_{j=1}^3 \sum_{|\mu| = j, |\gamma| = 3}
      M_\alpha \|\partial^\mu u\|_{H^{4-j}} \|\nabla \partial^{\gamma - \mu} J^\alpha u\|_{H^{j-1}}
      \frac{\tau^{m-1}}{(m-1)!}
    \\&\indeq\les
    \|u\|_{H^4}^2
     + \tau \|u\|_{H^4}
     \sum_{m= 2}^\infty \sum_{|\alpha| = m} 
       \left(M_\alpha \|J^\alpha u\|_{H^3} \frac{(m-1)\tau^{m-2}}{(m-1)!}\right)
       \frac{1}{m-1}
     \\&\indeq\les
     \|u\|_{H^4}^2
     + \tau \|u\|_{H^4}\|u\|_{Y_\tau^\co}
     .
   \end{split}
   \llabel{EQ78}
  \end{align}
This completes the proof of the lemma. 
\end{proof}

\subsection{Pressure Estimates}\label{sec03b}
We begin with a uniform estimate on $\phi(z) = z(1+z)^{-1}$.

\begin{Lemma}
\label{L04}
We have 
  \begin{equation}
   \sup_{N \geq 1}\|\phi^N\|_{W^{3,\infty}(\mathbb{R_+})}
   \les 1
   .
   \label{EQ78b}
  \end{equation}
\end{Lemma}

\begin{proof}
Without loss of generality, suppose $N \geq 4$.
We compute
  \begin{equation}
   \left( \phi^N\right)^{(3)}
   =
   N(N-1)(N-2)\phi^{N-3}(\phi')^3
   + 3N(N-1)\phi^{N-2} \phi' \phi''
   + N \phi^{N-1} \phi^{(3)}
   .
   \label{EQ78c}
  \end{equation}
Also, note that
  \begin{equation}
   \phi^{(k)} 
   = 
   (-1)^{k+1} k! (1-\phi)^{k+1} 
   \llabel{EQ78d}
  \end{equation}
for $k \geq 1$.
Since $0 \leq \phi \leq 1$,~\eqref{EQ78c} then gives us
  \begin{equation}
   |\left(\phi^N\right)^{(3)}|
   \les 
   N^3\phi^{N-3}(1-\phi)^4
   \leq
   N^3 
   \left(\frac{N-3}{N+1}\right)^{N-3} 
   \left( \frac{4}{N+1}\right)^4
   \les
   N^{-1}
   ,
   \label{EQ78e}
  \end{equation}
where, in the second inequality, we optimized $x^{N-3}(1-x)^4$ for $x \in [0,1]$.
Indeed,~\eqref{EQ78e} is stronger than~\eqref{EQ78b}; 
however, uniform boundedness is sufficient for the purposes of this paper.
The estimates are similar for lower-order derivatives. 
This completes the proof. 
\end{proof}
The pressure $p$ satisfies the elliptic Neumann problem given by
  \begin{align}
   \begin{split}
    -\Delta p 
     &=
     \partial_i u_k \partial_k u_i
     \inon{in $\mathbb{R}_+^3$,}
     \\
     \partial_3 p
     &= 0
     \inon{on $\{x_3 = 0\}$.}
   \end{split}
   \llabel{EQ78a}
  \end{align}
The following lemma gives $\Hco^{m,3}$ elliptic estimates where the dependence on $m$ in the constants is explicit.

\begin{Lemma}
\label{L07}
Assume that $p$ is a smooth solution to 
  \begin{align}
   \begin{split}
    -\Delta p &= v 
    \inon{in $\mathbb{R}_+^3$,}
    \\
    \partial_3 p &= 0
    \inon{on $\{x_3 = 0\}$.}
   \end{split}
   \label{EQ79}
  \end{align}
Then, for any $\alpha \in \mathbb{N}_0^3$ with $|\alpha| = m$ and $\alpha_3 \neq 0$, 
  \begin{equation}
   \|J^\alpha \partial_3 p\|_{H^3} 
   \les
   \sum_{\substack{s,t \in \mathbb{N}_0, |\beta| = m-1 \\ \beta' - \alpha' = (2s,2t)}} {s+t \choose s}
     \|J^\beta v\|_{H^3}
   .
  \label{EQ80}
  \end{equation}
Moreover, if $\alpha_3 \geq 2$, then there exists a universal constant $C > 0$ such that
  \begin{align}
   \begin{split}
    \|J^\alpha \partial_1 p\|_{H^3} 
    &\les
    \sum_{\substack{s,t \in \mathbb{N}_0, |\beta| = m-1 \\ \beta' - \alpha' = (2s+1,2t)}} {s+t \choose s}
      \|J^\beta v\|_{H^3}
      ,
    \\
    \|J^\alpha \partial_2 p\|_{H^3} 
    &\les
    \sum_{\substack{s,t \in \mathbb{N}_0, |\beta| = m-1 \\ \beta' - \alpha' = (2s, 2t+1)}} {s+t \choose s}
      \|J^\beta v\|_{H^3}
    .
   \end{split}
   \label{EQ81}
  \end{align}
Finally, if $\alpha_3 = 1$, we have
  \begin{align}
   \begin{split}
    \|J^\alpha \partial_1 p\|_{H^3}
    &\les
    \|\partial^{\alpha'} v\|_{H^3}
    ,
    \\
    \|J^\alpha \partial_2 p\|_{H^3}
    &\les
    \|\partial^{\alpha'} v\|_{H^3}
    .
   \end{split}
   \label{EQ82}
  \end{align}
\end{Lemma}

\begin{proof}
Denote $\Delta' = \partial_{11} + \partial_{22}$. 
From~\eqref{EQ79}$_1$, we get
  \begin{equation}
   \partial_{3}^2 p
   =
   -\Delta' p
   .
   \llabel{EQ83}
  \end{equation}
By induction, 
we have
  \begin{equation}
   \partial_3^{2k+2} p
   =
   (-\Delta')^{k+1}p
   -\sum_{j=0}^k \partial_3^{2j} (-\Delta')^{k-j} v
   \label{EQ84}
  \end{equation}
for $k \geq 0$.
Let $|\alpha| = m$ with $\alpha_3 = 2k+1$.
Multiplying both sides of~\eqref{EQ84} by $\phi^{2k+1}$ and applying $\partial^{\alpha'}$
gives us 
  \begin{equation}
   J^\alpha \partial_3 p
   =
   \partial^{\alpha'}Z_{2k+1} \partial_3 p
   = 
   \phi^{2k+1} (-\Delta')^{k+1}\partial^{\alpha'}p
   -\sum_{j=0}^k \phi^{2(k-j) + 1} Z_{2j} (-\Delta')^{k-j}\partial^{\alpha'} v
   .
   \label{EQ85}
  \end{equation}
If $\alpha_3 = 2k+2$,  then
applying $\partial_3$ to~\eqref{EQ84} and multiplying both sides of the resulting equation by $\phi^{2k+2}$
yields
  \begin{equation}
   J^\alpha \partial_3 p
   =
   \partial^{\alpha'}Z_{2k+2} \partial_3 p
   =
   \phi^{2k+1} Z (-\Delta')^{k+1} \partial^{\alpha'}p
   - \sum_{j=0}^k \phi^{2(k-j) +1} Z_{2j+1}(-\Delta')^{k-j}\partial^{\alpha'} v
   .
   \label{EQ86}
  \end{equation}

Next, we estimate the first term on the right-hand sides of~\eqref{EQ85} and~\eqref{EQ86}.
Note that $g = (-\Delta')^k \partial^{\alpha'}p$ contains only tangential derivatives, and hence satisfies
  \begin{align}
   \begin{split}
    -\Delta g &= (-\Delta')^k \partial^{\alpha'} v
    \inon{in $\mathbb{R}_+^3$,}
    \\
    \partial_3 g &= 0
    \inon{on $\{x_3 = 0\}$.}
   \end{split}
   \llabel{EQ87}
  \end{align}    
Thus, elliptic regularity gives 
  \begin{equation}
   \|\Delta' g\|_{H^3}
   \leq
   \|g\|_{H^5}
   \les
   \|(-\Delta')^k \partial^{\alpha'} v\|_{H^3}
   \llabel{EQ88}
  \end{equation}
and 
  \begin{equation}
   \|Z \Delta' g\|_{H^3}
   \leq
   \|Zg\|_{H^5}
   \les
   \|Z (-\Delta')^k\partial^{\alpha'} v\|_{H^3}
   .
   \label{EQ89}
  \end{equation}
Note that for~\eqref{EQ89}, we used $\partial_3(Zg) = 0$ on $\{x_3= 0\}$.

Hence, for $\alpha_3 = 2k+1 \geq 1$,~\eqref{EQ85} 
implies
  \begin{align}
   \begin{split}
    \|J^\alpha \partial_3 p\|_{H^3}
    &\leq
    \|\phi^{2k+1}(-\Delta')^{k+1} \partial^{\alpha'} p\|_{H^3}
    + \sum_{j=0}^k \|\phi^{2(k-j)+1} Z_{2j} (-\Delta')^{k-j} \partial^{\alpha'} v\|_{H^3}
    \\&\les
    \|(-\Delta')^{k+1} \partial^{\alpha'} p\|_{H^3}
    + \sum_{j=0}^k \|Z_{2j} (-\Delta')^{k-j}\partial^{\alpha'} f\|_{H^3}
    \\&\les
    \sum_{j=0}^k \sum_{s=0}^{k-j} {k-j \choose s}
      \|\partial_1^{2s + \alpha_1} \partial_2^{2k -2j - 2s +\alpha_2} Z_{2j} v\|_{H^3}
    ,
   \end{split}
   \label{EQ90}
  \end{align}
where we used Lemma~\ref{L04} in the second line and the identity,
  \begin{equation}
   (\Delta')^m w 
   =
   \sum_{s=0}^m {m\choose s}
     \partial_1^{2s} \partial_2^{2m-2s} w
   ,
   \llabel{EQ91}
  \end{equation}
in the last line. 
If $\alpha_3 = 2k+2 \geq 2$, we similarly get 
  \begin{equation}
   \|J^\alpha \partial_3 p\|_{H^3}
   \les
   \sum_{j=0}^k \sum_{s=0}^{k-j}{ k-j \choose s}
     \|\partial_1^{2s+\alpha_1} \partial_2^{2k-2j-2s + \alpha_2} \partial_3^{2j+1} v\|_{H^3}
   .
   \label{EQ92}
  \end{equation}

For brevity, we re-index the right-hand side of~\eqref{EQ90} with
$t = k-j-s \geq 0$ and $\beta = (2s+ \alpha_1, 2k-2j -2s + \alpha_2, 2j) = (2s + \alpha_1, 2t + \alpha_2 , \alpha_3 - 1 -2s -2t) \in \mathbb{N}_0^3$, from which we are left with~\eqref{EQ80} for $\alpha_3 =2k+1 \geq 1$. 
Similarly, from~\eqref{EQ92}, re-indexing with  $t = k-j-s$ and $\beta = (2s+\alpha_1, 2k-2j-2s + \alpha_2, 2j+1)$ gives~\eqref{EQ80} for $\alpha_3 = 2k+2 \geq 2$.

Next, the estimates in~\eqref{EQ81} follow similarly after multiplying~\eqref{EQ85} and~\eqref{EQ86} by $\phi$ and applying $\partial_1$ or~$\partial_2$. 

Finally, to see~\eqref{EQ82}, note that $g = \partial^{\alpha'} p$ satisfies $-\Delta g = \partial^{\alpha'} v$ in $\mathbb{R}_+^3$ with $\partial_3 g = 0$ on $\{x_3 =0\}$.
Hence, for $\alpha_3 = 1$, elliptic regularity implies
  \begin{equation}
   \|J^\alpha \partial_1 p\|_{H^3}
   =
   \|\phi \partial_3\partial_1 \partial^{\alpha'} p\|_{H^3}
   =
   \|\phi \partial_3 \partial_1 g\|_{H^3}
   \les
   \|\partial^{\alpha'} v\|_{H^3}
   .
   \llabel{EQ93}
  \end{equation}
The same argument gives $\|J^\alpha \partial_2 p\|_{H^3} \les \|\partial^{\alpha'} v\|_{H^3}$.
This completes the proof of the lemma. 
\end{proof}

The next lemma is proven in~\cite{KV2} for $m \geq 3$.
In fact, it is true for $m \geq 1$, which can be seen by direct computation.
We omit the proof. 

\begin{Lemma}
\label{L08}
We have 
  \begin{equation}
   \sum_{s=0}^{[\beta_1/2]} \sum_{t=0}^{[\beta_2/2]}
     {\beta_1 + \beta_2 - 2s - 2t \choose \beta_1 - 2s} {s+t \choose s}
   \les
   m {\beta_1 + \beta_2 \choose \beta_1}
   \llabel{EQ94}
  \end{equation}
for any $m \geq 1$ and any $\beta = (\beta_1,\beta_2, m-1 - \beta_1 -\beta_2) \in \mathbb{N}_0^3$. 
Additionally, if $\beta_1 \geq 1$, we have
  \begin{equation}
   \sum_{s=0}^{[(\beta_1-  1)/2]} \sum_{t=0}^{[\beta_2/2]}
     {\beta_1 + \beta_2 - 2s -1 -2t \choose \beta_1 - 2s -1} {s+t \choose s}
   \les
   m {\beta_1 + \beta_2 \choose \beta_1}
   ,
   \llabel{EQ95}
  \end{equation}
while if $\beta_2 \geq 1$, we have
  \begin{equation}
   \sum_{s=0}^{[\beta_1/2]}\sum_{t=0}^{[(\beta_2 - 1)/2]}
     {\beta_1 + \beta_2 - 2s - 2t - 1 \choose \beta_1 - 2s}{s + t \choose s}
   \les
   m {\beta_1 + \beta_2 \choose \beta_1}
   .
   \llabel{EQ96}
  \end{equation}
\end{Lemma}
The following conormalization lemma helps control the term $J^\gamma \partial_i u_3 J^{\beta - \gamma}\partial_3 u_i$ for $i=1,2$;
it is essential for the proof of Lemma~\ref{L03}. 
The proof, which we omit, is almost identical to that of Lemma~\ref{L06}. 

\begin{Lemma}
\label{L09}
Let $0 \leq \gamma < \beta$. 
Then 
  \begin{equation}
   \|J^\gamma \bar{\partial} u_3 J^{\beta - \gamma} \partial_3 u_h\|_{H^3}
   \les
   E_{\beta,\gamma}^1
   + E_{\beta,\gamma}^2
   + E_{\beta, \gamma}^3
   + E_{\beta, \gamma}^4
   ,
   \llabel{EQ97}
  \end{equation}
where
  \begin{align}
   \begin{split}
    E_{\beta,\gamma}^1
    &=
    1_{\{\gamma_3 = 0, \beta_3 \geq 1, \gamma' < \beta'\}}
      \|J^\gamma \bar{\partial}^2 u\|_{H^3} \|J^{\beta -\gamma + e_3 - \bar{e}}u\|_{H^3},
    \\
    E_{\beta,\gamma}^{2}
    &= 
    1_{\{\gamma_3 = 0, \beta_3 \geq 1, \gamma' = \beta'\}}
      \beta_3
      (\|J^\gamma \bar{\partial}u\|_{H^3} + \|J^\gamma \bar{\partial}^2 u\|_{H^3}) \|J^{\beta- \gamma} u\|_{H^3} ,
    \\
    E_{\beta,\gamma}^3
    &=
    1_{\{\gamma_3 \geq 1\}}
      \|J^{\gamma - e_3} \bar{\partial}^2 u\|_{H^3} \|J^{\beta - \gamma + e_3} u\|_{H^3}
    + 1_{\{\gamma_3 = 0\}} 
      \|J^\gamma \bar{\partial} u\|_{H^3} \|J^{\beta- \gamma + e_3} u\|_{H^3},
    \\
    E_{\beta,\gamma}^4
    &=
    1_{\{\gamma_3 = 0, \beta_3 \geq 1, \gamma' < \beta'\}}
    \|J^\gamma \bar{\partial} u\|_{H^3} \|J^{\beta - \gamma + e_3 -\bar{e}} u\|_{H^3}
      + 1_{\{\gamma_3 = 0, \beta_3 = 0\}}
      \|J^\gamma \bar{\partial} u\|_{H^3} \|J^{\beta- \gamma} u\|_{H^3}
    .
   \end{split}
   \llabel{EQ98}
  \end{align}
\end{Lemma}
\begin{remark}
\label{R02}
The casework above is exhaustive and is the same as~\eqref{EQ33} after relabeling.
The terms $E_{\beta,\gamma}^i$ are organized according to energy level.
With regard to the estimates in the proof of Lemma~\ref{L03}, we can identify
  \begin{equation}
   E_{\beta,\gamma}^3 
   \sim
   \|J^\beta \bar{\partial} u\|_{H^3}\|J^{\alpha - \beta + e_3} u\|_{H^3}
   \llabel{EQ99}
  \end{equation}
and 
  \begin{equation}
   E_{\beta,\gamma}^4
   \sim 
   1_{\{\gamma_3 = 0, \gamma' < \alpha'\}}
     \|J^\beta \bar{\partial} u\|_{H^3} \|J^{\alpha- \beta} u\|_{H^3}
   .
   \llabel{EQ100}
  \end{equation}
We note that $E_{\beta,\gamma}^1$ and $E_{\beta,\gamma}^2$ are special in that the restriction on $\alpha,\beta$ plays a crucial role in the closure of the pressure estimate. 
\end{remark}

We are now ready to prove Lemma~\ref{L03}.

\begin{proof}[Proof of Lemma~\ref{L03}]
We claim that
  \begin{equation}
   \sum_{m= 1}^\infty \sum_{|\alpha| = m, \alpha_3 \neq 0}
     M_\alpha \|J^\alpha\partial_\ell p\|_{H^3}
     \frac{\tau^{m-1}}{(m-1)!}
   \les
   \sum_{m = 1}^\infty \sum_{|\beta| = m-1}
     M_\beta \|J^\beta(\partial_i u_k \partial_k u_i)\|_{H^3}
     \frac{m\tau^{m-1}}{(m-1)!}
   ,
   \quad
   \ell = 1,2,3
   .
   \label{EQ101}
  \end{equation}
For $\ell = 3$, we obtain
  \begin{align}
   \begin{split}
    &
    \sum_{|\alpha| = m, \alpha_3 \neq 0}
      M_\alpha \|J^\alpha \partial_3 p\|_{H^3}
    \\&\indeq\indeq\les
    \sum_{|\alpha| = m, \alpha_3 \neq 0} \sum_{\substack{s,t\in \mathbb{N}_0, |\beta| = m-1 \\\beta' -\alpha' =(2s,2t)}}
      M_\alpha {s+t \choose s} \|J^\beta(\partial_i u_k \partial_k u_i)\|_{H^3}
    \\&\indeq\indeq\les
    \sum_{|\beta| = m-1} \sum_{s=0}^{[\beta_1/2]} \sum_{t=0}^{[\beta_2/2]}
     {\beta_1 + \beta_2 -2s -2t \choose \beta_1 - 2s} {s+t \choose s}
     \|J^\beta(\partial_i u_k \partial_k u_i)\|_{H^3}
    \\&\indeq\indeq\les
    m
    \sum_{|\beta| = m-1}
      M_\beta \|J^\beta(\partial_i u_k \partial_k u_i)\|_{H^3}
    ,
   \end{split}
   \label{EQ102}
  \end{align}
where, in the first and third lines, we used Lemmas~\ref{L07} and~\ref{L08}, respectively.
In the second line, note that $\beta' -\alpha' = (2s,2t)$ implies
  \begin{equation}
   M_\alpha
   =
   {\alpha_1 + \alpha_2 \choose \alpha_1}
   =
   {\beta_1 + \beta_2 - 2s - 2t \choose \beta_1 - 2s}
   .
   \llabel{EQ103}
  \end{equation}

Next, we consider $\ell = 1$ in~\eqref{EQ101}. We write
  \begin{equation}
   \sum_{|\alpha| = m, \alpha_3 \neq 0}
     M_\alpha \|J^\alpha \partial_1 p\|_{H^3}
   =
   \sum_{|\alpha| = m, \alpha_3 = 1}
     M_\alpha \|J^\alpha \partial_1 p\|_{H^3}
   + \sum_{|\alpha| = m, \alpha_3 \geq 2 }
       M_\alpha \|J^\alpha \partial_1 p\|_{H^3}  
   ,
   \llabel{EQ104}
  \end{equation}
where, by~\eqref{EQ82} in Lemma~\ref{L07},
  \begin{equation}
   \sum_{|\alpha| =m, \alpha_3 =1}
     M_\alpha\|J^\alpha \partial_1 p\|_{H^3}
   \les
   \sum_{|\alpha| =m,\alpha_3 = 1}
   M_\alpha \|\partial^{\alpha'}( \partial_i u_k \partial_k u_i) \|_{H^3}
   =
   \sum_{|\beta| = m-1, \beta_3 =0}
   M_\beta \|J^\beta(\partial_i u_k \partial_k u_i)\|_{H^3}
   ,
   \label{EQ105}
  \end{equation}
and, by Lemmas~\ref{L07} and~\ref{L08},
  \begin{align}
   \begin{split}
    &
    \sum_{|\alpha| = m, \alpha_3 \geq 2}
      M_\alpha\|J^\alpha \partial_1 p\|_{H^3}
    \\&\indeq\indeq\les
    \sum_{|\alpha| = m, \alpha_3 \geq 2}
    \sum_{\substack{s,t \in \mathbb{N}_0, |\beta| = m-1 \\ \beta' - \alpha' = (2s+1,2t)}}
      M_\alpha {s+t \choose s} \|J^\beta(\partial_i u_k \partial_k u_i) \|_{H^3}
    \\&\indeq\indeq=
    \sum_{|\beta| = m-1, \beta_1 \geq 1}\sum_{s=0}^{[(\beta_1-1)/2]} \sum_{t=0}^{[\beta_2/2]}
      {\beta_1 - 1 + \beta_2 -2s-2t \choose \beta_1 -2s - 1}{s+t \choose s}
      \|J^\beta (\partial_i u_k \partial_k u_i)\|_{H^3}
    \\&\indeq\indeq\les
    m 
    \sum_{|\beta| = m-1, \beta_1 \geq 1}
      M_\beta \|J^\beta(\partial_i u_k \partial_k u_i)\|_{H^3}
    .
   \end{split}
   \label{EQ106}
  \end{align}
Upon combining~\eqref{EQ105} and~\eqref{EQ106},
we infer
  \begin{equation}
   \sum_{|\alpha| = m, \alpha_3 \neq 0}
     M_\alpha \|J^\alpha \partial_1 p\|_{H^3}
   \les
   m
   \sum_{|\beta| =m-1} 
     M_\beta \|J^\beta(\partial_i u_k \partial_k u_i)\|_{H^3}
   .
   \label{EQ107}
  \end{equation}
The same bound holds for $\ell = 2$. 
We then multiply~\eqref{EQ102} and~\eqref{EQ107} by $\tau^{m-1}(m-1)!^{-s}$ and sum over $m\geq 1$ to obtain~\eqref{EQ101}, completing the proof of our claim.

Next, using the Leibniz rule on~\eqref{EQ101} yields 
  \begin{align}
    \begin{split}
    &
   \sum_{m= 1}^\infty \sum_{|\alpha| = m, \alpha_3 \neq 0}
     M_\alpha \|J^\alpha \nabla p\|_{H^3}
     \frac{\tau^{m-1}}{(m-1)!}
   \\&\indeq
   \les
   \sum_{m = 1}^\infty \sum_{|\beta| = m-1} \sum_{j=0}^{m-1} \sum_{|\gamma| = j}{\beta \choose \gamma}
     M_\beta \|J^\gamma \partial_i u_k J^{\beta-\gamma} \partial_k u_i\|_{H^3}
     \frac{m\tau^{m-1}}{(m-1)!}
   .
  \end{split}
   \llabel{EQ108}
  \end{align}
We split $\|J^\gamma \partial_i u_k J^{\beta - \gamma} \partial_k u_i\|_{H^3}$ into tangential and normal parts. 
We write
  \begin{align}
   \begin{split}
    &
    \|J^\gamma \partial_i u_k J^{\beta - \gamma} \partial_k u_i\|_{H^3}
    \\&\indeq
    \leq
    \sum_{i,k=1,2} \|J^\gamma\partial_i u_k J^{\beta - \gamma} \partial_k u_i\|_{H^3}
    + \|J^\gamma \partial_3 u_3 J^{\beta - \gamma} \partial_3 u_3\|_{H^3}
    \\&\indeq\indeq
    + \|J^\gamma \bar{\partial} u_3 J^{\beta - \gamma} \partial_3 u_h\|_{H^3}
    + \|J^\gamma \partial_3 u_h J^{\beta - \gamma} \bar{\partial} u_3\|_{H^3}
    \\&\indeq
    \leq
    \|J^\gamma \bar{\partial} u_h J^{\beta - \gamma} \bar{\partial} u_h\|_{H^3}
    + \|J^\gamma \bar{\partial} u_3 J^{\beta - \gamma} \partial_3 u_h\|_{H^3}
    + \|J^\gamma \partial_3 u_h J^{\beta - \gamma} \bar{\partial} u_3\|_{H^3}
    ,
   \end{split}
   \llabel{EQ109}
  \end{align}
where we used the incompressibility condition in the last inequality. 
Our estimates are then organized through the following bound:
  \begin{equation}
   \sum_{m = 1}^\infty \sum_{|\alpha| = m} 
     M_\alpha \|J^\alpha \nabla p\|_{H^3}
     \frac{\tau^{m-1}}{(m-1)!}
   \les
   \sum_{m = 1}^\infty \sum_{j=0}^{m-1}
     (
       \calP_{m,j}^{\rm tan}
       + \calP_{m,j}^{\nor, 1}
       + \calP_{m,j}^{\nor,2}
     )
   ,
   \label{EQ110}
  \end{equation}
where 
  \begin{align}
   \begin{split}
    \calP_{m,j}^{\rm tan}
    &=
    \sum_{|\beta| = m-1} \sum_{|\gamma| = j}{\beta \choose \gamma}
      M_\beta \|J^\gamma \bar{\partial} u_h J^{\beta - \gamma} \bar{\partial} u_h\|_{H^3}
      \frac{m\tau^{m-1}}{(m-1)!}
    ,
    \\
    \calP_{m,j}^{\nor,1}
    &=
    \sum_{|\beta| = m-1} \sum_{|\gamma| = j}{\beta \choose \gamma}
      M_\beta \|J^\gamma \bar{\partial} u_3 J^{\beta - \gamma } \partial_3 u_h\|_{H^3} 
      \frac{m\tau^{m-1}}{(m-1)!}
    ,
    \\
    \calP_{m,j}^{\nor,2}
    &=
    \sum_{|\beta| = m-1} \sum_{|\gamma| = j} {\beta \choose \gamma}
      M_\beta \|J^\gamma \partial_3 u_h J^{\beta - \gamma} \bar{\partial} u_3\|_{H^3}
      \frac{m\tau^{m-1}}{(m-1)!}
    .
   \end{split}
   \llabel{EQ111}
  \end{align}
First, we show that 
  \begin{equation}
   \sum_{m = 1}^\infty \sum_{j=0}^{m-1}
     \calP_{m,j}^{\rm tan}
   \les
   \|u\|_{H^4}^2
   +\tau\|u\|_{H^4} \|u\|_{Y_\tau^\co}
   +\tau\|u\|_{X_\tau^\co}\|u\|_{Y_\tau^\co}
   .
   \label{EQ112}
  \end{equation}
Indeed, we extract the lower-order terms to obtain
  \begin{equation}
   \sum_{m = 1}^\infty \sum_{j=0}^{m-1} 
     \calP_{m,j}^{\rm tan}
   \leq 
   \calP_{1,0}^{\rm tan} 
   + \sum_{m = 2}^\infty 
       \calP_{m,m-1}^{\rm tan}
   + \sum_{m = 2}^\infty \sum_{j=0}^{m-2}
       \calP_{m,j}^{\rm tan}
   ,
   \label{EQ113}
  \end{equation}
where $\calP_{1,0}^{\rm tan} \les \|u\|_{H^4}^2$, and
  \begin{align}
   \begin{split}
    \sum_{m = 2}^\infty \calP_{m,m-1}^{\rm tan}
    &\les
    \tau \|u\|_{H^4} 
    \sum_{m = 2}^\infty \sum_{|\beta| = m-1} \sum_{|\gamma| = m-1, \gamma \leq \beta}
      \Big( 
        M_\gamma \|J^\gamma \bar{\partial} u\|_{H^3} 
        \frac{(m-1)\tau^{m-2}}{(m-1)!}
      \Big)
      \frac{m}{m-1}
    \\&\les
    \tau \|u\|_{H^4} \|u\|_{Y_\tau^\co}
    .
   \end{split}
   \llabel{EQ114}
  \end{align}
For the third term on the right-hand side of~\eqref{EQ113}, 
we consider the cases when $j$ is small or large compared to~$m-2$.
For small $j$, we have
  \begin{align}
   \begin{split}
    \sum_{m = 2}^\infty \sum_{j=0}^{[(m-2)/2]} \calP_{m,j}
    &\les
    \tau
    \sum_{m = 2}^\infty \sum_{|\beta| = m-1} \sum_{j=0}^{[(m-2)/2]} \sum_{|\gamma| = j}
      \Big(
       M_\gamma \|J^\gamma \bar{\partial} u\|_{H^3} 
       \frac{\tau^j}{j!}
      \Big)
    \\&\indeq\indeq\indeq\indeq
    \times 
    \Big(
     M_{\beta-\gamma}\|J^{\beta-\gamma} \bar{\partial} u\|_{H^3}
     \frac{(m-j-1)\tau^{m-j-2}}{(m-j-1)!}
     \Big) 
     \mathcal{A}_{\alpha,\beta,s}
    \\&\les
    \tau \|u\|_{X_\tau^\co} \|u\|_{Y_\tau^\co}
    ,
   \end{split}
   \label{EQ115}
  \end{align}
where, for $m \geq 2$ and $0 \leq j \leq [(m-1)/2]$,
  \begin{equation}
   \mathcal{A}_{\alpha,\beta,s}
   =
   M_\beta M_{\gamma}^{-1} M_{\beta - \gamma}^{-1} {\beta \choose \gamma}
   \frac{j! (m-j-1)!}{(m-1)!}
   \frac{m}{m-j-1}
   \leq
   {m-1\choose j}^{1-s} \frac{m}{m-j-1}
   \les 1
   .
   \llabel{EQ116}
  \end{equation}
For large $j$, 
we have 
  \begin{align}
   \begin{split}
    \sum_{m \geq 2} \sum_{j=[(m-2)/2]+1}^{m-2} 
      \calP_{m,j}^{\rm tan}
    &\les
    \tau
    \sum_{m \geq 2} \sum_{|\beta| = m-1} \sum_{j=[(m-2)/2]+1}^{m-2} \sum_{|\gamma| = j} 
      \Big( 
       M_\gamma \|J^\gamma \bar{\partial}u\|_{H^3}
       \frac{j\tau^{j-1}}{j!}
       \Big)
    \\&\indeq\indeq\indeq\indeq
    \times
    \Big(
     M_{\beta - \gamma} \|J^{\beta - \gamma} \bar{\partial} u\|_{H^3}
     \frac{\tau^{m-j-1}}{(m-j-1)!}
     \Big)
     \mathcal{A}_{\alpha,\beta,s}
    \\&\les
    \tau \|u\|_{X_\tau^\co}\|u\|_{Y_\tau^\co}
   ,
   \end{split}
   \label{EQ117}
  \end{align}
where, for $m \geq 2$ and $[(m-2)/2] < j \leq m-2$, 
  \begin{equation}
   \mathcal{A}_{\alpha,\beta,s}
   =
   {m-1\choose j}^{1-s} \frac{m}{j} 
   \les 
   1
   .
   \llabel{EQ118}
  \end{equation}
This completes the proof of~\eqref{EQ112}.

Now, we establish
  \begin{equation}
   \sum_{m=1}^\infty \sum_{j=0}^{m-1} 
     \left( \calP_{m,j}^{\nor, 1} + \calP_{m,j}^{\nor,2}\right)
   \les
   \|u\|_{H^4}^2
   + (\tau + \tau^2)\|u\|_{Y_\tau^\co}
     (\|u\|_{H^4} + \|u\|_{X_\tau^\co})
   .
   \label{EQ119}
  \end{equation}
By interchanging indices, it suffices to show the bound holds only for $\calP_{m,j}^{\nor,1}$. 
Similarly to~\eqref{EQ110}, 
we bound
\begin{equation}
   \sum_{m=1}^\infty \sum_{j=0}^{m-1} 
    \calP_{m,j}^{\nor, 1}
   \les
   \calP_{1,0}^{\nor, 1}
   + \sum_{m=1}^\infty
       \calP_{m,m-1}^{\nor,1}
   + \sum_{m=2}^\infty \sum_{j=0}^{m-2}
       \calP_{m,j}^{\nor,1}
   ,
   \llabel{EQ120}
  \end{equation}
where
  \begin{equation}
   \calP_{1,0}^{\nor,1}
   + \sum_{m=1}^\infty \calP_{m,m-1}^{\nor,1}
   \les
   \|u\|_{H^4}^2
   + \tau \|u\|_{H^4}\|u\|_{Y_\tau^\co}
   .
   \llabel{EQ121}
  \end{equation}
For the third term, we use Lemma~\ref{L09} to write
  \begin{equation}
   \sum_{m=2}^\infty \sum_{j=0}^{m-2} 
     \calP_{m,j}^{\nor,1}
   \les
   \sum_{m=2}^\infty \sum_{|\beta| = m-1} \sum_{j=0}^{m-2} \sum_{|\gamma| = j} 
     \sum_{i=1}^4 {\beta \choose \gamma} M_\beta E_{\beta,\gamma}^i
     \frac{m \tau^{m-1}}{(m-1)!}
   .
   \llabel{EQ122}
  \end{equation}
For $E_{\beta,\gamma}^1$,
we get
  \begin{align}
   \begin{split}
    &
    \sum_{m=2}^\infty \sum_{|\beta| = m-1} \sum_{j=0}^{m-2} \sum_{|\gamma| = j} {\beta \choose \gamma}
      M_\beta 1_{\{\gamma_3 =0, \beta_3 \geq 1, \gamma' < \beta'\}}
      \|J^\gamma \bar{\partial}^2 u\|_{H^3} \|J^{\beta - \gamma + e_3 - \bar{e}} u\|_{H^3}
      \frac{m\tau^{m-2}}{(m-1)!}
    \\&\indeq\indeq\les
    \tau
    \sum_{m=2}^\infty \sum_{|\beta| = m-1} \sum_{j=0}^{m-2} \sum_{|\gamma|=j}
      \Big(
        M_\gamma \|J^\gamma \bar{\partial}^2 u\|_{H^3}
        \frac{(j+1)\tau^j}{(j+1)!}
      \Big)
    \\&\indeq\indeq\indeq\indeq\indeq\indeq
    \times
    \Big(
      M_{\beta - \gamma} \|J^{\beta - \gamma + e_3-\bar{e}}u\|_{H^3}
      \frac{\tau^{m-j-2}}{(m-j-2)!}
    \Big)
    \mathcal{A}_{\alpha,\beta,s}1_{\{\gamma_3 =0, \beta_3 \geq 1\}}
    \\&\indeq\indeq\les
    \tau\|u\|_{X_\tau^\co}\|u\|_{Y_\tau^\co}
    ,
   \end{split}
   \llabel{EQ123}
  \end{align}
since, for $m \geq 1$ and $0\leq j \leq m-2$, 
  \begin{align}
   \begin{split}
    \mathcal{A}_{\alpha,\beta,s}
    1_{\{\gamma_3 =0, \beta_3 \geq 1\}}
    &=
    M_\beta M_{\gamma}^{-1} M_{\beta - \gamma}^{-1} {\beta \choose \gamma}
    \frac{(j+1)! (m-j-2)!}{(m-1)!}
    \frac{m}{j+1}
    1_{\{\gamma_3 =0, \beta_3 \geq 1\}}
    \\&\leq
    {m-1-\beta_3 \choose j} {m-1 \choose j+1}^{-s} \frac{m}{j+1}
    \leq
    {m-2 \choose j} {m-1 \choose j+1}^{-s} \frac{m}{j+1}
    \\&\leq
    {m-1 \choose j+1}^{1-s} \frac{j+1}{m-1} \frac{m}{j+1}
    \leq 1
    .
   \end{split}
   \llabel{EQ124}
  \end{align}
For $E_{\beta,\gamma}^2$, 
we first estimate the highest order product in $E_{\beta,\gamma}^2$, namely $\beta_3\|J^\gamma \bar{\partial}^2 u\|_{H^3} \|J^{\beta - \gamma} u\|_{H^3}$.
We acquire
  \begin{align}
   \begin{split}
    &
    \sum_{m=2}^\infty \sum_{|\beta| =m-1} \sum_{j=0}^{m-2}\sum_{|\gamma| = j} {\beta \choose \gamma}
      M_\beta1_{\{\gamma_3 = 0, \beta_3 \geq 1, \gamma' = \beta'\}}
      \beta_3
      \|J^\gamma \bar{\partial}^2 u\|_{H^3}\|J^{\beta - \gamma} u\|_{H^3}
      \frac{m\tau^{m-1}}{(m-1)!}
    \\&\indeq\indeq\les
    \tau\|u\|_{H^4}
    \sum_{m=2}^\infty \sum_{|\beta| =m-1} \sum_{|\gamma| = m-2} 
      \Big(
        M_\gamma \|J^\gamma \bar{\partial}^2 u\|_{H^3}
        \frac{(m-1)\tau^{m-2}}{(m-1)!}
      \Big)
      \mathcal{A}_{\alpha,\beta,s}^1
      1_{\{\gamma_3 = 0, \gamma' = \beta'\}}
    \\&\indeq\indeq\indeq\indeq
    + \tau
    \sum_{m=3}^\infty \sum_{|\beta| = m-1} \sum_{j=0}^{m-3}\sum_{|\gamma| = j}
        \Big(
          M_\gamma \|J^\gamma \bar{\partial}^2 u\|_{H^3}
          \frac{\tau^{j+1}}{(j+1)!}
        \Big)
    \\&\indeq\indeq\indeq\indeq\indeq\indeq\indeq
    \times
      \Big(
        M_{\beta-\gamma} \|J^{\beta-\gamma} u\|_{H^3}
        \frac{(m-j-2)\tau^{m-j-3}}{(m-j-2)!}
      \Big)
      \mathcal{A}_{\alpha,\beta,s}^2 
      1_{\{\gamma_3 = 0, \gamma' = \beta'\}}
    \\&\indeq\indeq\les
    \tau\|u\|_{H^4}\|u\|_{Y_\tau^\co}
    + \tau \|u\|_{X_\tau^\co}\|u\|_{Y_\tau^\co}
    .
   \end{split}
   \label{EQ125}
  \end{align}
To see this, note that $\beta_31_{\{\gamma_3 = 0, \gamma' = \beta'\}} = 1$ 
when $|\beta| = |\gamma| +1$.
Hence, for $m \geq 2$,
  \begin{align}
   \begin{split}
    \mathcal{A}_{\alpha,\beta,s}^1
    1_{\{\gamma_3 =0, \gamma' = \beta'\}}
    &=
    M_\beta M_\gamma^{-1} {\beta \choose \gamma} \frac{\beta_3 m}{m-1}
    1_{\{\gamma_3 =0, \gamma' = \beta'\}}
    \\&\leq
    {|\beta'| \choose |\gamma'| }{\beta_3 \choose \gamma_3} 
    \frac{\beta_3m}{m-1}
    1_{\{\gamma_3 = 0, \gamma' = \beta'\}}
    \leq
    \frac{m}{m-1}
    \leq 1
    .
   \end{split}
   \llabel{EQ126}
  \end{align}
More generally, 
$\beta_3 1_{\{\gamma_3 =0, \beta' =\gamma'\}} = m-j-1$ when $|\beta| = m-1$ and $|\gamma| =j$.
Thus, for $m\geq 3$ and $0 \leq j \leq m-3$, 
  \begin{align}
   \begin{split}
    \mathcal{A}_{\alpha,\beta,s}^2
    1_{\{\gamma_3 = 0, \gamma' = \beta'\}}
    &=
    M_{\beta} M_{\gamma}^{-1} M_{\beta-\gamma}^{-1}{\beta \choose \gamma}
    \frac{(j+1)! (m-j-2)!}{(m-1)!}
    \frac{\beta_3m}{m-j-2}
    1_{\{\gamma_3 = 0, \beta' = \gamma'\}}
    \\&\leq
    {|\beta'|\choose |\gamma'|}{\beta_3 \choose 0}
    {m-1 \choose j+1}^{-s}
    m
    \les 1
    ,
   \end{split}
   \llabel{EQ127}
  \end{align}
where, in the second line, we bounded $(m-j-1)(m-j-2)^{-1} \les 1$ and, in the last inequality, we used ${m-1\choose j+1}^{-s} \les (m-1)^{-s}$ for $0 \leq j \leq m-3$. 

For the lower-order product in $E_{\beta,\gamma}^{2}$, namely $\|J^\gamma \bar{\partial}u\|_{H^3} \|J^{\beta - \gamma} u\|_{H^3}$, a similar computation to~\eqref{EQ125} gives 
  \begin{equation}
   \tau^2 \|u\|_{H^4}\|u\|_{Y_\tau^\co} 
   + \tau^2 \|u\|_{X_\tau^\co}\|u\|_{Y_\tau^\co}
   .
   \llabel{EQ128}
  \end{equation}
For $E_{\beta,\gamma}^3$, we follow the same computation as in~\eqref{EQ115}--\eqref{EQ117}, 
to attain
  \begin{equation}
   \sum_{m=2}^\infty \sum_{|\beta| = m-1} \sum_{j=0}^{m-2} \sum_{|\gamma| = j} {\beta \choose \gamma}
     E_{\beta,\gamma}^3
     \frac{m\tau^{m-1}}{(m-1)!}
   \les
   \tau \|u\|_{X_\tau^\co} \|u\|_{Y_\tau^\co}
   .
   \llabel{EQ129}
  \end{equation}
Finally, the lower-order term $E_{\beta,\gamma}^4$ admits the bound
  \begin{equation}
   \sum_{m=2}^\infty \sum_{|\beta| = m-1} \sum_{j=0}^{m-2} \sum_{|\gamma| = j} {\beta \choose \gamma}
     M_\beta E_{\beta,\gamma}^4 
     \frac{m\tau^{m-1}}{(m-1)!}
   \les
   \tau^2 \|u\|_{X_\tau^\co} \|u\|_{Y_\tau^\co}
   ,
   \llabel{EQ130}
  \end{equation}
though we omit the computation for the sake of brevity. 
This completes the proof of~\eqref{EQ119} and, therefore, the proof of the lemma.
\end{proof}
\subsection{Auxiliary estimates}\label{sec03c}
Here, we prove Lemma~\ref{L20}, which is used to show that $u \in C([0,T_\ast); X_{\tau(t)}^\co)$.

\begin{proof}[Proof of Lemma~\ref{L20}]
The proof is similar to that of previous lemmas, so we only estimate one term and explain how to apply the same reasoning to the others.
Let $\sigma < \tau$.
Recall that 
  \begin{equation}
   \|u\cdot \nabla u\|_{X_\sigma^\co}
   + \|\nabla p\|_{X_\sigma^\co}
   \leq
   \|u_h \bar{\partial} u\|_{X_\sigma^\co}
   + \|u_3 \partial_3 u\|_{X_\sigma^\co}
   + \|\nabla p\|_{X_\sigma^\co}
   .
   \label{EQ130a}
  \end{equation}
Our chosen example is $\|u_h \bar{\partial} u\|_{X_\sigma^\co}$. 
We have
  \begin{align}
   \begin{split}
    \|u_h \bar{\partial} u\|_{X_\sigma^\co}
    &\leq
    \sum_{m=1}^\infty \sum_{|\alpha| = m} \sum_{j=0}^m \sum_{|\beta| = j}{\alpha \choose \beta}
      M_\alpha 
      \|J^\beta u\|_{H^3} \|J^{\alpha - \beta} \bar{\partial} u\|_{H^3}
      \frac{\sigma^{m-1}}{(m-1)!}
    \\&\les
    \sigma^{-1}\|u\|_{H^3}\|u\|_{X_\tau^\co}
    + \sum_{m=1}^{\infty} \sum_{|\alpha| =m} \sum_{j=1}^m \sum_{|\beta|= j}
        \Big( 
         M_\beta 
         \|J^\beta u\|_{H^3} 
         \frac{\tau^{j-1}}{(j-1)!}
        \Big)
    \\&\indeq\indeq\indeq\indeq
        \times \Big(
         M_{\alpha -\beta} 
         \|J^{\alpha - \beta} \bar{\partial}u\|_{H^3}
         \frac{\tau^{m-j}}{(m-j)!}
        \Big)
        \mathcal{A}_{\alpha,\beta,s}
        \left( \frac{\sigma}{\tau}\right)^{m-1}
    \\&\les
    \sigma^{-1} \|u\|_{H^3} \|u\|_{X_\tau^\co}
    + \|u\|_{X_\tau^\co}^2
    ,
   \end{split}
   \llabel{EQ130b}
  \end{align}
where, since $\sigma < \tau$,
  \begin{equation}
   \mathcal{A}_{\alpha,\beta,s}
   \left( \frac{\sigma}{\tau}\right)^{m-1}
   =
   M_\alpha M_\beta^{-1} M_{\alpha- \beta}^{-1} {\alpha \choose \beta}
   \frac{(j-1)! (m-j)!}{(m-1)!}
   \left( \frac{\sigma}{\tau}\right)^{m-1}
   \leq
   \frac{m}{j} \left( \frac{\sigma}{\tau}\right)^{m-1}
   \les
   1
   .
   \llabel{EQ130c}
  \end{equation}
Indeed, for the other terms in~\eqref{EQ130a}, we repeat most of the estimates in Lemmas~\ref{L01}--\ref{L03}, but instead bound the terms by a polynomial in $\|u\|_{H^4}$ and~$\|u\|_{X_\tau^\co}$.
Then we use the difference between $\sigma$ and $\tau$ to control any leftover polynomial of the indices  $m$ or~$j$.
The factor $\sigma^{-1}$ often appears when estimating lower order terms, but this is harmless for the continuity argument in which this lemma is used.
\end{proof}

Finally, we prove Lemma~\ref{L21}.
The proof is very similar to that of Lemmas~\ref{L01} and~\ref{L02},
so we give a sketch, pointing out the differences. 

\begin{proof}[Proof of Lemma~\ref{L21}]
First we show that 
  \begin{equation}
   \|u \cdot \nabla u\|_{X_\tau^\co}
   \les
   \|u\|_{\widetilde{X}_\tau^\co} \|u\|_{\widetilde{Y}_\tau^\co}
   .
   \llabel{EQ502}
  \end{equation}
Similarly to~\eqref{EQ53}, 
we may write 
  \begin{equation}
   \|u\cdot \nabla u\|_{X_\tau^\co}
   \leq
   \sum_{m=1}^\infty \sum_{j=0}^m
     (\mathcal{C}_{m,j}^{\rm tan} + \mathcal{C}_{m,j}^{\rm nor})
   ,
   \llabel{EQ503}
  \end{equation}
where
  \begin{equation}
   C_{m,j}^{\rm tan} 
   =
   \sum_{|\alpha| =m} \sum_{|\beta| =j, \beta \leq \alpha} {\alpha \choose \beta}
     M_\alpha \|J^\beta u_h J^{\alpha - \beta} \bar{\partial} u\|_{H^3}
     \frac{\tau^{m-1}}{(m-1)!}
   ,
   \llabel{EQ504}
  \end{equation}
and 
  \begin{equation}
   C_{m,j}^{\rm nor}
   = 
   \sum_{|\alpha| =m} \sum_{|\beta| = j, \beta \leq \alpha} {\alpha \choose \beta}
     M_\alpha \|J^\beta u_3 J^{\alpha - \beta} \partial_3 u\|_{H^3}
     \frac{\tau^{m-1}}{(m-1)!}
     .
   \llabel{EQ505}
  \end{equation}
We only consider $C_{m,j}^{\rm nor}$ as the tangential term is simpler. 
Extracting lower-order terms, 
we get 
  \begin{equation}
   \sum_{m=1}^\infty \sum_{j=0}^m \mathcal{C}_{m,j}^{\nor}
   =
   \mathcal{C}_{1,1}^{\rm nor}
   + \sum_{m=1}^\infty \mathcal{C}_{m,0}^{\rm nor}
   + \sum_{m=1}^\infty \mathcal{C}_{m,m}^{\rm nor}
   + \sum_{m=2}^\infty \sum_{j=1}^{m-1} \mathcal{C}_{m,j}^{\rm nor}
   .
   \label{EQ506}
  \end{equation}
Note that the only difference between~\eqref{EQ506} and~\eqref{EQ62} is the additional term $\sum \mathcal{C}_{m,0}^{\rm nor}$. 
For this term, we have
  \begin{equation}
    \sum_{m=1}^\infty \mathcal{C}_{m,0}
    =
    \sum_{m=1}^\infty \sum_{|\alpha| =m} 
      M_\alpha \|u_3 J^\alpha \partial_3 u\|_{H^3}
      \frac{\tau^{m-1}}{(m-1)!}
    \les
    \|u\|_{\mathcal{H}}
    \sum_{m=1}^\infty
      M_\alpha 
      \|J^{\alpha + e_3} u\|_{H^3}
      \frac{m\tau^{m-1}}{m!}
    \les
    \|u\|_{\widetilde{X}_\tau^\co} \|u\|_{\widetilde{Y}_\tau^\co}
    ,
    \llabel{EQ507}
  \end{equation}
where we used~\eqref{EQ28b} in the second inequality.
For the first term in~\eqref{EQ506}, 
we get
  \begin{equation}
   \mathcal{C}_{1,1}^{\rm nor}
   =
   \sum_{|\alpha| = 1} \sum_{|\beta|=1, \beta \leq \alpha}{\alpha \choose \beta}
     M_\alpha
     \|J^\beta u_3 J^{\alpha-\beta} \partial_3u\|_{H^3}
   \les
   \|\phi^{-1} J^\beta u_3\|_{H^3}\|J^{e_3} u\|_{H^3}
   .
   \label{EQ508}
  \end{equation}
If $\beta_3 =1$, then this is immediately bounded by
$\|u\|_{\mathcal{H}}^2$, which in turn is bounded by $\|u\|_{\widetilde{X}_\tau^\co} \|u\|_{\widetilde{Y}_\tau^\co}$.
If $\beta_3 =0$, then, using~\eqref{EQ28b}, we bound
  \begin{equation}
   \mathcal{C}_{1,1}^{\rm nor}
   \les
   (\|J^\beta u\|_{H^3} + \|J^{\beta + \bar{e}} u_h\|_{H^3})
   \|J^{e_3} u\|_{H^3}
   \les
   \|u\|_{\mathcal{H}}^2
   + \|u\|_{Y_\tau^\co} \|u\|_{\mathcal{H}}
   \les
   \|u\|_{\widetilde{X}_\tau^\co} \|u\|_{\widetilde{Y}_\tau^\co}
   .
   \label{EQ509}
  \end{equation}
Note that the gain in derivative from the case $\beta_3 = 0$ is bounded by the first term in the sum~$\|u\|_{Y_\tau^\co}$.
Next, the third and fourth terms in~\eqref{EQ506} are bounded as they were in the proof of Lemma~\ref{L01}, except the appearance of $\|u\|_{H^4}$ in all estimates may be replaced by $\|u\|_{\mathcal{H}}$ by the same reasoning as in~\eqref{EQ508} and~\eqref{EQ509}.
Bounding the gain in derivative by the $Y_\tau$ norm and performing a similar conormalization also yields
  \begin{equation}
   \|u\cdot \nabla u\|_{\mathcal{H}}
   \les
   \|u\|_{\widetilde{X}_\tau^\co}\|u\|_{\widetilde{Y}_\tau^\co}
   .
   \llabel{EQ510}
  \end{equation}
This completes the proof of the advection bound in~\eqref{EQ501}. 
The pressure estimate follows a similar analysis. 
\end{proof}

\section{Proof of Theorem~\ref{T02}: Persistence in $A_\tau^\co$}\label{sec04}
The purpose of this section is to prove Theorem~\ref{T02}, which establishes persistence in the analytic norm $A_\tau^\co$ equipped with the conormal operator $L^k = (\phi \partial_3)^k$.
Before this, we state three lemmas. 
The first two lemmas bound commutator terms coming from the advection in~\eqref{EQ00}.

\begin{Lemma}
\label{L10}
We have 
  \begin{equation}
   \sum_{m=1}^\infty \sum_{|\alpha| = m} \sum_{0 < \beta \leq \alpha} {\alpha \choose \beta}
     M_\alpha \eps^{\alpha_3}
     \|K^\beta u\cdot \nabla K^{\alpha - \beta} u\|_{H^3}
     \frac{\tau^{m-1}}{(m-1)!}
   \les
   \|u\|_{H^4}^2
   + (\tau + \tau^2)
     \|u\|_{B_\tau^\co}
     (\|u\|_{H^4} + \|u\|_{A_\tau^\co})
   .
   \label{EQ205}
  \end{equation}
\end{Lemma}

\begin{Lemma}
\label{L11}
We have 
  \begin{align}
   \begin{split}
   &
   \sum_{m=1}^\infty \sum_{|\alpha| = m} \sum_{0 \leq \beta < \alpha} {\alpha \choose \beta}
     M_\alpha \eps^{\alpha_3} 
     \|K^\beta u_3 [K^{\alpha - \beta}, \partial_3] u\|_{H^3}
     \frac{\tau^{m-1}}{(m-1)!}
   \\&\indeq\indeq\les
   \|u\|_{H^4}^2
   + (\tau + \tau^2)
     \|u\|_{B_\tau^\co}
     (\|u\|_{H^4} + \|u\|_{A_\tau^\co})
   \end{split}
   \label{EQ206}
  \end{align}
and 
  \begin{equation}
   \sum_{m=1}^\infty \sum_{|\alpha| = m} \sum_{0 \leq |\gamma| \leq 3} \sum_{0 < \mu \leq \gamma}{\gamma \choose \mu}
     M_\alpha \eps^{\alpha_3}
     \|\partial^\mu u \cdot \nabla \partial^\gamma K^\alpha u\|_{H^3}
     \frac{\tau^{m-1}}{(m-1)!}
   \les
   \|u\|_{H^4}^2
   + \tau \|u\|_{H^4} \|u\|_{B_\tau^\co}
   .
   \label{EQ207}
  \end{equation}
\end{Lemma}

The proofs of Lemmas~\ref{L10} and~\ref{L11} are in Section~\ref{sec04a}.
The third lemma controls the pressure and is proven in Section~\ref{sec04b}.

\begin{Lemma}
\label{L12}
We have 
  \begin{equation}
   \|\nabla p\|_{A_\tau^\co}
   \les 
   \|u\|_{H^4}^2
   + (\tau + \tau^2)
     \|u\|_{B_\tau^\co}
     (\|u\|_{H^4} + \|u\|_{A_\tau^\co})
     .
   \label{EQ208}
  \end{equation}
\end{Lemma}

Now, we prove Theorem~\ref{T02}. 

\begin{proof}[Proof of Theorem~\ref{T02}]
The construction argument is analogous to that of Theorem~\ref{T02}, so we only provide the a~priori estimates. 
From definitions~\eqref{EQ202}--\eqref{EQ203}, we have
  \begin{equation}
   \frac{d}{dt} \|u\|_{A_\tau^\co}
   =
   \dot{\tau} \|u\|_{B_\tau^\co}
   + \sum_{m =1}^\infty
       \left( \frac{d}{dt} |u|_{\Eco^{m,3}}\right)
       \frac{\tau^{m-1}}{(m-1)!}
   .
   \llabel{EQ209}
  \end{equation}
By testing the Euler equations~\eqref{EQ00} with $M_\alpha \eps^{\alpha_3}\partial^\gamma K^\alpha u$ and summing over $|\alpha| =m$ and $|\gamma| \leq 3$,
we obtain
  \begin{align}
   \begin{split}
    \frac{d}{dt} |u|_{\Eco^{m,3}}
    &\leq
    \sum_{|\alpha| = m} \sum_{0 < \beta \leq \alpha} {\alpha \choose \beta}
      M_\alpha \eps^{\alpha_3}
      \|K^\beta u \cdot \nabla K^{\alpha - \beta} u\|_{H^3}
    + \sum_{|\alpha| = m} \sum_{0 \leq \beta < \alpha} {\alpha \choose \beta}
        M_\alpha \eps^{\alpha_3}
        \|K^\beta u_3 [K^{\alpha - \beta}, \partial_3] u\|_{H^3}
    \\&\indeq\indeq
    + \sum_{|\alpha| = m} \sum_{0 \leq |\gamma| \leq 3} \sum_{0 < \mu \leq \gamma} {\gamma \choose \mu}
        M_\alpha \eps^{\alpha_3} 
        \|\partial^\mu u \cdot \nabla \partial^\gamma K^\alpha u\|_{L^2}
    + \sum_{|\alpha| =m} 
        M_\alpha \eps^{\alpha_3}
        \|K^\alpha \nabla p\|_{H^3}
    .
   \end{split}
   \llabel{EQ210}
  \end{align}
After multiplying the above by $\tau^{m-1}(m-1)!^{-s}$, summing over $m \geq 1$, and using Lemmas~\ref{L10}--\ref{L12},
we arrive at
  \begin{equation}
   \frac{d}{dt} \|u\|_{A_\tau^\co}
   - \dot{\tau} \|u\|_{B_\tau^\co}
   \leq
   C\|u\|_{H^4}^2
   + C(\tau+\tau^2)
     \|u\|_{B_\tau^\co}
     (\|u\|_{H^4} + \|u\|_{B_\tau^\co})
   .
   \llabel{EQ211}
  \end{equation}
The proof then proceeds exactly as in Theorem~\ref{T01}.
\end{proof}

In preparation for Sections~\ref{sec04a} and~\ref{sec04b}, we establish two estimates pertaining to the commutator $[K^{\alpha-\beta},\partial_3]$. 

\begin{Lemma}
\label{L14}
The conormal weight $\phi$ is uniformly analytic on~$\mathbb{R}_+$.
That is, for $j \geq 0$,
  \begin{equation}
   \|\partial_x^j \phi\|_{L^\infty(\mathbb{R}_+)} 
   \leq
   C'C^j j!
   ,
   \label{EQ226a}
  \end{equation}
where $C',C$ are positive universal constants.
Moreover, for $j \geq 0$,
  \begin{equation}
   \|L^j\phi\|_{W^{3,\infty}(\mathbb{R}_+)}
   \leq
   C_0 C_1^j j!
   ,
   \label{EQ226}
  \end{equation}
where $C_0,C_1$ are positive universal constants.
The bound~\eqref{EQ226} also holds with $\phi$ replaced by~$\phi^2$.
\end{Lemma}

\begin{proof}
The bound~\eqref{EQ226a} is clear from the formula,
  \begin{equation}
   \phi^{(k)} 
   = 
   (-1)^{k+1} k!(1-\phi)^{k+1}
   ,
   \llabel{EQ226b}
  \end{equation}
using also $\phi(x) \leq 1$ for $x > 0$.

Now, we show that~\eqref{EQ226} holds.
We first claim that 
  \begin{equation}
   \|L^j \phi\|_{L^\infty(\mathbb{R}_+)}
   \leq 
   C_0 C_1^j j!
   \label{EQ227}
  \end{equation}
for $j \geq 0$, where $C_0, C_1$ are positive constants. 
Indeed, define $\zeta \colon \mathbb{R}_+ \to \mathbb{R}$ by $\zeta(x) = x + \log x$. 
Then  $\zeta$ is invertible and $\zeta^{-1}(y) = W(e^y)$,
where $W$ is the Lambert $W$-function.
Denoting $y = \zeta(x)$,
we have $x = W(e^y)$ and
  \begin{equation}
   L^j \phi(x) 
   =
   (\phi(x) \partial_x)^j \phi(x)
   =
   \partial_y^j \phi(W(e^y))
   .
   \llabel{EQ228}
  \end{equation}
Since $f = \phi \circ W \circ \exp$ lies in $[0,1]$ and solves
  \begin{equation}
   f' 
   =
   (1-f)^2f
   ,
   \llabel{EQ228a}
  \end{equation}
it follows that $f$ is uniformly analytic on~$\mathbb{R}$.
Hence, there exist constants $C_0$ and $C_1$ such that
  \begin{equation}
   \|L^j \phi\|_{L^\infty(\mathbb{R}_+)}
   =
   \|\partial_y^j (f \circ W \circ \exp)\|_{L^\infty(\mathbb{R})}
   \leq
   C_0 C_1^j j!
   ,
   \llabel{EQ229}
  \end{equation}
which completes the proof of~\eqref{EQ227}.

To bound the higher derivatives in the $W^{3,\infty}$ norm, we explicitly compute~$\partial_x^3 L^j \phi$. 
However, we only estimate $\|\partial_x L^j \phi\|_{L^\infty}$
and omit the computations for the second and third derivatives as they follow the same idea.
Define $\psi(y) = 1 + W(e^y)^{-1}$.
From our coordinate change $y = x + \log x$, we get
  \begin{equation}
   \partial_x L^j \phi(x)
   =
   \psi(y) \partial_y^{j+1} \phi(W(e^y))
   .
   \label{EQ230}
  \end{equation}
For $y \geq 0$, 
  \begin{equation}
   |\partial_x L^j \phi(x)|
   \leq 
   |\psi(y)| \|\partial_y^{j+1} \phi(W(e^y))\|_{L^\infty(\mathbb{R})}
   \leq
   C C_0 C_1^{j+1} (j+1)!
   \leq
   C_2 C_3^j j!
   ,
   \label{EQ231}
  \end{equation}
where we used $|\psi(y)| \leq C$ for $y \geq 0$.

For $y < 0$, we cannot simply repeat the estimate~\eqref{EQ231} since $\psi(y)$ is unbounded as $y \to -\infty$.
Indeed, 
  \begin{equation}
   \psi(y)
   \sim
   e^{-y} + 2 -\frac{1}{2} e^y + O(e^{2y})
   \quad\quad
   \text{as } y \to -\infty
   .
   \llabel{EQ232}
  \end{equation}
However, we may take advantage of the extra derivative on $\phi(W(e^y))$ in~\eqref{EQ230}, which exhibits a decay rate of $e^y$ as $y \to -\infty$.
Denoting $H(u) = (d/dz)|_{z=u}\phi(W(z))$,
we have
  \begin{equation}
   \psi(y) \partial_y^j \partial_y \phi(W(e^y))
   =
   \psi(y) \partial_y^j ( e^y H(e^y))
   = 
   \psi(y) e^y 
   \sum_{0\leq \ell \leq j} {j \choose \ell} 
     \partial_y^\ell H(e^y)
   ,
  \label{EQ233}
  \end{equation}
where $|\psi(y) e^y|$ is bounded for $y \leq 0$.
Now, note that 
  \begin{equation}
   H(e^y) 
   =
   (1-\phi(W(e^y)))^2
   \phi(W(e^y))e^{-y}
   ,
   \llabel{EQ233a}
  \end{equation}
where $g(y) = \phi(W(e^y)) e^{-y}$ solves $g' = g^3 e^{2y} - 2g^2 e^y$ and has a compact image for $y \leq 0$. 
Moreover, $e^{2y}$ and $e^{y}$ are bounded for $y \leq 0$.
Therefore, $g$ is uniformly analytic on $y \leq 0$, implying $H\circ \exp$ is uniformly analytic as well.
This gives us 
  \begin{equation}
   |\partial_x L^j \phi(x)|
   \les
   \sum_{0\leq \ell \leq j} {j \choose \ell}
   C^\ell \ell!
   \les
   C^j j!
   ,
   \llabel{EQ234}
  \end{equation}
completing the proof of the lemma.
\end{proof}
\begin{remark}
\label{R04}
The above theorem holds for~$\phi^2 $. 
To see this, note that $f = \phi \circ W \circ \exp$ satisfies $f' = (1-f)^2 f^2$.
We also note that~\eqref{EQ226} is promptly verified by Proposition~\ref{P01} in the Appendix since $(\phi \partial_x)^n \phi = (\phi \partial_x)^{n+1} x$.
\end{remark}

The next lemma bounds the commutator $[L^j, \partial_3]f$. 
For brevity, we work in one dimension.

\begin{Lemma}
\label{L15}
Let $f$ be a smooth function on~$\mathbb{R}_+$.
For $j \geq 1$, we have 
  \begin{equation}
   [L^j, \partial_x] f
   =
   \sum_{\ell=0}^{j-1} {j \choose \ell}
     c_\ell(x) 
     L^{j-\ell} f
   ,
   \llabel{EQ235}
  \end{equation}
where 
  \begin{equation}
   \|c_\ell \|_{W^{3,\infty}(\mathbb{R}_+)} 
   \les
   C^\ell \ell!
   \llabel{EQ235a}
  \end{equation}
for all $\ell \geq 0$, where $C$ is a positive universal constant. 
\end{Lemma}

\begin{proof}
Under the change of coordinates $y = x +\log x$, 
we find 
  \begin{align}
   \begin{split}
    [L^j, \partial_{x}] f(x)
    &=
    [\partial_{y}^j, \psi(y)\partial_{y_3}]f(W(e^{y}))
    =
    \sum_{0< \ell \leq j} {j\choose \ell}
      \partial_{y}^\ell \psi(y) \partial_{y}^{j+1 - \ell} f(W(e^{y}))
    \\&=
    \sum_{\ell =1}^j {j \choose \ell}
      \psi(y)^{-1}\partial_{y}^\ell \psi(y) 
      \psi(y) \partial_{y} \partial_{y}^{j-\ell} f(W(e^{y}))
    .
   \end{split}
   \llabel{EQ236}
  \end{align}
Changing back to the original coordinates yields
  \begin{equation}
   [L^j, \partial_{x}] f
   = 
   \sum_{\ell=1}^{j} {j \choose \ell}  
     \phi L^{\ell}\left(\frac{1}{\phi}\right) \partial_{x} L^{j-\ell} f
   .
   \llabel{EQ237}
  \end{equation}
Finally, Lemma~\ref{L19} yields
  \begin{equation}
   \|\phi L^{\ell} (\phi^{-1})\|_{W^{3,\infty}}
   \les
   C^\ell \ell!
   .
   \llabel{EQ237a}
  \end{equation}
This completes the proof of the lemma.
\end{proof}

\subsection{Advection estimates}\label{sec04a}
First, we state and prove a conormalization lemma. 

\begin{Lemma}
\label{L13}
Let $0 \leq \beta < \alpha$.
Then
  \begin{equation}
    \|K^\beta u_3 \partial_3 K^{\alpha- \beta} u\|_{H^3}
    \les
    F_{\alpha,\beta}^1
    + F_{\alpha,\beta}^2
    + F_{\alpha,\beta}^3
   ,
   \llabel{EQ212}
  \end{equation}
where
  \begin{align}
   \begin{split}
    F_{\alpha,\beta}^1
    &= 
    1_{\{\beta_3 = 0\}} 
      \|K^\beta u\|_{H^3} \|K^{\alpha - \beta + e_3} u\|_{H^3}
    ,\\
    F_{\alpha,\beta}^2
    &= 
    1_{\{\beta_3 = 0\}}
      \|K^\beta u\|_{H^3} \|K^{\alpha - \beta} u\|_{H^3}
    + 1_{\{\beta_3 \geq 1\}}
        \|K^{\beta - e_3} u\|_{H^3}\|K^{\alpha - \beta + e_3} u\|_{H^3}
    ,\\
    F_{\alpha,\beta}^3
    &= 
    1_{\{\beta_3 \geq 1\}} 
      \|K^{\beta - e_3} u\|_{H^3} \|K^{\alpha - \beta} u\|_{H^3}
    .
   \end{split}
   \llabel{EQ213}
  \end{align}
\end{Lemma}

Note that the terms $F_{\alpha,\beta}^i$ are organized based on energy levels.

\begin{proof}[Proof of Lemma~\ref{L13}]
We start by applying the Leibniz rule to the $H^3$ norm, giving 
  \begin{equation}
   \|K^\beta u_3 \partial_3 K^{\alpha - \beta} u\|_{H^3}
   \les
   \sum_{0 \leq |\gamma | \leq 3} 
     \|K^\beta u_3 \partial^\gamma \partial_3 K^{\alpha - \beta} u\|_{L^2}
   + \sum_{0 \leq |\gamma| \leq 3} \sum_{0 < \mu \leq \gamma} 
       \|\partial^\mu K^\beta u_3 \partial^{\gamma - \mu} \partial_3 K^{\alpha - \beta} u\|_{L^2}
   ,
   \label{EQ214}
  \end{equation}
where we separated the terms corresponding to $\mu=0$.
Without loss of generality, we only consider the terms corresponding to $|\gamma| = 3$ 
as the other terms give subordinate bounds. 

We examine the case when $\beta_3 \geq 1$.
For the first term on the right-hand side of~\eqref{EQ214}, we have 
  \begin{align}
   \begin{split}
    &
    \|K^\beta u_3 \partial^\gamma \partial_3 K^{\alpha - \beta} u\|_{L^2}
    \\&\indeq\indeq=
    \|\phi^{-1} K^\beta u_3 \phi \partial^\gamma \partial_3 K^{\alpha - \beta} u\|_{L^2}
    \\&\indeq\indeq\leq
    \|\partial_3 K^{\beta - e_3} u_3\|_{L^\infty} \|[\phi,\partial^\gamma] \partial_3 K^{\alpha - \beta} u\|_{L^2}
    + \|K^{\beta - e_3} u\|_{H^3} \|\partial^\gamma K^{\alpha - \beta + e_3} u\|_{L^2}
    \\&\indeq\indeq\leq
    \|K^{\beta - e_3} u\|_{H^3} 
    (\|K^{\alpha - \beta} u\|_{H^3} + \|K^{\alpha - \beta + e_3} u\|_{H^3} )
    ,
   \end{split}
   \label{EQ215}
  \end{align}
where we note that
$[\phi, \partial_3^3] f \les \|\phi\|_{W^{3,\infty}} \| f\|_{H^2}$ for sufficiently smooth functions~$f$.
For the sum in~\eqref{EQ214}, we use the multiplicative Sobolev inequality~\eqref{EQ28a} to get
  \begin{align}
   \begin{split}
     &
     \sum_{|\gamma| = 3} \sum_{j=1}^3 \sum_{|\mu| = j, \mu \leq \gamma}
       \|\partial^\mu K^\beta u_3 \partial^{\gamma -\mu} \partial_3 K^{\alpha- \beta} u\|_{L^2}
     \\&\indeq\indeq\les
     \sum_{|\gamma| = 3} \sum_{j=1}^3 \sum_{|\mu|=j, \mu \leq \gamma}
       \|\partial^\mu K^\beta u_3\|_{H^{3-j}} \|\partial^{\gamma - \mu} \partial_3 K^{\alpha - \beta} u\|_{H^{j-1}}
    \les 
    \|K^\beta u\|_{H^3} \|K^{\alpha - \beta} u\|_{H^3}
    .
   \end{split}
   \llabel{EQ216}
  \end{align}

For $\beta_3 = 0$, we repeat the same argument as above, except that we replace~\eqref{EQ215} with
  \begin{equation}
   \|K^\beta u_3 \partial^\gamma \partial_3 K^{\alpha - \beta} u\|_{L^2}
   \les
   \|K^\beta u\|_{H^3}
   (\|K^{\alpha - \beta} u\|_{H^3} + \|K^{\alpha - \beta + e_3} u\|_{H^3})
   ,
   \llabel{EQ217}
  \end{equation}
where we used~\eqref{EQ28} and the embedding $H^3 \hookrightarrow W^{1,\infty}$ to estimate $\|\phi^{-1} K^\beta u_3\|_{L^\infty} \les \|K^\beta u\|_{H^3}$.
This completes the proof. 
\end{proof}

The proof of Lemma~\ref{L10} is similar to that of Lemma~\ref{L01}. 
Hence, we only give a sketch of the proof. 

\begin{proof}[Proof of Lemma~\ref{L10}]
We rewrite the left-hand side of~\eqref{EQ205} as 
  \begin{equation}
   \sum_{m=1}^\infty \sum_{|\alpha| = m} \sum_{0 < \beta \leq \alpha} {\alpha \choose \beta}
     M_\alpha \eps^{\alpha_3}
     \|K^\beta u \cdot \nabla K^{\alpha - \beta} u\|_{H^3}
     \frac{\tau^{m-1}}{(m-1)!}
   = 
   \sum_{m=1}^\infty \sum_{j=1}^m 
     (\mathcal{C}_{m,j}^{\rm tan} + \mathcal{C}_{m,j}^{\rm nor})
  ,
  \llabel{EQ218}
  \end{equation}
where 
  \begin{equation}
   \mathcal{C}_{m,j}^{\rm tan}
   =
   \sum_{|\alpha| =m} \sum_{|\beta| =j, \beta \leq \alpha} {\alpha \choose \beta}
     M_\alpha \eps^{\alpha_3}
     \|K^\beta u_h K^{\alpha - \beta} \bar{\partial} u\|_{H^3}
     \frac{\tau^{m-1}}{(m-1)!}
   \llabel{EQ219}
  \end{equation}
and 
  \begin{equation}
   \mathcal{C}_{m,j}^{\rm nor}
   =
   \sum_{|\alpha| = m} \sum_{|\beta| =j,\beta \leq \alpha} {\alpha \choose \beta} 
     M_\alpha \eps^{\alpha_3}
     \|K^\beta u_3 \partial_3 K^{\alpha - \beta} u\|_{H^3}
     \frac{\tau^{m-1}}{(m-1)!}
   .
   \llabel{EQ220}
  \end{equation}
With similar computations as for~\eqref{EQ55a}, we get
  \begin{equation}
   \sum_{m=1}^\infty \sum_{j=1}^m \mathcal{C}_{m,j}^{\rm tan}
   \les 
   \|u\|_{H^4}^2
   + \tau \|u\|_{H^4}\|u\|_{B_\tau^\co}
   + \tau \|u\|_{A_\tau^\co} \|u\|_{B_\tau^\co}
   .
   \llabel{EQ221}
  \end{equation}
For the normal terms, we write
  \begin{equation}
   \sum_{m=1}^\infty \sum_{j=1}^m \mathcal{C}_{m,j}^{\rm nor}
   \leq
   \mathcal{C}_{1,1}^{\rm nor}
   + \sum_{m=1}^\infty \mathcal{C}_{m,m}^{\rm nor}
   + \sum_{m=2}^\infty \sum_{j=1}^{m-1} \mathcal{C}_{m,j}^{\rm nor}
   ,
   \llabel{EQ222}
  \end{equation}
where, similarly to~\eqref{EQ63}, 
  \begin{equation}
   \mathcal{C}_{1,1}^{\rm nor} 
   + \sum_{m=1}^\infty \mathcal{C}_{m,m}^{\rm nor}
   \les 
   \|u\|_{H^4}^2
   + \tau \|u\|_{H^4} \|u\|_{B_\tau^\co}
   .
   \llabel{EQ223}
  \end{equation}
Next, by Lemma~\ref{L13}, 
  \begin{equation}
   \sum_{m=2}^\infty \sum_{j=1}^{m-1} \mathcal{C}_{m,j}^{\rm nor}
   \leq
   \sum_{m=2}^\infty \sum_{|\alpha| = m} \sum_{j=1}^{m-1} \sum_{|\beta| = j, \beta \leq \alpha} 
     \sum_{i=1}^3 {\alpha \choose \beta} M_\alpha \eps^{\alpha_3}
     F_{\alpha,\beta}^i \frac{\tau^{m-1}}{(m-1)!}
   .
   \llabel{EQ224}
  \end{equation}
  The estimates for $F_{\alpha,\beta}^1$ and $F_{\alpha,\beta}^2$ are similar to those for $E_{\alpha,\beta}^3$ and $E_{\alpha,\beta}^4$ in Lemma~\ref{L01}, respectively. 
The lower-order term $F_{\alpha,\beta}^3$ can also be estimated similarly,
leading to
  \begin{equation}
   \sum_{m=2}^\infty \sum_{j=1}^{m-1} \mathcal{C}_{m,j}^{\rm nor}
   \les 
   \|u\|_{H^4}^2
   + (\tau + \tau^2) \|u\|_{B_\tau^\co} (\|u\|_{H^4} + \|u\|_{A_\tau^\co})
   .
   \llabel{EQ225}
  \end{equation}
Combining this with~\eqref{EQ233} completes the proof of the lemma. 
\end{proof}

Now, we prove Lemma~\ref{L11} which bounds two commutators produced by the advection term.
For the rest of Section~\ref{sec04}, we assume that $\eps$ is small enough for any absorption arguments which take place.

\begin{proof}[Proof of Lemma~\ref{L11}]
The proof of~\eqref{EQ207} is identical to that of~\eqref{EQ18}, so we omit it.

Before we prove~\eqref{EQ206}, we must conormalize the commutator product. 
Let $0\leq \beta < \alpha$ with $1 \leq \ell \leq \alpha_3 - \beta_3$.
Using similar reasoning to Lemma~\ref{L13}, we have 
  \begin{equation}
   \|K^\beta u_3 \partial_3 L^{\alpha-\beta -\ell e_3} u\|_{H^3}
   \les
   \sum_{i=1}^3 F_1(\beta,\alpha - \beta - \ell e_3)
   ,
   \llabel{EQ237b}
  \end{equation}
where the terms $F_i$ are defined as 
  \begin{align}
   \begin{split}
    F_1(\beta,\alpha - \beta - \ell e_3)
    &=
    1_{\{\beta_3 = 0\}}
      \|K^\beta u\|_{H^3}
      \|K^{\alpha - \beta -(\ell-1)e_3} u\|_{H^3}
      ,
    \\
    F_2(\beta,\alpha-\beta-\ell e_3)
    &=
    1_{\{\beta_3 \geq 1\}}
      \|K^{\beta - e_3} u\|_{H^3}
      \|K^{\alpha- \beta - \ell e_3} u\|_{H^3}
    \\&\indeq\indeq
    + 1_{\{\beta_3 \geq 1\}}
       \|K^{\beta-e_3} u\|_{H^3}
       \|K^{\alpha - \beta - (\ell-1)e_3} u\|_{H^3}
      ,
    \\
    F_3(\beta, \alpha - \beta - \ell e_3)
    &=
    1_{\{\beta_3 =0\}}
      \|K^\beta u\|_{H^3}
      \|K^{\alpha - \beta -\ell e_3} u\|_{H^3}
   .
   \end{split}
   \llabel{EQ237c}
  \end{align}
However, in the following estimates, the restriction $\beta_3 = 0$ is not needed in estimating the sum of~$F_1$.
Hence, we replace $F_1$ with
  \begin{equation}
   F(\beta,\alpha-\beta-\ell e_3)
   =
   \|K^\beta u\|_{H^3}
   \|K^{\alpha - \beta -(\ell-1)e_3} u\|_{H^3}
   .
   \llabel{EQ237d}
  \end{equation}
Moreover, $F_2$ and $F_3$ are of lower order compared to $F$ and, by repeating the argument in this proof, one can show that the sums of $F_2$ and $F_3$ admit the same bound as~$F$.

We write
  \begin{align}
   \begin{split}
    &
    \sum_{m=1}^\infty \sum_{|\alpha| = m} \sum_{j=0}^{m-1} \sum_{|\beta| = j} {\alpha \choose \beta}
      M_\alpha \eps^{\alpha_3} 
      \|K^\beta u_3 [K^{\alpha - \beta}, \partial_3 ] u\|_{H^3} 
      \frac{\tau^{m-1}}{(m-1)!}
    \\&\indeq\indeq\les
    \sum_{m=1}^\infty \sum_{|\alpha| = m} \sum_{j=0}^{m-1} \sum_{|\beta|=j} 
    \sum_{\ell =1}^{\alpha_3 - \beta_3} {\alpha \choose \beta} \frac{(\alpha_3 - \beta_3)!}{(\alpha_3 - \beta_3-\ell)!} C^\ell
      M_\alpha \eps^{\alpha_3}
      F(\beta, \alpha-\beta -\ell e_3) 
      \frac{\tau^{m-1}}{(m-1)!}
    \\&\indeq\indeq=
    \sum_{\ell=1}^{\infty} \sum_{p,j= 0}^\infty \sum_{\substack{|\eta| =p \\ |\beta| = j}}{\eta + \beta + \ell e_3 \choose \beta}
      \frac{(\eta_3 + \ell)!}{\eta_3!} C^\ell
      M_{\eta + \beta} \eps^{\eta_3 + \beta_3 + \ell} 
      F(\beta, \eta)
      \frac{\tau^{p+j + \ell - 1}}{(p+j + \ell -1)!}
    \\&\indeq\indeq =:
    \sum_{\ell=1}^\infty \sum_{p,j= 0}^\infty 
      \mathcal{G}_{\ell, p, j}
    ,
   \end{split}
   \llabel{EQ238}
  \end{align}
where, in the second-to-last line, we re-indexed with $\eta = \alpha - \beta - \ell e_3$ and $p= |\eta| =m-j-\ell$.
We separate lower order terms using
  \begin{align}
    \sum_{p=1}^\infty \sum_{j=0}^\infty \mathcal{G}_{\ell, p, j}
    &=
    \mathcal{G}_{\ell,0,0} 
    + \mathcal{G}_{\ell,0,1}
    + \sum_{p=1}^\infty 
        \mathcal{G}_{\ell,p,0} 
    + \sum_{j=2}^\infty
        \mathcal{G}_{\ell,0,j}
    + \sum_{p,j=1}^\infty
       \mathcal{G}_{\ell,p,j}
   \label{EQ239}
  \end{align}
and claim that 
  \begin{equation}
   \sum_{\ell=1}^\infty 
   (\mathcal{G}_{\ell,0,0} + \mathcal{G}_{\ell,0,1})
   + \sum_{\ell=1}^\infty\sum_{p=1}^\infty 
     \mathcal{G}_{\ell,p,0} 
   + \sum_{\ell=1}^\infty\sum_{j=2}^\infty
     \mathcal{G}_{\ell,0,j}
   \les
   \|u\|_{H^4}^2
   + \tau \|u\|_{H^4} \|u\|_{B_\tau^\co}
   .
   \llabel{EQ240}
  \end{equation}
First, note that 
  \begin{align}
   \begin{split}
    \sum_{\ell=1}^\infty \mathcal{G}_{\ell,0,0}
    &=
    \sum_{\ell=1}^\infty \sum_{|\eta| = 0}\sum_{|\beta|=0} 
      {\eta + \beta + \ell e_3 \choose \beta} \frac{(\eta_3 + \ell)!}{\eta_3!} C^\ell
    \\&\indeq\indeq
      \times M_{\eta + \beta} \eps^{\eta_3 + \beta_3 + \ell}
      \|K^\beta u\|_{H^3} \|K^{\eta + e_3}u\|_{H^3}
      \frac{\tau^{p+j+\ell -1}}{(p+j+\ell -1)!}
    \\&\les
    C\eps \|u\|_{H^4}^2
    \sum_{\ell=1}^\infty 
      \ell(C\eps\tau)^{\ell-1} 
    \les
    \|u\|_{H^4}^2
    ,
   \end{split}
   \llabel{EQ242}
  \end{align}
and, similarly, 
  \begin{equation}
   \sum_{\ell=1}^\infty
     G_{\ell,0,1}
   \les
   \|u\|_{H^4}^2
   .
   \llabel{EQ242a}
  \end{equation}
Next, we obtain
  \begin{align}
   \begin{split}
    \sum_{\ell=1}^\infty \sum_{p=1}^\infty G_{\ell, p, 0}
    &=
    \sum_{\ell =1} \sum_{p=1}^\infty \sum_{|\eta| =p}\sum_{|\beta| =0} {\eta + \beta + \ell e_3 \choose \beta}
      \frac{(\eta_3 + \ell)!}{\eta_3!} C^\ell
    \\&\indeq\indeq 
      \times M_{\eta + \beta} \eps^{\eta_3 + \beta_3 + \ell}
      \|K^{\beta} u\|_{H^3} \|K^{\eta + e_3}u\|_{H^3}
      \frac{\tau^{p+\ell-1}}{(p+\ell-1)!}
    \\&\les
    C\eps \tau \|u\|_{H^3}
    \sum_{\ell=1}^\infty \sum_{p=1}^\infty \sum_{|\eta| =p}
      \Big(
        M_\eta \eps^{\eta_3} \|K^{\eta+e_3} u\|_{H^3}
        \frac{p \tau^{p-1}}{p!}
      \Big)
      (C\eps \tau)^{\ell-1}
      \mathcal{A}_{\eta,\ell,s}
    \\&\les
    \tau \|u\|_{H^4} \|u\|_{B_\tau^\co}
    ,
   \end{split}
   \label{EQ243}
  \end{align}
where, for $|\beta| = 0$,
  \begin{align}
   \begin{split}
    \mathcal{A}_{\eta,\ell,s}
    &=
    M_{\eta+ \beta} M_\eta^{-1} {\eta + \beta + \ell e_3 \choose \beta} \eps^{\beta_3}
    \frac{p!}{(p+\ell-1)!}
    \frac{(\eta_3 + \ell)!}{\eta_3!}
    \frac{1}{p}
    \\&\leq
    \frac{p!}{(p+\ell-1)!} 
    \frac{(p+\ell)!}{p!}
    \frac{1}{p}
    \leq
    \frac{p+\ell}{p}
    \les
    \ell
    .
   \end{split}
   \llabel{EQ244}
  \end{align}
Note that $\sum_{\ell=0}^\infty \ell (C\eps \tau)^\ell < \infty$ justifies the last line of~\eqref{EQ243}. 
Following a similar argument, we also get
  \begin{equation}
    \sum_{\ell = 1}^\infty \sum_{j=2}^\infty 
      \mathcal{G}_{\ell, 0,j}
    \les
    \tau^2 \|u\|_{H^4} \|u\|_{B_\tau^\co}
    .
    \llabel{EQ244b}
  \end{equation}
Finally, we bound the last term in~\eqref{EQ239}. 
We write 
  \begin{align}
   \begin{split}
    \sum_{\ell=1}^\infty \sum_{p=1}^\infty \sum_{j=1}^\infty \mathcal{G}_{\ell,p,j}
    &=
    \sum_{\ell=1}^\infty \sum_{p=1}^\infty \sum_{|\eta| =p} \sum_{j=1}^\infty \sum_{|\beta| =j}
      {\eta+\beta + \ell e_3 \choose \beta} \frac{(\eta_3 + \ell)!}{\eta_3!} C^\ell
    \\&\indeq\indeq
      \times M_{\eta + \beta} \eps^{\eta_3 + \beta_3 + \ell}
      1_{\{\beta_3 = 0\}}\|K^\beta u\|_{H^3} \|K^{\eta + e_3} u\|_{H^3}
      \frac{\tau^{p+j + \ell -1}}{(p+j + \ell -1)!}
    \\&=
    \tau
    \sum_{\ell=1}^\infty \sum_{p=1}^\infty \sum_{|\eta| =p} \sum_{j=1}^\infty \sum_{|\beta|=j}
      \Big(
        M_\beta \eps^{\beta_3}
        \|K^\beta u\|_{H^3}
        \frac{\tau^{j-1}}{(j-1)!}
      \Big)
    \\&\indeq\indeq
      \times \Big(
        M_\eta \eps^{\eta_3} 
        \|K^{\eta+e_3} u\|_{H^3}
        \frac{p\tau^{p-1}}{p!}
      \Big)
      (C\eps \tau)^\ell
      \mathcal{A}_{\eta,\beta,\ell,s}
    \\&\les
    \tau\|u\|_{A_\tau^\co}\|u\|_{B_\tau^\co}
    .
   \end{split}
   \llabel{EQ245}
  \end{align}
For the last line above, 
we note that 
  \begin{align}
   \begin{split}
    \mathcal{A}_{\eta,\beta,\ell,s}
    &=
    M_{\beta+ \eta} M_\beta^{-1} M_\eta^{-1}
    {\eta+ \beta + \ell e_3 \choose \beta}
    \frac{(\eta_3 + \ell)!}{\eta_3!} 
    \frac{(j-1)! p!}{(p+j + \ell -1)!}
    \frac{1}{p}
    \\&\leq
    {p+j + \ell \choose j} 
    \frac{(p + \ell)!}{(p-1)!} 
    \frac{(j-1)! p!}{(p+j+\ell -1)!} 
    \frac{1}{p}
    \leq 
    \frac{p+j+\ell}{jp}
    \les \ell
    ,
   \end{split}
   \llabel{EQ246}
  \end{align}
completing the proof.
\end{proof}

\subsection{Pressure estimates}\label{sec04b}
In this section, we prove Lemma~\ref{L12}, which bounds $\|\nabla p\|_{A_\tau^\co}$.
Before this, we provide a lemma which controls the product~$\phi f$.

\begin{Lemma}
\label{L16}
Let $f \in C^\infty$. 
There exists a positive universal constant $C$ such that 
  \begin{equation}
   \sum_{m=2}^\infty \sum_{\substack{|\alpha| =m \\ \alpha_3 \geq 2}}
     M_\alpha \eps^{\alpha_3}
     \|K^{\alpha - 2e_3} (\phi f) \|_{H^3}
     \frac{\tau^{m-1}}{(m-1)!}
   \leq
   \frac{C}{1 - C\eps \tau_0}
   \sum_{m=2}^\infty\sum_{\substack{|\alpha| = m \\ \alpha_3 \geq 2}} 
     M_\alpha \eps^{\alpha_3}
     \|K^{\alpha - 2e_3} f\|_{H^3}
     \frac{\tau^{m-1}}{(m-1)!}
  .
  \llabel{EQ300}
  \end{equation}
\end{Lemma}

\begin{proof}
By the Leibniz rule and Lemma~\ref{L14}, we have
  \begin{align}
   \begin{split}
    &
    \sum_{m=2}^\infty \sum_{\substack{|\alpha| = m \\ \alpha_3 \geq 2}}
      M_\alpha \eps^{\alpha_3}
      \|K^{\alpha_3 -2e_3} (\phi f) \|_{H^3}
      \frac{\tau^{m-1}}{(m-1)!}
    \\&\indeq\indeq=
    \sum_{m=2}^\infty \sum_{\substack{|\alpha| = m \\ \alpha_3 \geq 2}}
    \sum_{j =0}^{\alpha_3 -2} {\alpha_3 -2 \choose j}
      M_\alpha \eps^{\alpha_3}
      \|L^j \phi\|_{W^{3,\infty}} \|K^{\alpha - (j+2)e_3} f\|_{H^3}
      \frac{\tau^{m-1}}{(m-1)!}
    \\&\indeq\indeq\leq
    \sum_{m=2}^\infty \sum_{\substack{|\alpha| = m \\ \alpha_3 \geq 2}} 
    \sum_{j=0}^{\alpha_3 - 2} \frac{(\alpha_3 -2)!}{(\alpha_3 - j -2 )!} C^j
      \|K^{\alpha - (j+2)e_3} f\|_{H^3}
      \frac{\tau^{m-2}}{(m-1)!}
    \\&\indeq\indeq =
    \sum_{j=0}^\infty \sum_{p =2}^\infty \sum_{|\eta| = p}
      \Big(
        M_\eta \eps^{\eta_3}
        \|K^{\eta - 2e_3} f\|_{H^3}
        \frac{\tau^{p-1}}{(p-1)!}
      \Big)
      (C\eps \tau)^j
      \frac{(\eta_3 + j- 2)!}{(\eta_3 -2)!}
      \frac{(p-1)!}{(p+j - 1)!}
    ,
   \end{split}
   \llabel{EQ301}
  \end{align}
where we denote $\eta = \alpha - j e_3$ and $p= |\eta| = m - j$.
From here, we bound 
  \begin{equation}
   \frac{(\eta_3 +j -2)!}{(\eta_3 -2)!} 
   \frac{(p-1)!}{(p+j-1)!}
   \leq
   \frac{(p+j -2)!}{(p-2)!}
   \frac{(p-1)!}{(p+j-1)!}
   \leq 
   \frac{p-1}{p+j -1}
   \leq 1
   .
   \llabel{E302}
  \end{equation}
Thus, noting that $\sum_{j =0}^\infty (C\eps \tau)^j \leq C(1-C\eps \tau)^{-1}$, the proof is complete.
\end{proof}

Here, we make the following technical note, which transfers two conormal derivatives to tangential derivatives.
Effectively, this allows us to absorb the highest order terms in the pressure estimate using the weight~$\eps^{\alpha_3}$.
Namely, for $ n \geq 2$, 
  \begin{align}
   \begin{split}
    Z_n \partial_k p
    &=
    Z_{n-2} (\phi^2 \partial_3^2 (\partial_k p))
    + Z_{n-2}( \phi' Z \partial_k p)
    =
    Z_{n-2}(\phi^2 (\Delta - \Delta')\partial_k p)
    + Z_{n-2} (\phi' Z \partial_k p)
    .
   \end{split}
   \llabel{EQ303}
  \end{align}
Hence, if $\alpha_3 \geq 2$, 
  \begin{equation}
   K^\alpha \partial_k p
   =
   K^{\alpha - 2e_3} (\phi^2 \partial_k (\partial_j v_i \partial_i v_j))
   -\sum_{j=1,2} K^{\alpha +2e_j - 2e_3} (\phi^2(\partial_k p))
   + K^{\alpha - 2e_3} (\phi' Z \partial_k p)
   .
   \label{EQ304}
  \end{equation}
Now, we prove the main pressure estimate.  

\begin{proof}[Proof of Lemma~\ref{L12}]
It suffices to show that 
  \begin{equation}
   \|\nabla p\|_{A_\tau^\co} 
   \les
   \|u\|_{H^4}^2
   + \tau\|u\|_{A_\tau^\co} \|u\|_{B_\tau^\co}
   + (\eps^2 +\eps) \|\nabla p\|_{A_\tau^\co}
   ,
   \llabel{EQ304a}
  \end{equation}
since taking $\eps > 0$ sufficiently small allows us to absorb the higher-order term into the left-hand side, implying~\eqref{EQ208}.

First, we claim that 
  \begin{equation}
   \sum_{m=1}^\infty \sum_{\substack{|\alpha| =m \\ \alpha_3= 0,1}}
     M_\alpha \eps^{\alpha_3} 
     \|K^\alpha \nabla p\|_{H^3}
     \frac{\tau^{m-1}}{(m-1)!}
     \les
     \|u\|_{H^4}^2 + \tau \|u\|_{A_\tau^\co} \|u\|_{B_\tau^\co}
   .  
   \label{EQ305}
  \end{equation}
We only establish~\eqref{EQ305} for $\alpha_3 = 1$ as the case when $\alpha_3 = 0$ is similar. 
For $\alpha_3 =1$, elliptic regularity gives
  \begin{equation}
   \|K^\alpha \partial_k p\|_{H^3}
   =
   \|\partial^{\alpha'} p\|_{H^5}
   \les
   \|\partial^{\alpha'} (\partial_j v_i \partial_i v_j)\|_{H^3}
   \leq
   \|\partial^{\alpha'} (\bar{\partial} v_h \bar{\partial} v_h)\|_{H^3}
   + \|\partial^{\alpha'} (\bar{\partial} v_3 \partial_3 v_h)\|_{H^3}
   ,
   \label{EQ306}
  \end{equation}
where we used incompressibility in the last inequality. 
For the purely tangential term in~\eqref{EQ306}, we obtain
  \begin{align}
   \begin{split}
    &
    \sum_{m=1}^\infty \sum_{|\alpha| =m} 
      M_\alpha \eps^{\alpha_3} 
      \|\partial^{\alpha'}(\bar{\partial} v_h \bar{\partial} v_h)\|_{H^3}
      \frac{\tau^{m-1}}{(m-1)!}
    \\&\indeq\indeq= 
    \sum_{m=1}^\infty \sum_{|\alpha'|=m-1} \sum_{j=0}^{m-1} \sum_{|\beta'| = j} {\alpha' \choose \beta'}
      M_\alpha \eps^{\alpha_3} 
      \|\partial^{\beta'} \bar{\partial v}\|_{H^3} \|\partial^{\alpha' - \beta'} \bar{\partial}v\|_{H^3}
      \frac{\tau^{m-1}}{(m-1)!}
    \\&\indeq\indeq\les
    \|u\|_{H^4}^2
    + \tau \|u\|_{H^4}
    \sum_{m=2}^\infty \sum_{|\alpha'| = m-1} \sum_{|\beta'| =m-1} 
      \Big( 
        M_\beta \eps^{\beta_3} \|\partial^{\beta'} \bar{\partial} v\|_{H^3}
        \frac{(m-1)\tau^{m-2}}{(m-1)!}
      \Big)
      \mathcal{A}_{\alpha,\beta,s}^1
    \\&\indeq\indeq\indeq\indeq
    +\sum_{m=2}^{\infty} \sum_{|\beta'| = m-1} \sum_{j=0}^{m-2} \sum_{|\beta| = j}
      \Big(
        M_\beta \eps^{\beta_3} 
        \|\partial^{\beta'} \bar{\partial}v\|_{H^3}
        \frac{\tau^j}{j!}
      \Big)
    \\&\indeq\indeq\indeq\indeq\indeq\indeq
      \times \Big( 
        M_{\alpha-\beta} \eps^{\alpha_3 - \beta_3} 
        \|\partial^{\alpha' - \beta'} \bar{\partial} v\|_{H^3}
        \frac{(m-j-1) \tau^{m-j-2}}{(m-j-1)!}
      \Big)
      \mathcal{A}_{\alpha,\beta,s}^2
    \\&\indeq\indeq
    \les
    \|u\|_{H^4}^2
    + \tau \|u\|_{B_\tau^\co}
      (\|u\|_{H^4} + \|u\|_{A_\tau^\co})
    ,
   \end{split}
   \llabel{EQ307}
  \end{align}
where, for $|\alpha'| = |\beta'| = m-1$, 
  \begin{equation}
   \mathcal{A}_{\alpha,\beta,s}^1
   =
   M_\alpha M_\beta^{-1} {\alpha'\choose \beta'} 
   \frac{1}{m-1} 
   \eps^{\alpha_3 - \beta_3}
   \les 
   \frac{1}{m-1} 
   \leq 1
   ,
   \llabel{EQ308}
  \end{equation}
and
  \begin{equation}
    \mathcal{A}_{\alpha,\beta,s}^2
    = 
    M_\alpha M_\beta^{-1} M_{\alpha-\beta}^{-1} {\alpha' \choose \beta'}
    \frac{j! (m-j-1)!}{(m-1)!}
    \frac{1}{m-j-1}
    \leq
    {m-1 \choose j}
    {m-1 \choose j}^{-s}
    \frac{1}{m-j-1}
    \leq
    1
    .
   \llabel{EQ309}
  \end{equation}
For the normal term in~\eqref{EQ306}, its corresponding sum gives
  \begin{align}
   \begin{split}
    &
    \sum_{m=1}^\infty \sum_{|\alpha'| = m-1} \sum_{j=0}^{m-1} \sum_{|\beta'| = j} {\alpha' \choose \beta'}
      M_\alpha \eps^{\alpha_3}
      \|\partial^{\beta'} \bar{\partial} v_3 \partial^{\alpha' - \beta'} \partial_3 v_h\|_{H^3}
      \frac{\tau^{m-1}}{(m-1)!}
    \\&\indeq\indeq\les
    \|u\|_{H^4}^2
    + \tau \|u\|_{H^4} 
      \sum_{m=2}^\infty \sum_{|\alpha'| = m-1} \sum_{|\beta'| = m-1}
        \Big( 
           M_\beta \eps^{\beta_3} 
           \|\partial^{\beta'} \bar{\partial} v_3\|_{H^3}
           \frac{(m-1)\tau^{m-2}}{(m-1)!}
        \Big)
        \mathcal{A}_{\alpha,\beta,s}^1
    \\&\indeq\indeq\indeq\indeq
     + \sum_{m=2}^\infty \sum_{|\alpha'| = m-1} \sum_{j=0}^{m-2} \sum_{|\beta'| = j}{\alpha' \choose \beta'}
         M_\alpha \eps^{\alpha_3} 
         \|\partial^{\beta'} \bar{\partial}^2 v\|_{H^3}\|\partial^{\alpha' - \beta'} Z v\|_{H^3}
         \frac{\tau^{m-1}}{(m-1)!}
     \\&\indeq\indeq\les
     \|u\|_{H^4}^2
     + \tau \|u\|_{H^4} \|u\|_{B_\tau^\co}
     \\&\indeq\indeq\indeq\indeq
     + \sum_{m=2}^\infty \sum_{|\alpha'| = m-1}\sum_{j=0}^{m-2} \sum_{|\beta'| = j}
       \Big( 
         M_\beta \eps^{\beta_3}
         \|\partial^{\beta'} \bar{\partial}^2 u\|_{H^3}
         \frac{(j+1)\tau^{j}}{(j+1)!}
       \Big)
     \\&\indeq\indeq\indeq\indeq\indeq\indeq
       \times \Big(
         M_{\alpha - \beta} \eps^{\alpha_3 - \beta_3} 
         \|\partial^{\alpha' - \beta'} Z v\|_{H^3}
         \frac{\tau^{m-j-1}}{(m-j-1)!}
       \Big)
       \mathcal{A}_{\alpha,\beta,s}^2
     \\&\indeq\indeq\les
     \|u\|_{H^4}^2
     + \tau \|u\|_{B_\tau^\co} 
       (\|u\|_{H^4} + \|u\|_{A_\tau^\co})
    ,
   \end{split}
   \label{EQ310}
  \end{align}
where, for the third term after the first inequality, we conormalized the product.
To justify the last line of~\eqref{EQ310}, we bound 
  \begin{equation}
   \mathcal{A}_{\alpha,\beta,s}^1
   = 
   M_\alpha M_{\beta} {\alpha' \choose \beta'} \frac{1}{m-1} 
   \leq 
   1  
   \llabel{EQ311}
  \end{equation}
and 
  \begin{equation}
    \mathcal{A}_{\alpha,\beta,s}^2
    =
    M_\alpha M_\beta^{-1} M_{\alpha-\beta}^{-1} 
    {\alpha' \choose \beta'}
    \frac{(j+1)! (m-j-1)!}{(m-1)!}
    \frac{1}{j+1}
    \leq
    {m-1 \choose j}^{1-s} (j+1)^{s-1}
    \leq 1
    .
   \llabel{EQ312}
  \end{equation}
This completes the proof of~\eqref{EQ305} when $\alpha_3 = 1$.

Now, assume that $\alpha_3 \geq 2$. 
Recalling the decomposition~\eqref{EQ304}, we use Lemma~\ref{L16} to bound
  \begin{align}
   \begin{split}
    & 
    \sum_{m=2}^\infty \sum_{\substack{|\alpha| =m \\ \alpha_3 \geq 2}}
      M_\alpha \eps^{\alpha_3} 
      \|K^{\alpha + 2\bar{e} - 2e_3}(\phi^2 \partial_k p)\|_{H^3}
      \frac{\tau^{m-1}}{(m-1)!}
    \\&\indeq\indeq\leq
    \frac{C\eps^2}{1-C\eps \tau_0}
    \sum_{m=2}^\infty \sum_{\substack{|\alpha| = m \\ \alpha_3 \geq 2}}
      M_\alpha \eps^{\alpha_3 -2} 
      \|K^{\alpha_3 +2\bar{e} - 2e_3} \partial_k p\|_{H^3}
      \frac{\tau^{m-1}}{(m-1)!}
    \les
    \eps^2 \|\nabla p\|_{A_\tau^\co}
   \end{split}
   \llabel{EQ313}
  \end{align}
and 
  \begin{align}
   \begin{split}
    &
    \sum_{m=2}^\infty \sum_{\substack{|\alpha| = m\\ \alpha_3 \geq 2 }}
      M_\alpha \eps^{\alpha_3} 
      \|K^{\alpha -2e_3} (\phi' Z \partial k p)\|_{H^3}
      \frac{\tau^{m-1}}{(m-1)!}
    \\&\indeq\indeq \leq
    \frac{C}{1-C\eps \tau_0}
    \sum_{m=2} \sum_{\substack{|\alpha|=m \\ \alpha_3 \geq 2}}
      M_\alpha \eps^{\alpha_3} 
      \|K^{\alpha - e_3} \partial_k p\|_{H^3}
      \frac{\tau^{m-1}}{(m-1)!}
    \les
    \eps \|\nabla p\|_{A_\tau^\co}
    .
   \end{split}
   \llabel{EQ314}
  \end{align}

Next, we claim that 
  \begin{equation}
   \sum_{m=1}^\infty \sum_{|\alpha|=m, \alpha_3 \geq 2} 
     M_\alpha \eps^{\alpha_3} 
     \|K^\alpha \nabla p\|_{H^3}
     \frac{\tau^{m-1}}{(m-1)!}
   \les
   \tau \|u\|_{A_\tau^\co} \|u\|_{B_\tau^\co}
   .
   \llabel{EQ314a}
  \end{equation}
For the first term in~\eqref{EQ304}, we record some preliminary simplifications.
Recall that, formally, 
  \begin{equation}
   \partial_k(\partial_j v_i \partial_i v_j)
   =
   \partial_k (\bar{\partial} v_h \bar{\partial} v_h)
   + \partial_k (\bar{\partial}v_3 \partial_3 v_h)
   .
   \llabel{EQ315}
  \end{equation}
If $k=3$, we write
  \begin{align}
   \begin{split}
    \phi^2 \partial_3 (\partial_j v_i \partial_i v_j)
    &=
    \phi Z (\bar{\partial}v_h \bar{\partial} v_h)
    + \phi Z(\bar{\partial} v_3 \partial_3 v_h)
     =
    \phi \bar{\partial} Z v_h \bar{\partial}v_h
    + \bar{\partial} v Z_2 v
    - \phi' \bar{\partial} v Zv
    ,
   \end{split}
   \label{EQ316}
  \end{align}
and if $k=1,2$, we obtain
  \begin{equation}
   \phi^2 \partial_k(\partial_j v_i \partial_i v_j)
   =
   \phi \bar{\partial}^2 v Zv
   + \phi \bar{\partial} v Z \bar{\partial}v
   + \phi^2 \bar{\partial}^2 v \bar{\partial}v
   .
   \label{EQ317}
  \end{equation}
While our goal is to estimate the analytic norm of $\phi^2 \partial_k (\partial_j v_i \partial_i v_j)$, by considering the energy levels present in~\eqref{EQ316} and~\eqref{EQ317}, it suffices to only estimate the analytic norm of
  \begin{equation}
   P(\phi)\bar{\partial} Z \bar{\partial} v
   + \phi' \bar{\partial}v Z v
   ,
   \llabel{EQ318}
  \end{equation}
where $P(\phi) = 1 + \phi + \phi^2$.
For example, there are no differences in the estimates of $\bar{\partial}^2 v \bar{\partial}v$ versus $Z_2 v \bar{\partial}v$.
Moreover, we omit the estimate for the analytic norm of $\phi' \bar{\partial}v Zv$ as it is a lower-order term and $\phi' = (1-\phi)^2$; it admits nearly the same estimate as $P(\phi) \bar{\partial}Zv \bar{\partial}v$.

We claim that 
  \begin{equation}
   \sum_{m=2}^\infty \sum_{\substack{|\alpha| =m \\ \alpha_3 \geq 2}}
     M_\alpha \eps^{\alpha_3}
     \|K^{\alpha - 2e_3} ( P(\phi) \bar{\partial} Zv \bar{\partial} v)\|_{H^3}
     \frac{\tau^{m-1}}{(m-1)!}
   \les
   \tau \|u\|_{A_\tau^\co} \|u\|_{B_\tau^\co}
   .
   \label{EQ319}
  \end{equation}
By Lemma~\ref{L16}, we may neglect the polynomial $P(\phi)$ in~\eqref{EQ319}.
Beginning with a change of index given by $\gamma = \alpha - 2e_3$, we have
  \begin{align}
   \begin{split}
    &
    \sum_{m=2}^\infty \sum_{\substack{|\alpha| =m \\ \alpha_3 \geq 2}}
      M_\alpha \eps^{\alpha_3}
      \|K^{\alpha-2e_3} (\bar{\partial}Zv \bar{\partial v})\|_{H^3}
      \frac{\tau^{m-1}}{(m-1)!}
    \\&\indeq\indeq=
    \eps^2 
    \sum_{m=2}^\infty \sum_{|\gamma| = m-2}
      M_\gamma \eps^{\gamma_3}
      \|K^\gamma (\bar{\partial} Z v \bar{\partial} v)\|_{H^3}
      \frac{\tau^{m-1}}{(m-1)!}
    \\&\indeq\indeq\les
    \eps^2
    \sum_{m=2}^\infty \sum_{|\gamma| = m-2} \sum_{j=0}^{m-2} \sum_{|\beta|= j}
      \Big(
        M_\beta \eps^{\beta_3} 
        \|K^\gamma \bar{\partial} Z v\|_{H^3}
        \frac{(j+1)\tau^j}{(j+1)!}
      \Big)
    \\&\indeq\indeq\indeq\indeq\times
      \Big(
        M_{\gamma- \beta} \eps^{\gamma_3 - \beta_3}
        \|K^{\gamma - \beta} \bar{\partial} v\|_{H^3}
        \frac{\tau^{m-j-1}}{(m-j-2)!}
      \Big)
      \mathcal{A}_{\gamma, \beta, s}
    \\&\indeq\indeq\les
    \tau \eps^2 \|u\|_{A_\tau^\co} \|u\|_{B_\tau^\co}
    ,
   \end{split}
   \llabel{EQ320}
  \end{align}
where, for $m \geq 2$ and $0\leq j \leq m-2$,
  \begin{align}
   \begin{split}
    \mathcal{A}_{\gamma,\beta,s}
    &=
    M_\gamma M_\beta^{-1} M_{\gamma-\beta}^{-1} {\gamma \choose \beta}
    \frac{(j+1)! (m-j-2)!}{(m-1)!}
    \frac{1}{j+1}
    \\&\leq
    {m-2 \choose j} 
    {m-1\choose j+1}^{-s}
    \frac{1}{j+1}
    = 
    {m-1 \choose j+1}^{1-s} 
    \frac{j+1}{m-1}
    \frac{1}{j+1}
    \\&\leq
    \frac{1}{m-1}
    \leq 1
    .
   \end{split}
   \llabel{EQ321}
  \end{align}
This completes the proof.
\end{proof}

\section{Proof of Theorem~\ref{T03}: Equivalence in norm of $X_\tau^\co$ and $P_\tau^\co$}\label{sec05}
In this section, we prove Theorem~\ref{T03}, which compares the norms of $X_\tau^\co$ and~$P_\tau^\co$.
Recall that our conormal operators are defined as $Z_n = \phi^n \partial_3^n $ and $L^n = (\phi \partial_3)^n$.
The conormal analytic norms are given by
  \begin{equation}
    \|u\|_{X_\tau^\co}
    =
    \sum_{m=1}^\infty
     \sum_{|\alpha| = m} 
      M_\alpha
      \|J^\alpha u\|_{H^3}
    \frac{\tau^{m-1}}{(m-1)!}
   \llabel{EQ322}
  \end{equation}
and
  \begin{equation}
   \|u\|_{P_\tau^\co}
   =
   \sum_{m=1}^\infty
    \sum_{|\alpha| =m}
     M_\alpha
     \|K^\alpha u\|_{H^3}
    \frac{\tau^{m-1}}{(m-1)!}
    ,
   \llabel{EQ323}
  \end{equation}
where $J^\alpha = \partial^{\alpha'}Z_{\alpha_3}$ and $K^\alpha = \partial^{\alpha'}L^{\alpha_3}$.

\begin{proof}[Proof of Theorem~\ref{T03}]
Let $0 < \tau < (4C_1)^{-1}$.
First, we establish the second inequality in~\eqref{EQ30a}. 
More precisely, we claim that
  \begin{equation}
   \|u\|_{P_\tau^\co} 
   \les
   \frac{1}{\sqrt{1-4C\tau}}
   \|u\|_{X_{\rho_1}^\co}
   ,
   \label{EQ324}
  \end{equation}
where $\rho_1(\tau) = 2\tau (1+\sqrt{1-4C\tau})^{-1}$. 
Indeed, Proposition~\ref{P01} and Remark~\ref{R07} yield
  \begin{equation}
    \|u\|_{P_\tau^\co} 
    =
    \sum_{m=1}^\infty \sum_{|\alpha| = m}
      M_\alpha 
      \|K^\alpha u\|_{H^3}
      \frac{\tau^{m-1}}{(m-1)!}
    \les
    \sum_{m=1}^\infty \sum_{|\alpha| = m} \sum_{k=0}^{\alpha_3 -1}
      C_1^k \frac{(\alpha_3 + k-1)!}{k!(\alpha_3 - k-1)!}
      \|J^{\alpha-ke_3} u\|_{H^3}
      \frac{\tau^{m-1}}{(m-1)!}
      .
   \llabel{EQ325}
  \end{equation}
Now, denote $\eta = \alpha - ke_3$ and $p = |\eta| =m-k$.
This gives us 
  \begin{equation}
   \|u\|_{P_\tau^\co}
   \les
   \sum_{p=1}^\infty \sum_{|\eta| =p, \eta_3 \geq 1}\sum_{k=0}^\infty
     \Big(
       M_\eta 
       \|J^\eta u\|_{H^3}
       \frac{\tau^{p-1}}{(p-1)!}
     \Big)
     (C_1\tau)^k\mathcal{A}_{\eta,k}
   ,
   \label{EQ326}
  \end{equation}
where
  \begin{align}
   \begin{split}
    \mathcal{A}_{\eta,k}
    &=
    \frac{(p-1)!}{(p+k-1)!} \frac{(\eta_3 + 2k - 1)!}{k!(\eta_3 -1)!}
    \leq
    \frac{(p-1)!}{(p+k-1)!} \frac{(p+2k-1)!}{k!(p-1)!}
    \\&= 
    \frac{(p+2k-1)!}{k!(p+k-1)!}
    =
    {p+2k -1 \choose k}
    .
   \end{split}
   \llabel{EQ327}
  \end{align}
By equation~(2.5.15) in~\cite{W}, we have
  \begin{equation}
   \sum_{k=0}^\infty {p+2k-1 \choose k} (C\tau)^k
   =
   \frac{1}{\sqrt{1-4C\tau}}
   \left( \frac{2}{1+ \sqrt{1-4C\tau}}\right)^{p-1}
   .
   \label{EQ328}
  \end{equation}
Substituting~\eqref{EQ328} into~\eqref{EQ326} completes the proof of~\eqref{EQ324}.

Now, we prove the first inequality in~\eqref{EQ30a}. 
Let $0 < \rho < (4C_2)^{-1} $. 
By a similar proof as above but with Proposition~\ref{P02},
we obtain
  \begin{equation}
   \|u\|_{X_\rho^\co} 
   \les
   \frac{1}{\sqrt{1-4C_2 \rho}}
   \|u\|_{P_{\tau(\rho)}^\co}
   ,
   \label{EQ329}
  \end{equation}
where $\tau(\rho) = 2\rho (1+\sqrt{1-4C_2\rho})^{-1}$.
Equivalently, we may write $\rho$ in terms of $\tau$ as 
  \begin{equation}
   \rho(\tau)
   =
   \tau - C_2 \tau^2
   .
   \llabel{EQ330}
  \end{equation}
Since $\rho < (4C_2)^{-1}$ is equivalent to $\tau < (2C_2)^{-1}$, the bound~\eqref{EQ329} then yields
  \begin{equation}
   (1-2C_2 \tau )
   \|u\|_{X_{\rho(\tau)}^\co}
   \les
   \|u\|_{P_\tau^\co}
   \llabel{EQ331}
  \end{equation}
for $\tau < (2C_2)^{-1}$.
This completes the proof. 
\end{proof}

\appendix
\section{}\label{sec06}
The purpose of this section is to compare the operators $Z_k = \phi^k \partial_x^k$ and $L^k = (\phi \partial_x)^k$, where $\phi(x) = x(1+x)^{-1}$ and $x > 0$.
Our first identity expands $(\phi \partial_x)^n$ in terms of~$\phi^k \partial_x^k$. 
\begin{Proposition}
\label{P01}
We have
  \begin{equation}
   (\phi \partial_x)^n 
   =
   \sum_{k=1}^n 
     d_k^n(x)
     \phi^k \partial_x^k 
   ,
   \label{EQ401}
  \end{equation}
where
  \begin{equation}
   d_k^n(x)
   =
   \sum_{\lambda \vdash n-k}
     c_{\lambda}^n 
     \left( \prod_{i=1}^{\ell(\lambda)} \lambda_i! \right)
     (-1)^{n-k + \ell(\lambda)}
     \phi^{n-k-\ell(\lambda)} (1-\phi)^{n-k + \ell(\lambda)}
   .
   \label{EQ401a}
  \end{equation}
In~\eqref{EQ401a}, $\lambda \vdash n-k$ means $\lambda = (\lambda_1,\dots, \lambda_{\ell})$ is a partition of $n-k$ and $\ell = \ell(\lambda)$ denotes the length of the partition. 
Moreover, 
  \begin{equation}
   \|d_k^n\|_{W^{3,\infty}}
   \les
     C^{n-k}
     \frac{(2n-k-1)!}{(n-k)!(k-1)!} 
   ,
   \label{EQ402}
  \end{equation}
where $C$ is a positive universal constant less than~$1$.
\end{Proposition}
\begin{remark}
\label{R07}
In application to Theorem~\ref{T03}, we use Proposition~\ref{P01} to write
  \begin{equation}
   \|(\phi \partial_3)^n f\|_{H^3}
   \les
   \sum_{k=1}^n 
     C^{n-k}
     \frac{(2n-k-1)!}{(n-k)!(k-1)!} 
     \|\phi^k \partial_3^kf\|_{H^3}
   =
   \sum_{k=0}^{n-1} 
     C^k 
     \frac{(n+k-1)!}{k!(n-k -1)!}
     \|\phi^{n-k} \partial_3^{n-k} f\|_{H^3}
     .
   \llabel{EQ402a}
  \end{equation}
\end{remark}
Next, we state the inverse identity to~\eqref{EQ401}.
\begin{Proposition}
\label{P02}
We have
  \begin{equation}
   \phi^n \partial_x^n
   =
   \sum_{k=1}^n 
     a_k^n(x) 
    (\phi \partial_x)^k 
   ,
   \label{EQ403}
  \end{equation}
where
  \begin{equation}
   a_k^n(x)
   =
   \sum_{\lambda \vdash n-k} 
     c_{\lambda}^n 
     \Bigg(
       \prod_{i=1}^{\ell(\lambda)} 
       \phi(\phi \partial_x)^{\lambda_i}\left( \frac{1}{\phi}\right)
     \Bigg)
     .
   \llabel{EQ403a}
  \end{equation}
Moreover, 
  \begin{equation}
   \|a_k^n\|_{W^{3,\infty}}
   \les
    C^{n-k}
    \frac{(2n-k-1)!}{(n-k)!(k-1)!}
   ,
   \llabel{EQ404}
  \end{equation}
where $C$ is a positive universal constant.
\end{Proposition}
\begin{remark}
\label{R05}
We can think of the identities~\eqref{EQ401} and~\eqref{EQ403} in terms of $n \times n$ matrix multiplication. 
Hence, the triangular matrices $(d_k^m)_{k,m =1}^n$ and $(a_k^m)_{k,m=1}^n$ are inverse to each other.
\end{remark}
Before we prove Proposition~\ref{P01}, we provide some exposition related to normally ordered forms.
Define $\Lambda = \sum_{i\geq 0} y_{i+1}\partial_{y_i}$. 
We recursively define the polynomials $U_n = U_n(y_0,\dots, y_{n-1},t)$ as 
  \begin{equation}
   U_0 = 1
   \quad\quad
   \text{and}
   \quad\quad
   \forall n \geq 0,\,
   U_{n+1} = y_0 U_n t + y_0 \Lambda(U_n)
   .
   \label{EQ404a}
  \end{equation}
Furthermore, 
for $\lambda = (\lambda_1, \dots, \lambda_{\ell})$, 
we define $y_\lambda = \prod_{i=1}^{\ell(\lambda)} y_{\lambda_i}$.
For a more extensive and abstract exposition, see~\cite{BLR}.
Below is a lemma which summarizes Theorem~2.3 and Proposition~2.6 of the same paper in our specific case.

\begin{Lemma}
\label{L17}
Assume $n \geq 1$. 
There exist positive integers $c_\lambda^n$, where $\lambda$ runs through the set of integer partitions of size $0 \leq |\lambda| < n$, such that 
  \begin{equation}
    U_n(y_0,y_2,\dots, y_{n-1}; t)
    =
    \sum_{k=1}^n \sum_{\lambda \vdash n-k}
      c_\lambda^n y_0^{n-\ell(\lambda)} y_\lambda
      t^k
   .
   \label{EQ404b}
  \end{equation}
Moreover, for a smooth function $h$ on $\mathbb{R}_+$, 
  \begin{equation}
   (h \partial_x)^n
   =
   U_n|_{y_i = h^{(i)}, t = \partial_x}
   =
   \sum_{k=1}^n \sum_{\lambda \vdash n-k} 
     c_\lambda^n h^{n-k-\ell(\lambda)}
     \left(\prod_{i=1}^{\ell(\lambda)} h^{(\lambda_i)} \right)
       h^k\partial_x^k
   .
   \llabel{EQ404c}
  \end{equation}
\end{Lemma}
\begin{remark}
\label{R06}
The constants $c_\lambda^n$ satisfy certain recursive relations which are given in~\cite{BLR}. 
However, they are unnecessary for the proof of Propositions~\ref{P01} and~\ref{P02} and are omitted. 
For the reader's interest, however, we note that 
  \begin{equation}
   \sum_{\lambda \vdash n-k} c_\lambda^n  
   =
   {n \brace j}
   ,
  \llabel{EQ404d}
  \end{equation}
where ${n \brace j}$ denotes the Stirling numbers of the first kind.
This can be shown by identifying it with the polynomial~$U_n|_{y_i =1}(t)$.
For a full proof of this as well as other combinatorial interpretations of $c_\lambda^n$, see~\cite[Section~4.1]{BLR}.
\end{remark}

Before we prove Proposition~\ref{P01}, we give a lemma which establishes an exact identity for the constant coefficients in~\eqref{EQ401a}.

\begin{Lemma}
\label{L18}
Let $n \geq 1$ and $1\leq k \leq n$. 
Then
  \begin{equation}
   S_{n,k}
   :=
   \sum_{\lambda \vdash n-k}
     c_\lambda^n 
     \prod_{i=1}^{\ell(\lambda)} \lambda_i!
   =
   \frac{(2n-k-1)!}{2^{n-k} (n-k)! (k-1)!}
   .
   \label{EQ405}
  \end{equation}
\end{Lemma}

\begin{proof}[Proof of Lemma~\ref{L18}]
It suffices to show that
  \begin{equation}
   S_{n+1,k}
   =
   (2n-k)S_{n,k} + S_{n,k-1}
   \comma
   1\leq k \leq n+1
   ,
   \label{EQ406}
  \end{equation}
where $S_{n,n} =S_{1,1} = 1$ and $S_{n,k} = 0$ for $k > n$. 
Assuming~\eqref{EQ406} holds, we prove~\eqref{EQ405} by induction.
For $n =1$, we have $S_{1,1} = 1$.
Now assume~\eqref{EQ405} holds for a fixed $n\geq 1$ and $1 \leq k \leq n$.
The recursion formula gives
  \begin{align}
   \begin{split}
    S_{n+1,k} 
    &= 
    \frac{(2n-k)!}{2^{n-k} (n-k)!(k-1)!}
    + \frac{(2n-k+1)!}{2^{n-k+1} (n-k+1)! (k-2)!}
    \\&=
    \frac{(2n-k)!(2)(n+1-k)}{2^{n+1-k} (n+1-k)!(k-1)!}
    + \frac{(2n-k+1)! (k-1)}{2^{n+1-k} (n+1-k)!(k-1)!}
    \\&=
    \frac{(2(n+1)-k-1)!}{2^{n+1-k} (n+1-k)!(k-1)!}
    .
   \end{split}
   \llabel{EQ406a}
  \end{align}
To prove~\eqref{EQ406}, we begin by identifying $S_{n,k}$ with the polynomial $U_n$ evaluated at $y_i = i!$.
From~\eqref{EQ404b}, we have
  \begin{equation}
   U_n|_{y_i = i!}(t)
   =
   \sum_{k=1}^n \sum_{\lambda \vdash n-k} 
     c_\lambda^n 
     \left(\prod_{i=1}^{\ell(\lambda)} \lambda_i!\right)
     t^k
   =
   \sum_{k=1}^n S_{n,k} t^k
   .
   \label{EQ407}
  \end{equation}
By using the recursive definition of $U_n(y,t)$, see~\eqref{EQ404a}, we derive 
  \begin{equation}
    U_{n+1}
    =
    y_0 U_n t
    + y_0 \Lambda U_n
    =
    \sum_{k=1}^n \sum_{\lambda \vdash n-k} 
      c_{\lambda}^n
      y_0^{n-\ell(\lambda) + 1} y_\lambda t^{k+1}
      + \sum_{k=1}^n \sum_{\lambda \vdash n-k}
          c_\lambda^n y_0
          \Lambda( y_0^{n-\ell(\lambda)} y_\lambda)
          t^k
    ,
   \label{EQ408}
  \end{equation}
where, recalling that $\Lambda := \sum_{i\geq 0} y_{i+1}\partial_{y_i}$,
  \begin{equation}
   y_0\Lambda (y_0^{n-\ell(\lambda)} y_{\lambda})
   =
   (n-\ell(y))y_1 y_0^{n-\ell(\lambda)} y_\lambda
   + y_0^{n-\ell(y) + 1} 
   \sum_{j=1}^{\ell(\lambda)} 
     y_{\lambda_j + 1}
     \prod_{i=1, i\neq j}
       y_{\lambda_i}
   .
   \label{EQ409}
  \end{equation}
Furthermore, taking $y_i = i!$ yields
  \begin{equation}
   \sum_{j=1}^{\ell(\lambda)} 
     y_{\lambda_j + 1}
     \prod_{i=1, i\neq j}
       y_{\lambda_i}
   =
   \sum_{j=1}^{\ell(\lambda)}
     (\lambda_j + 1)!
     \prod_{i=1, i\neq j}
       \lambda_i!
   =
   (n-k + \ell(\lambda))
   \prod_{i=1}\lambda_i!
   ,
   \label{EQ410}
  \end{equation}
where we used that $\lambda \vdash n-k$.
Next, substituting~\eqref{EQ409} and~\eqref{EQ410} into~\eqref{EQ408} gives us 
  \begin{align}
   \begin{split}
   U_{n+1}|_{y_i = i!}
   &=
   \sum_{k=1}^n 
   \Bigg(
     \sum_{\lambda \vdash n-k}
     c_\lambda^n 
     \prod_{i=1}^{\ell(\lambda)} 
       \lambda_i!
   \Bigg) 
   t^{k+1}
   + \sum_{k=1}^n 
   (2n-k)
   \Bigg(
     \sum_{\lambda \vdash n-k} 
     c_\lambda^n
     \prod_{i=1}^{\ell(\lambda)}
       \lambda_i!
    \Bigg)
    t^k
    .
   \end{split}
   \llabel{EQ411}
  \end{align}
By~\eqref{EQ407}, this is exactly
  \begin{equation}
   \sum_{k=1}^{n+1} S_{n+1,k} t^k
   =
   \sum_{k=1}^n S_{n,k} t^{k+1}
   + \sum_{k=1}^n (2n-k) S_{n,k} t^k
   .
   \llabel{EQ412}
  \end{equation}
By equating the terms corresponding to $t^k$, we arrive at~\eqref{EQ406},
completing the proof.
\end{proof}
\begin{remark}
The sequence defined by~\eqref{EQ406} is the triangle of coefficients of Bessel polynomials, see~\cite[A001497]{OEIS}.
\end{remark}
We now prove Proposition~\ref{P01} which expands $(\phi \partial_x)^n$ in terms of $\phi^k \partial_x^k$. 

\begin{proof}[Proof of Proposition~\ref{P01}]
Recall that $\phi^{(p)} = (-1)^{p+1}p!(1-\phi)^{p+1}$. 
Thus, by Lemma~\ref{L17}, we immediately obtain~\eqref{EQ401}, which establishes the identity given by~\eqref{EQ401} and~\eqref{EQ401a}. 

For the $H^3$ estimate~\eqref{EQ402}, we first verify that 
  \begin{equation}
   \|\phi^{n-k-\ell(\lambda)} (1-\phi)^{n-k+\ell(\lambda)}\|_{W^{3,\infty}}
   \les
   (n-k)^3
   .
   \label{EQ413}
  \end{equation}
For simplicity, we replace $n-k$ with $M$ and $\ell(\lambda)$ with $N$ and establish the bound for $M \geq N$.
The $L^\infty$ norm is bounded uniformly by $1$ since $\phi \leq 1$ for all $x > 0$.
For the third derivative, the worst case is when $M = N$.
We get,
  \begin{equation}
   \partial_x^3 (1-\phi)^{2M}
   \les
   M^3(1-\phi)^{2M+3}
   \leq
   M^3
   .
   \llabel{EQ413a}
  \end{equation}
The first and second derivatives give subordinate bounds. 

Now, let $f$ be a smooth function.
Using~\eqref{EQ413} and Lemma~\ref{L18} on~\eqref{EQ401a}, we get
  \begin{equation}
   \|d_k^n\|_{W^{3,\infty}}
   \les
    (n-k)^3
    \Bigg(
      \sum_{\lambda \vdash n-k} 
       c_\lambda^n 
       \prod_{i=1}^{\ell(\lambda)} 
        \lambda_i!
    \Bigg)
   \les
   C^{n-k}\frac{(2n-k-1)!}{2^{n-k} (n-k)!(k-1)!}
   .
   \llabel{EQ414}
  \end{equation}
For the second inequality, note that $(n-k)^3/2^{n-k} \les (3/4)^{n-k}$.
This completes the proof.
\end{proof}

Next, we consider expanding $\phi^n \partial_x^n$ in terms of $(\phi\partial_x)^k$, that is, Proposition~\ref{P02}.
The method we take is to perform a change of variables such that $\partial_x^n = (\psi(y) \partial_y)^n$ and use Proposition~\ref{P01}.
Being that the resulting non-constant  coefficients in this expansion will depend on $\psi$, this merits a preliminary estimate. 

\begin{Lemma}
\label{L19}
We have 
  \begin{equation}
   \left\|\phi (\phi \partial_x)^n \frac{1}{\phi} \right\|_{W^{3,\infty}}
   \leq
   C^n n!
   ,
   \llabel{EQ415}
  \end{equation}
where $C$ is a positive universal constant.
\end{Lemma}

\begin{proof}
Note that $\partial_x^k \phi^{-1} = (-1)^k k! x^{-k-1}$ and  
  \begin{equation}
   \phi^k \partial_x^k \left( \frac{1}{\phi}\right)
   =
   (-1)^k k! 
   \frac{1}{(1+x)^k} 
   \frac{1}{x}
   =
   (-1)^k k! 
   \frac{1}{(1+x)^{k+1}} 
   \frac{1}{\phi}
   .
  \llabel{EQ416}
  \end{equation}  
Thus, Proposition~\ref{P01} yields
  \begin{equation}
    \phi(\phi \partial_x)^n \left( \frac{1}{\phi}\right)
    =
    \sum_{k=1}^n \sum_{\lambda \vdash n-k}
      c_\lambda^n (-1)^{n}
      \left( \prod_{i=1}^{\ell(\lambda)} \lambda_i!\right)
        \phi^{n-k-\ell(\lambda)} (1-\phi)^{n-k+\ell(\lambda)}
      k!
      \frac{1}{(1+x)^{k+1}}
   .
   \llabel{EQ417}
  \end{equation}
By~\eqref{EQ413} and Proposition~\ref{P01}, we obtain
  \begin{equation}
   \|\phi (\phi \partial_x)^n(1/\phi)\|_{W^{3,\infty}}
   \les
   \sum_{k=1}^n 
     k^3(n-k)^3
     k!
     \frac{(2n-k-1)!}{2^{n-k}(n-k)!(k-1)!}
   \leq
   \sum_{k=1}^n 
    C^{n-k} 
     \frac{(2n-k-1)!}{(n-k)!}
   \leq
   C^nn!
   ,
   \llabel{EQ418}
  \end{equation}
where $C$ denotes a general constant.
For the last inequality, note that $\frac{(2n-k-1)!}{(n-k)!} = (n-1)! {2n -k -1 \choose n-k}$ and 
  \begin{equation}
   \sum_{k=1}^n 
     C^{n-k} 
     (n-1)! 
     {2n-k-1 \choose n-k}
   =
   (n-1)!
   \sum_{k=0}^{n-1}
     C^k {n+k-1 \choose k}
   \leq
   (n-1)!C^n 2^{n-1}
   \leq
   C^nn!
   .
   \llabel{EQ418a}
  \end{equation}
This completes the proof.
\end{proof}
In the proof of Proposition~\ref{P02}, 
we perform the following change of variables.
Define $\zeta \colon \mathbb{R}_+ \to \mathbb{R}$ by $\zeta(x) = x + \log x$. 
Then  $\zeta$ is invertible and $\zeta^{-1}(y) = W(e^y)$,
where $W$ is the Lambert $W$-function. 
Denoting $y = \zeta(x)$ and $\psi(y) = 1/\phi(W(e^y))$,
we have $x = W(e^y)$ and
  \begin{equation}
   \partial_x^n f(x)
   =
   (\psi(y)\partial_y)^n f(W(e^y)) 
   \llabel{EQ419}
  \end{equation}
for a smooth function $f \colon \mathbb{R}_+ \to \mathbb{R}$. 
In the following we write $\phi = \phi(x)$ for brevity.

\begin{proof}[Proof of Proposition~\ref{P02}]
By Proposition~\ref{P01} and our change of variables above,
  \begin{align}
   \begin{split}
    \phi^n \partial_x^n f(x)
    &=
    \phi^n(W(e^y))
    \sum_{k=1}^n \sum_{\lambda \vdash n-k}
      c_\lambda^n
      \psi(y)^{n-\ell(\lambda)}
      \Bigg(
        \prod_{i=1}^{\ell(\lambda)}
        \partial_y^{\lambda_i} \psi(y)
      \Bigg)
      \partial_y^k f(W(e^y))
    \\&=
    \phi^n
    \sum_{k=1}^n \sum_{\lambda \vdash n-k}
      c_\lambda^n 
      \phi^{\ell(\lambda) - n}
      \Bigg(
      \prod_{i=1}^{\ell(\lambda)}
        (\phi\partial_x)^{\lambda_i} \left( \frac{1}{\phi}\right)
      \Bigg)
      (\phi \partial_x)^k f(x)
    \\&=
    \sum_{k=1}^n \sum_{\lambda \vdash n-k} 
      c_\lambda^n 
      \Bigg(
      \prod_{i=1}^{\ell(\lambda)}
        \phi
        (\phi \partial_x)^{\lambda_i}\left( \frac{1}{\phi}\right)
      \Bigg)
      (\phi \partial_x)^k f(x)
   .
   \end{split}
  \llabel{EQ420}
  \end{align}
Then, using Lemma~\ref{L19}, 
we obtain
  \begin{equation}
    \|a_{k^n}\|_{W^{3,\infty}}
    \les
      C^{n-k}
      \sum_{\lambda \vdash n-k}
      c_\lambda^n
      \Bigg(
        \prod_{i=1}^{\ell(\lambda)}
         \lambda_i! 
      \Bigg)
    =
    C^{n-k} \frac{(2n-k-1)!}{(n-k)!(k-1)!}
    ,
   \llabel{EQ421}
  \end{equation}
completing the proof. 
\end{proof}

\section*{Acknowledgments}
BI and IK were supported in part by the NSF grant DMS-2205493.
The authors have not used AI while preparing this paper.

\end{document}